%% file: __arxiv.tex
\documentclass[preprint,aos]{imsart}

\RequirePackage{amsthm,amsmath,amsfonts,amssymb}
\RequirePackage[numbers]{natbib}
\RequirePackage[colorlinks,citecolor=blue,urlcolor=blue]{hyperref}
\RequirePackage{graphicx} 

\startlocaldefs
\theoremstyle{plain}
\newtheorem{proposition}{Proposition}
\newtheorem{corollary}{Corollary}
\newtheorem{theorem}{Theorem}
\newtheorem{lemma}{Lemma}
\theoremstyle{remark}
\newtheorem{definition}{Definition}
\newtheorem{assumption}{Assumption}
\newtheorem{remark}{Remark}
\newtheorem{example}{Example}

\usepackage[utf8]{inputenc}
\usepackage[T1]{fontenc}
\usepackage[english]{babel}

\usepackage{mathtools}
\usepackage{stmaryrd}
\usepackage{subcaption}
\usepackage{enumitem}
\usepackage{tikz}
\usetikzlibrary{matrix}
\usetikzlibrary{arrows}

\newcommand{\rom}[1]{\lowercase\expandafter{\romannumeral #1\relax}}
\DeclareMathOperator*{\argmin}{arg\,min}
\DeclareMathOperator*{\argmax}{arg\,max}

\newcommand{\card}{\textrm{Card}}

\newcommand{\N}{\mathbf{\mathbb{N}}}
\newcommand{\PP}{\mathbf{\mathbb{P}}}
\newcommand{\EE}{\mathbf{\mathbb{E}}}

\usepackage{algorithm}
\usepackage{algpseudocode}
\algblock{Input}{EndInput}
\algnotext{EndInput}
\algblock{Output}{EndOutput}
\algnotext{EndOutput}
\algblock{Estimation}{EndEstimation}
\algnotext{EndEstimation}
\algblock{Set}{EndSet}
\algnotext{EndSet}

\newcommand{\1}{\mathbf{1}}
\newcommand{\intd}{\mathrm{d}}
\newcommand{\ClusterRisk}{\mathcal{R}_n^{\mathrm{clust}}}
\newcommand{\ArianneRisk}{\mathcal{R}_n^{\mathrm{MRSS}}}

\newcommand{\ClassifRisk}{\mathcal{R}_n^{\mathrm{class}}}

\usepackage[sectionbib]{bibunits}
\defaultbibliography{Bibliography}
\usepackage{xr}
\endlocaldefs

\begin{document}

\begin{bibunit}[imsart-number]

\begin{frontmatter}
\title{Clustering risk in Non-parametric Hidden Markov and I.I.D. Models}
\runtitle{Model-based Clustering using Non-parametric Hidden Markov Models}

\begin{aug}
\author[A]{\fnms{{\'E}lisabeth}~\snm{Gassiat}\ead[label=e1]{elisabeth.gassiat@universite-paris-saclay.fr}},
\author[A]{\fnms{Ibrahim}~\snm{Kaddouri}\ead[label=e2]{ibrahim.kaddouri@universite-paris-saclay.fr}}
\and
\author[B]{\fnms{Zacharie}~\snm{Naulet}\ead[label=e3]{zacharie.naulet@inrae.fr}}
\address[A]{Université Paris-Saclay, CNRS, Laboratoire de mathématiques d’Orsay, 91405, Orsay, France\printead[presep={,\ }]{e1,e2}}
\address[B]{Université Paris-Saclay, INRAE, MaIAGE, 78350, Jouy-en-Josas, France\printead[presep={,\ }]{e3}}

\end{aug}

\begin{abstract}
\input{main-abstract}
\end{abstract}



\end{frontmatter}


\input{main-content}

\begin{acks}[Acknowledgments]
The authors would like to thank Ziv Scully for providing a proof for \cite[Lemma~S1.1]{GKN2025SM} and the associate editor for pointing out the question of the relationship between the Bayes classifier and the Bayes clusterer.
\end{acks}

\begin{funding}
{\'E}lisabeth Gassiat is supported by \textit{Institut Universitaire de France}. All the authors are supported by the \textit{Agence Nationale de la Recherche} under projects ANR-21-CE23-0035-02 and ANR-23-CE40-0018-02. 
\end{funding}

\begin{supplement}
\stitle{Clustering risk in Non-parametric Hidden Markov and I.I.D. Models: supplementary material}
\sdescription{This supplementary material contains all the missing proofs of the article.}
\end{supplement}

  \putbib%
\end{bibunit}

\newpage

\setcounter{page}{1}
\setcounter{section}{0}
\setcounter{table}{0}
\setcounter{figure}{0}
\setcounter{theorem}{0}
\setcounter{proposition}{0}
\setcounter{lemma}{0}
\setcounter{equation}{0}

\renewcommand{\thepage}{S\arabic{page}}
\renewcommand{\thesection}{S\arabic{section}}
\renewcommand{\thetable}{S\arabic{table}}
\renewcommand{\thefigure}{S\arabic{figure}}
\renewcommand{\thetheorem}{S\arabic{section}.\arabic{theorem}}
\renewcommand{\theproposition}{S\arabic{section}.\arabic{proposition}}
\renewcommand{\thelemma}{S\arabic{section}.\arabic{lemma}}
\renewcommand{\theequation}{S\arabic{section}.\arabic{equation}}

\begin{bibunit}[imsart-number]

\begin{frontmatter}
\title{Clustering risk in Non-parametric Hidden Markov and I.I.D. Models: supplementary material}
\runtitle{Model-based Clustering using Non-parametric Hidden Markov Models}

\begin{aug}
\author[A]{\fnms{{\'E}lisabeth}~\snm{Gassiat}\ead[label=e1]{elisabeth.gassiat@universite-paris-saclay.fr}},
\author[A]{\fnms{Ibrahim}~\snm{Kaddouri}\ead[label=e2]{ibrahim.kaddouri@universite-paris-saclay.fr}}
\and
\author[B]{\fnms{Zacharie}~\snm{Naulet}\ead[label=e3]{zacharie.naulet@inrae.fr}}
\address[A]{Université Paris-Saclay, CNRS, Laboratoire de mathématiques d’Orsay, 91405, Orsay, France\printead[presep={,\ }]{e1,e2}}
\address[B]{Université Paris-Saclay, INRAE, MaIAGE, 78350, Jouy-en-Josas, France\printead[presep={,\ }]{e3}}

\end{aug}

\begin{abstract}
\input{supplement-abstract}
\end{abstract}



\end{frontmatter}

   \input{supplement-content}

\begin{acks}[Acknowledgments]
 The authors would like to thank Ziv Scully for providing a proof for Lemma \ref{lem:max_dev:Scully} and the associate editor for pointing out the question of the relationship between the Bayes classifier and the Bayes clusterer.
\end{acks}
\begin{funding}
{\'E}lisabeth Gassiat is supported by \textit{Institut Universitaire de France}. All the authors are supported by the \textit{Agence Nationale de la Recherche} under projects ANR-21-CE23-0035-02 and ANR-23-CE40-0018-02.
\end{funding}
   
  \putbib%
\end{bibunit}

\end{document}

%% file: main-abstract.tex
We conduct an in-depth analysis of the Bayes risk of clustering in the context of Hidden Markov and i.i.d. models. In both settings, we identify the situations where this risk is comparable to the Bayes risk of classification and those where its minimizer, the Bayes clusterer, can be derived from the Bayes classifier. While we demonstrate that clustering based on the Bayes classifier does not always match the optimal Bayes clusterer, we show that this difference is primarily theoretical and that the Bayes classifier remains nearly optimal for clustering. A key quantity emerges, capturing the fundamental difficulty of both classification and clustering tasks. Furthermore, by leveraging the identifiability of HMMs, we establish bounds on the clustering excess risk of a plug-in Bayes classifier in the general nonparametric setting, offering theoretical justification for its widespread use in practice. Simulations further illustrate our findings.

%% file: main-content.tex
\section{Introduction}
Clustering is the problem of organizing a collection of objects into groups, ensuring that elements within each group exhibit greater \textit{similarity} to each other than to those in other groups. Clustering plays a crucial role in various domains such as machine learning, pattern recognition and image processing, helping to reveal the hidden structure of data without requiring prior labels of the observations. In contrast, classification, while closely related to clustering, seeks not only to group similar objects together but also to assign class labels to them.

Mixture models are a common framework in which the two problems can be formally defined. In these models, observations $\mathbf{Y} = (Y_1,Y_2,\dots)$ are independent conditional on unobserved random variables $\mathbf{X} = (X_1,X_2,\dots)$ taking values in $\mathbb{X} = \{1,\dots,J\}$ that represent the labels of the classes in which observations originated, with $J$ being the total number of classes. In this work, we consider the settings where the latent variables $\mathbf{X}$  are i.i.d. or form a Markov Chain. Since the i.i.d. setting is a strict subcase of the Markovian case, without loss of generality we consider the general model
\begin{equation}
\label{eq:general-model}
\begin{split}
Y_i \mid \mathbf{X}  &\overset{\mathrm{ind}}{\sim} F_{X_i}\qquad i=1,2,\dots\\
\mathbf{X} &\sim \mathrm{Markov(\nu,Q)}
\end{split}
\end{equation}
with parameter $\theta = (\nu,Q,(F_x)_{x\in \mathbb{X}})$ where $\nu$ is the initial distribution of the chain (in the next identified with probability vectors on $\mathbb{X}$), $Q$ the transition matrix of the hidden chain and $(F_x)_{x\in \mathbb{X}}$ are the emission distributions (\textit{aka}. populations). In particular, the case of i.i.d. latent variables is obtained by restricting the parameters of the previous model to transition matrices $Q$ having identical lines equal to $\nu$. The reason why we emphasize the i.i.d. subcase in the paper is not only because of its wide use, but also because some interesting phenomena emerge only in the absence of dependencies. More details on the model are given in Section~\ref{sec:model}.

In this context, clustering is the task of recovering the partition $\{A_1,\dots,A_m\} $ of $\{1,\dots,n\}$ induced by the labels (\textit{aka}. groups), namely $i,j \in A_k \iff X_i = X_j$.  Measuring the loss incured by a clusterer $g$ requires defining a notion of similiarity between two partitions, which can be achieved in many ways. Here we make the choice of using the  misclassification error metric \cite{meila2001experimental, meila2005comparing} which is based on measuring the best overlap achievable by matching the elements of the two partitions. A formal definition of the loss is stated later in Section~\ref{sec:-1}. Pursuing the goal of establishing solid decision-theoretic foundations for clustering, we aim at characterizing the \textit{Bayes clusterer}, namely the oracle clusterer $g_{\theta}$ that minimizes the risk of clustering $g\mapsto \ClusterRisk(\theta,g)$ when the parameter $\theta$ is known. The Bayes clusterer remains relatively unexplored, computationally demanding, and generally does not have an explicit formula. A widely held belief among statisticians is that using the \textit{Bayes classifier} for clustering is a reasonnable approach. The Bayes classifier solves the seemingly related problem of minimizing the risk of classification $h\mapsto \ClassifRisk(\theta,h)$, defined using the loss function that counts the number of missclassified observations. Note that classification aims at completely identifying the hidden variables $X_1,\dots,X_n$ based on the observed variables $Y_1,\dots,Y_n$, which is infeasible without prior knowledge of some labels (supervised learning). In contrast with the Bayes clusterer, though, the Bayes classifier has a well-defined closed-form solution  and has been extensively studied \cite{DGL96}. Since a clusterer can be obtained from a classifier by retaining only the partition induced by the groups, it is common-sense to use the Bayes classifier as a clusterer. This practice, however, while common amongst HMM practictionners, does not have solid decision-theoretic grounds. To the best of our knowledge, the relation between Bayes clusterer and Bayes classifier has never been investigated in the past. We fill this gap in the literature, by establishing somewhat curious results. In particular, we address the following questions:
\begin{enumerate}
    \item\label{q:i} Is there any clear link between the Bayes classifier and the Bayes clusterer? If so, under what condition?
    \item\label{q:ii} When can the Bayes risk of clustering be comparable to that of classification? What is the proper measure of separation that quantifies the difficulty of both classification and clustering problems in a general nonparametric context?
    \item\label{q:iii} Given that practitioners of Hidden Markov Models commonly use the plug-in Bayes classifier for clustering, is there any theoretical justification for this approach?
\end{enumerate}
We answer these questions in the context of the model \eqref{eq:general-model} with a focus on both the i.i.d. and dependent subcases. We voluntary establish a dichotomy between those two cases, to emphasize certain phenomena that occur only in the dependent case.

To answer Question~\ref{q:i}, we show that surprisingly, the clustering using the Bayes classifier equals that of the Bayes clusterer if and only if observations are i.i.d. from a mixture of two components. In the remaining situations (more than two components or dependent HMMs), there is always a set of parameters for which the clusterer obtained from the Bayes classifier and the Bayes clusterer differ with non-zero probability. See Theorems \ref{thm:bayes-risk:min:iid:J=2}, \ref{thm:bayes-risk:min:iid:J>2}, \ref{thm:bayes-risk:min:hmm:J=2} and \ref{thm:bayes-risk:min:hmm:J>2}. Below is an informal statement of these theorems:
\begin{theorem}[Informal]
  The Bayes classifier and clusterer coincides in the i.i.d. setting with $J=2$. If $J\geq 3$ or the labels are dependent, then there exist distributions for which Bayes classifier and Bayes clusterer differ.
\end{theorem}

The Question~\ref{q:ii} is of interest because, thanks to the closed formula of the Bayes risk of classification, it can be easily analyzed and when the two risks are comparable, any bound on the risk of classification can be extrapolated to the risk of clustering. We show that surprisingly, the Bayes risk of clustering is comparable to the Bayes risk of classification in only two situations: when there are two clusters or when the risk of classification does not decrease exponentially fast in $n$. These results are proved for i.i.d. and HMM observations. See Corollary \ref{cor:equiv:class:clust} and Theorem \ref{thm:bayes-risk:lb:iid:J>2} for the i.i.d. model, Theorem  \ref{thm:bayes-risk:lb:hmm:J=2} and Theorem \ref{thm:bayes-risk:lb:hmm:J>2} for the HMM model. Thanks to this relationship between the risks studied, we can perform a precise analysis of the Bayes risk of clustering based on the simpler Bayes risk of classification. Understanding the dependence of the Bayes risks with respect to the model parameters (mainly the population densities) is important because this will clearly identify the appropriate notion of separation measuring the difficulty of each problem. From a practical point of view, guaranteeing that the Bayes risk is of a certain magnitude will therefore translate into a simple condition on the separation between densities, which is easily interpretable. In Theorem \ref{thm:bayes-risk:key_quantity} and Corollary \ref{cor:key_quantity}, we identify the key quantity driving the difficulty of clustering and classification which turns out to be
$$
\Lambda:=\int_{\mathbb{Y}}\min_{x_{0}\in\mathbb{X}}\left(\sum_{x\neq x_0}f_{x}(y)\right)d\mathcal{L}(y)
$$
in both HMMs and i.i.d. models. Here, $\left(f_{x}\right)_{x\in\mathbb{X}}$ are the population densities of the distributions $(F_x)_{x\in\mathbb{X}}$ with respect to a dominating measure $\mathcal{L}$ over the observation space $\mathbb{Y}$, that is $dF_{x} = f_{x}d\mathcal{L}$. Notice that when $J=2$, $\Lambda = 1-\|F_{1}-F_{2}\|_{\mathrm{TV}}$, where $\|F_{1}-F_{2}\|_{\mathrm{TV}}$ is the total variation distance between the two distributions, unsurprisingly showing that the difficulty of the clustering tasks is governed by the difficulty of the hypothesis testing between $F_0$ and $F_1$. An informal answer to Question \ref{q:ii} can be summarized in the following theorem:
\begin{theorem}[Informal]
    In both the i.i.d. and the HMM settings, when $J=2$ or when $J > 2$ and $\Lambda \gtrsim e^{-c n}$ for a positive constant $c$:
        \begin{equation*}
            \inf_{g}\ClusterRisk(\theta,g) \approx \inf_{h}\ClassifRisk(\theta,h) \asymp \Lambda
        \end{equation*}
    where $\inf_{g \in \mathcal{G}_n}\ClusterRisk(\theta,g)$ and $\inf_{h \in \mathcal{H}_n}\ClassifRisk(\theta,h)$ are respectively the Bayes risks of clustering and classification.
\end{theorem}

Finally, we turn our attention to Question~\ref{q:iii}. As it will be shown in Theorem \ref{thm:bayes-risk:min:hmm:J=2} and Theorem \ref{thm:bayes-risk:min:hmm:J>2}, the Bayes clusterer can not always be derived from the Bayes classifier. Nevertheless, the plug-in Bayes classifier remains the most commonly used method for clustering under Hidden Markov Models (HMMs); see \cite{KG17, SSS03, GGM14}. Despite its popularity, this approach lacks theoretical justification. Such a substitution is not justified unless the Bayes clusterer can be derived from the Bayes classifier or if the clustering risk of the Bayes classifier is comparable to the Bayes risk of clustering.  This highlights the need for theoretical guarantees on the performance of this procedure. We show in Corollaries \ref{cor:Bayes_classif:iid} and \ref{cor:Bayes_classif:hmm} that the clustering risk incurred by the Bayes classifier closely approximates the Bayes risk of clustering, showing thus the near optimality of the Bayes classifier for clustering. Furthermore, in Theorem~\ref{thm:excess_risk:clust:rate}, we establish that the clustering excess risk of a plug-in Bayes classifier decreases at the nonparametric estimation rate when using an appropriate estimator of the model parameters. This result provides a theoretical foundation for the practical use of the plug-in Bayes classifier. The main takeaway for practitioners working with HMMs is that, although a conceptual distinction exists between the Bayes classifier and the Bayes clusterer, in practice, the Bayes classifier is a reliable choice for clustering, with its excess risk provably controlled. We support this claim with numerical experiments in the nonparametric setting, demonstrating that HMM-based clustering can successfully recover latent structures in scenarios where standard methods fail. These results can be summarized in the following informal theorem.
\begin{theorem}[Informal]
   Let $g_\theta$ be the clustering procedure built using the Bayes classifier. The Bayes classifier is nearly optimal:
   \begin{equation*}
      \ClusterRisk(\theta,g_{\theta}) \approx \inf_{g \in \mathcal{G}_n}\ClusterRisk(\theta,g).
   \end{equation*}
   Furthermore, under the hidden Markov Model, when the emission densities are $s$-Hölder regular, and under some additional regularity conditions, there exists an estimator $\hat{\theta}$ of $\theta$ such that:
   \begin{equation*}
     \ClusterRisk(\theta,g_{\hat{\theta}}) - \inf_{g}\ClusterRisk(\theta,g) = \mathcal{O}\left(\left(\frac{\log(n)}{n}\right)^{\frac{s}{2s+1}}\right).
   \end{equation*}
\end{theorem}
Note that unlike specific cases such as parametric or translation models where many clustering algorithms can be proposed, the general nonparametric setting lacks a clear alternative to the plug-in procedure. This is why the excess risk reflects the nonparametric estimation rate.

The assumptions of the model and the definitions of the various risks we shall use in this work are described in Section \ref{sec:setting}. We state our main results in Section \ref{sec:main}, the proofs of which are detailed in the supplementary material \cite{GKN2025SM}. Section \ref{sec:simu} is devoted to simulation experiments, and Section \ref{sec:perspectives} to possible further work.

\section{Setting and definitions}
\label{sec:setting}

\subsection{Notations}
\label{sec:notations}

We note for $i\leq j$ $X_{i:j}:=(X_{i}, \ldots, X_{j})$ and  $Y_{i:j}:=(Y_{i}, \ldots, Y_{j})$. We denote $[n] \coloneqq \{1,\dots,n\}$ and $\mathcal{P}[n]$ is the set of partitions of $[n]$. The set of permutations of $[J]$ will be denoted by $\mathcal{S}_{J}$, and for any $\tau\in\mathcal{S}_{J}$, we consider  $\theta^{\tau}$, the model parameter when labels are permuted with $\tau$, that is 
$\theta^{\tau} = \left(\nu^{\tau}, Q^{\tau},
(f_{\tau(x)})_{x\in\mathbb{X}}\right)$ where $\nu^{\tau}=(\nu_{\tau (x)})_{x\in\mathbb{X}}$ and $Q^{\tau}=(Q_{\tau(x),\tau(x')})_{x,x'\in\mathbb{X}}$. For $h$ a measurable function, $\|h\|_{\infty}$ denotes the essential supremum of $h$, possibly infinite. Frobenius norm is denoted by $\|.\|_{\mathrm{F}}$ and $\|.\|$ stands for the operator norm. Given two sequences of positive numbers $(a_n)_n$ and $(b_n)_n$, $a_n \gtrsim b_n$ means that there exists an absolute constant $c > 0$ and $n_{0}\in\mathbb{N}$ such that $\left(\forall n\geq n_{0}\right)\quad a_{n} \geq c b_{n}$. In the HMM modeling, for any parameter $\theta$ and any $\tau\in{\cal S}_{J}$, the distribution of the observations under $\PP_{\theta}$ is the same as under $\PP_{\theta^{\tau}}$. In other words,  this distribution is invariant up to permutation of the labels given by the hidden states. This is known as the \textit{label-switching} issue.

\subsection{The model}
\label{sec:model}

We consider a Hidden Markov Model with $J$ hidden states taking value in a set of labels $\mathbb{X} = \left\{1, .., J\right\}$, and observations in a Polish space endowed with its Borel $\sigma$-field $(\mathbb{Y}, \mathcal{Y})$. We denote $\textbf{X} = (X_1, X_2, \dots)$ and $\textbf{Y} = (Y_1, Y_2,\dots)$  respectively the sequence of hidden states forming the Markov chain and the observations. 
We assume that the emission distributions have densities $f_{x}$, $x=1,\ldots J$, with respect to a dominating measure $\mathcal{L}$ on $(\mathbb{Y}, \mathcal{Y})$. The HMM assumption boils down to:
\begin{align*}
  \textbf{X}
  &\sim \mathrm{Markov}(\nu, Q)\\
  \PP_{\theta}(Y_{i}\in \cdot \mid \mathbf{X})
  &= f_{X_{i}}(\cdot)d\mathcal{L}
\end{align*}
where $\nu$ is the initial distribution of the chain and $Q$  the transition matrix of the hidden chain. Throughout the paper, it is assumed that only the beginning $Y_{1:n}$ of the sequence $\mathbf{Y}$ is observed, and nothing else. We set $\theta = \left(\nu, Q, \left(f_{x}\right)_{x \in \mathbb{X}} \right)$ the parameter of the model and $\Theta$ denotes the space of all valid parameters. The following assumption is made:

\begin{itemize}
    \item \textit{Independent case}. In this case we assume that all the lines of $Q$ are identical and equal to the vector $\nu$ of weights forming its stationary distribution. This corresponds to the usual mixture model with independent latent variables. The set of these parameters will be denoted 
    $\Theta^{\mathrm{ind}}$.
    \item \textit{Dependent case}. In this case we assume the lines of the transition matrix $Q$ not all equal, so that the $(X_i)_{i\geq 1}$ are not independent. The set of these parameters will be denoted $\Theta^{\mathrm{dep}}$. 
\end{itemize}
 Throughout this work, we shall consider the transition matrix being fixed, and we will be interested in how the separation of populations, understood as some quantity depending on the populations densities, drives the difficulty of the clustering task. We show that, under the HMM modelling, it is possible to cluster general populations without assuming they belong to some prescribed parametric family. While our primary interest is in the dependent case in which the emission densities can be identified without any further constraint, we obtain, however, results that have interest also in the widely used independent case, in particular the analysis of the Bayes risk and its minimizers. (see Section~\ref{sec:main}).

One main difference between the HMM and the i.i.d. situation in the analysis  of the Bayes risk of classification is that in the HMM modeling, the probability of a label $X_i$ given the observations $Y_{1},\ldots,Y_{n}$ depends on all the observations, whereas in the i.i.d. case it depends only on the associated observation $Y_i$. The vector of probabilities of the $X_i$'s given $Y_{1},\ldots,Y_{n}$ are called the smoothing distributions, they depend on $i$, $n$ and the observations. The vectors of probabilities of the $X_i$'s given the observations up to time $i$, $Y_{1},\ldots,Y_{i}$, are called the filtering distributions, they depend on $i$, $n$ and the involved observations. Recursive formulas verified by the filtering and smoothing distributions make the computations tractable under the HMM modelling. See Section 3 of \cite{CMT05} for details.

\subsection{The problem of clustering}\label{sec:-1}

For any $n\geq 1$, the finite sequence $X_{1:n} = (X_1,\dots,X_n)$ induces a random partition $\Pi_n = \{C_1,C_2,\dots\}$ of $[n]$ whose blocks -- the so-called \textit{clusters}  -- are the equivalence classes for the random equivalence relation $i\sim j \iff X_i = X_j$. The goal of clustering is to uncover this partition $\Pi_n$ on the sole basis of the observation $Y_{1:n} = (Y_1,\dots,Y_n)$. We define a \textit{clusterer}:
\begin{definition}[Clusterer]
  \label{def:clusterer}
  A $n$-clusterer is a measurable map $g : \mathbb{Y}^n \to \mathcal{P}[n]$. We denote by $\mathcal{G}_n$ the set of all $n$-clusterers.
\end{definition}

We measure the loss incurred by guessing $g(Y_{1:n})$ in place of $\Pi_n$ via the misclassification error distance \cite{meila2001experimental,meila2005comparing}. For two partitions $A$ and $B$ of $[n]$, this loss is defined by:
\begin{equation}\label{def:min_overlap}
    \ell\left(A, B\right) = 
    1 - \frac{1}{n}\sup_{\substack{M \subseteq \mathcal{E}(A, B)\\M\, \textrm{is\, a matching} } 
    } \sum_{ \{C,C' \} \in M }\card{(C \cap C' )}
\end{equation}
where the supremum is taken over the set of matchings. To define a matching, we build the complete bipartite graph $(A, B,\mathcal{E}(A, B))$ on vertices $A$ and $B$ with edge set $\mathcal{E}(A, B) \coloneqq \{ \{C,C' \}\;:\; C\in A,\, C' \in B \}$. Then we recall that a matching $M$ is a set $M\subseteq \mathcal{E}(A, B)$ of edges without common vertices (\textit{i.e}. each block of $A$ and $B$ appears in at most one edge of the matching). An example of a matching is depicted in Figure~\ref{fig:matching}. Then the risk of a clusterer $g$ can be defined as the expected loss of the partition $g(Y_{1:n})$ with respect to the true partition $\Pi_n$ which is summarized in the risk function $\ClusterRisk : \Theta \times \mathcal{G}_n \to [0,1]$
\begin{equation}
  \label{eq:clustering-risk:part}
  \ClusterRisk(\theta,g)%
  \coloneqq \EE_{\theta}\Bigg[\ell\left(g(Y_{1:n}), \Pi_n\right)\Bigg].
\end{equation}

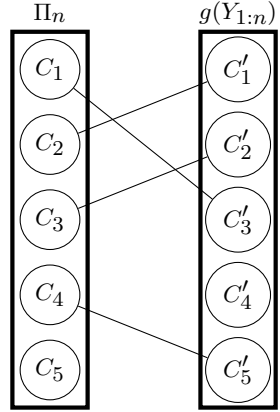
\begin{figure}[!htb]
    \centering
    \begin{tikzpicture}[scale=0.5]
        \node[shape=circle,draw=black] (A) at (0,0) {$C_1$};
        \node[shape=circle,draw=black] (B) at (0,-2) {$C_2$};
        \node[shape=circle,draw=black] (C) at (0,-4) {$C_3$};
        \node[shape=circle,draw=black] (D) at (0,-6) {$C_4$};
        \node[shape=circle,draw=black] (E) at (0,-8) {$C_5$};
        \node[shape=circle,draw=black] (F) at (5,0) {$C_1'$};
        \node[shape=circle,draw=black] (G) at (5,-2) {$C_2'$};
        \node[shape=circle,draw=black] (H) at (5,-4) {$C_3'$};
        \node[shape=circle,draw=black] (I) at (5,-6) {$C_4'$};
        \node[shape=circle,draw=black] (I) at (5,-8) {$C_5'$};
        \node at (0,1.5) {$\Pi_n$};
        \node at (5,1.5) {$g(Y_{1:n})$};
    
        \path [-] (A) edge node[right]{}(H);
        \path [-] (B) edge node[right]{}(F);
        \path [-] (C) edge node[right]{}(G);
        \path [-] (D) edge node[right]{}(I);
        \draw[ultra thick] (-1,1) rectangle (1,-9);
        \draw[ultra thick] (4,1) rectangle (6,-9);
    \end{tikzpicture}
    \caption{Example of a matching. Nodes on the left represent the clusters induced by the partition of $\Pi_n$; those on the right are the clusters of $g(Y_{1:n})$. Edges form a matching between the two partitions.}
    \label{fig:matching}
\end{figure}

A closely related notion is that of a \textit{classifier}:
\begin{definition}[Classifier]
  \label{def:classifier}
  A $n$-classifier is a measurable map $h : \mathbb{Y}^n \to \mathbb{X}^n$. We denote by $\mathcal{H}_n$ the set of all $n$-classifiers.
\end{definition}

One may find our Definition~\ref{def:classifier} different from standard textbook definitions \cite{DGL96}: we explain the reason of this choice later in Remark~\ref{rmk:classifier}. A classifier differs from a clusterer in that it not only seeks for the hidden partition, but also for the labels of the observations. Hence, the usage of a classifier only makes sense in a supervised framework where access to some labeled data is allowed in some way. In an unsupervised framework such as our model, $\mathbf{Y}$ does not contain any information about the labels and recovering them better than a lucky guess is impossible. It is however true that, to any $n$-classifier $h \in \mathcal{H}_n$ corresponds a unique $n$-clusterer $g\in\mathcal{G}_n$ which can be built via the map $\pi_n : \mathbb{X}^n \to \mathcal{P}[n]$ such that
\begin{equation*}
  g(Y_{1:n}) = \pi_n \circ h (Y_{1:n})%
  = \{ \{i\;:\; h_i(Y_{1:n}) = x \}\;:\; x \in \mathbb{X} \} \backslash \{ \varnothing \}
\end{equation*}
and any clusterer can be represented that way by choosing a specific labelling of the clusters. For this reason, the notions of clusterer and classifier are very much often amalgamated in the literature. We argue that it would be better to define them separately in order to avoid confusions between the risk of clustering $\ClusterRisk(\theta, \pi_n \circ h)$ and the risk of classification  $\ClassifRisk : \Theta \times \mathcal{H}_n \to [0,1]$ (relative to the loss counting number of misclassified observations)
\begin{equation}
  \label{eq:risk-classif}
  \ClassifRisk(\theta,h)%
  \coloneqq \EE_{\theta}\Bigg[\frac{1}{n}\sum_{i=1}^n\1_{h_i(Y_{1:n}) \ne X_i} \Bigg].
\end{equation}
Here again we insist that the definition of $\ClassifRisk$ in Equation~\eqref{eq:risk-classif} -- although mathematically correct -- has no statistical interest in an unsupervised model where the classifier $h$ is not allowed to see some of the labels. An easy exercise (see \cite[Lemma~S13.13]{GKN2025SM}) shows that the risk of clustering of $\pi_n \circ h$ can be rewritten as
\begin{equation}
  \label{eq:risk-clustering:classif}
  \ClusterRisk(\theta,\pi_n \circ h)%
  = \EE_{\theta}\Bigg[\min_{\tau \in \mathcal{S}_{J}}\frac{1}{n}\sum_{i=1}^n\1_{h_i(Y_{1:n}) \ne \tau(X_i)} \Bigg]
\end{equation}
and differs from \eqref{eq:risk-classif}. Note that it does not depend on the classifier $h$ chosen to represent the clusterer thanks to the infimum over the permutations of the labels. Note also that the infimum inside the expectation reflects the ease of the clustering problem compared to the classification problem, because unlike classification, clustering does not seek to identify the true labels themselves. 

It is customary to compare the performance of a given classifier $h$ to the best performance attainable by an oracle classifier, namely the \textit{Bayes risk of classification}:
\begin{equation*}
  \inf_{h\in \mathcal{H}_n} \ClassifRisk(\theta,h).
\end{equation*}
In particular, it is well-known that the previous optimization problem is solved by the (albeit non necessarily unique) so-called Bayes classifier $h^{\star}_{\theta} = (h_{\theta, 1}^{\star},\dots,h_{\theta, n}^{\star})$ such that
\begin{equation*}
  \PP_{\theta}(X_i = h_{\theta, i}^{\star}(Y_{1:n}) \mid Y_{1:n}) = \max_{x\in \mathbb{X}}\PP_{\theta}(X_i = x \mid Y_{1:n}),\qquad i=1,\dots,n.
\end{equation*}
In an unsupervised learning context, it makes sense to compare the risk of a clusterer to the \textit{Bayes risk of clustering}:
\begin{equation*}
  \inf_{g \in \mathcal{G}_n}\ClusterRisk(\theta,g).
\end{equation*}
As for the classification risk, the solution to the previous optimization problem exists and is obtained by the Bayes clusterer $g^{\star}_{\theta}$ such that $g^{\star}_{\theta}(Y_{1:n})$ is the partition that minimizes
\begin{equation}
  g \mapsto%
  \EE_{\theta}\Bigg[1 - \frac{1}{n}\sup_{\substack{M \subseteq \mathcal{E}(g(Y_{1:n}),\Pi_n)\\M\, \textrm{is\, a matching} } 
  } \sum_{ \{C,C' \} \in M }\card{(C \cap C' )} \bigg| Y_{1:n}\Bigg].
  \label{eq:clustering-risk:Bayes-clusterer}
\end{equation}
In contrast with classification, however, the Bayes clusterer has usually no simple expression. It is to be noted that there is no reason that $g^{\star}_{\theta} = \pi_n \circ h^{\star}_{\theta}$. An analysis will be conducted to determine when the equality $g^{\star}_{\theta} = \pi_n \circ h^{\star}_{\theta}$ holds. Although the inequality $\inf_{g \in \mathcal{G}_n}\ClusterRisk(\theta,g)  \leq \inf_{h\in \mathcal{H}_n}\ClassifRisk(\theta,h)$ is true and easily proved from \eqref{eq:risk-clustering:classif}, it is not guaranteed that $\inf_{g \in \mathcal{G}_n}\ClusterRisk(\theta,g)$ and $\inf_{h\in \mathcal{H}_n}\ClassifRisk(\theta,h)$ are equivalent (\textit{ie}. have comparable magnitude). In Sections~\ref{sec:comp-clust-class-iid} and~\ref{sec:comp-clust-class-hmm} we show that the equivalence holds both in the independent and in the dependent scenarios as long as the Bayes risk of classification is not exponentially small in $n$ or when there are only two clusters. Though we prove the two risks to be equivalent in some specific contexts, we provide a counter-example showing that this is not always true. 

\begin{remark}
  \label{rmk:classifier}
  It is common \cite{DGL96} to define a classifier as a function $h : \mathbb{Y} \to \mathbb{X}$ (as opposed to $\mathbb{Y}^n \to \mathbb{X}^n$ in our Definition~\ref{def:classifier}). This usual definition is motivated by the fact that in the independent scenario the law of $X_i \mid Y_{1:n}$ is that of $X_1 \mid Y_1$. Hence if one is willing to classify the vector $Y_{1:n}$, the Bayes classifier rewrites as $h^{\star}_{\theta} = (h_{\theta, 1}^{\star}(Y_1),\dots,h_{\theta, 1}^{\star}(Y_n))$ with $h_{\theta, 1}(y)$ maximizing $x\mapsto \PP_{\theta}(X_1= x \mid Y_1 = y)$ and the Bayes risk of classification equals $\PP_{\theta}(h_{\theta, 1}^{\star}(Y_1) \ne X_1)$. Thus classifying $Y_1$ or $Y_{1:n}$ is not very different. In the dependent case, however, the situation differs since $\textbf{X} \mid \textbf{Y}$ is a inhomogeneous Markov chain. This implies in particular that classifying $Y_1$ or $Y_{1:n}$ are different problems when $n\geq 2$, and the optimal solution for classifying $Y_{1:n}$ can not be obtained from the optimal solution of classifying $Y_1$.
\end{remark}

\begin{remark}
  \label{rmk:arianne}
  In \cite{MRR23}, the authors define the risk of clustering of a $n$-classifier $h$ in a different way:
$$
\ArianneRisk(\theta, h) \coloneqq
\EE_{\theta}\left[\inf_{\tau\in\mathcal{S}_{J}}\EE_{\theta}\left[\frac{1}{n}\sum_{i=1}^{n}\1_{\tau(X_{i})\neq h_{i}(Y_{1:n})}\bigg| Y_{1:n}\right]\right]
$$
Similarly, in \cite{LH16}, the risk of clustering is defined by:
$$
\tilde{\mathcal{R}}_n^{\mathrm{clust}}(\theta, h) \coloneqq
\inf_{\tau\in\mathcal{S}_{J}}\EE_{\theta}\left[\frac{1}{n}\sum_{i=1}^{n}\1_{\tau(X_{i})\neq
      h_{i}(Y_{1:n})}\right]
$$
Their definition is mathematically convenient as one can easily show that (see \cite[Lemma~S4.2]{GKN2025SM}):
$$
\inf_{h\in \mathcal{H}_n}\ArianneRisk(\theta, h) = \inf_{h\in \mathcal{H}_n}\tilde{\mathcal{R}}_n^{\mathrm{clust}}(\theta, h) = \inf_{h\in \mathcal{H}_n}\ClassifRisk(\theta,h)
$$
Hence, the Bayes clusterer relative to their risks is $g^{\star}_{\theta} = \pi_n \circ h^{\star}_{\theta}$ with $h^{\star}_{\theta}$ the Bayes classifier relative to the risk $\ClassifRisk(\theta, .)$. However, contrary to our definition, it seems that there is no suitable statistical interpretation of these risks of clustering. 

\end{remark}

\section{Main results}
\label{sec:main}

The Bayes risk of classification offers an advantage due to its closed formula, stemming from the explicit identification of the Bayes classifier. In contrast, the Bayes risk of clustering lacks this straightforward formulation. However, thanks to Equation \eqref{eq:risk-clustering:classif}, the risk of a clustering procedure can be seen as the risk of classification of the associated classifier up to the best random permutation (the one minimizing the sum inside the expectation in Equation \eqref{eq:risk-clustering:classif}). This is why a common idea is that risk of clustering and risk of classification are closely related. Following this intuition, authors of \cite{MR3546450, MR3845014} propose an inequality linking the minimax (over a specific class) risk of clustering to the minimax risk of classification, in the context of community detection under the stochastic block model. Their inequality is applied in \cite{Ndaoud18} in the context of the mixture of two Gaussian distributions and in \cite{LH16} for general subGaussian mixtures. Although their argument is neat, it relies heavily on two key ingredients: (i) in their model the latent partition is deterministic and viewed as a parameter, and  (ii) their point of view is minimax and thus, it is enough to consider some well-chosen parameters in order to lower bound the minimax risk. Thus their argument is not transposable to the situation where one is interested in bounding the Bayes risk of clustering at every possible value of the parameter with random labels.

\subsection{I.I.D. case}\label{sec:comp-clust-class-iid}
The following theorem shows that in the i.i.d. setting and when there are only two classes $(J=2)$, the clustering resulting from the Bayes classifier coincides with that of the Bayes clusterer, almost-everywhere. 
\begin{theorem}
    \label{thm:bayes-risk:min:iid:J=2}
    In the case of independent labels, if $J = 2$, then for all $\theta \in \Theta^{\mathrm{ind}}$ and all $n \geq 2$ 
    \begin{equation*}
        \quad g_{\theta}^{\star}(Y_{1:n}) = \pi_n\circ h^{\star}_{\theta}(Y_{1:n}) \quad\PP_{\theta}\textrm{-a.s}.
    \end{equation*}
\end{theorem}
The proof of Theorem \ref{thm:bayes-risk:min:iid:J=2} can be found in \cite[Section~S1]{GKN2025SM}. Thanks to this result, the difference between the two risks of clustering and classification can be bounded almost tightly, as shown in the next theorem. Let $\varepsilon_{n, \theta} = \frac{1}{2} - \inf_{h \in \mathcal{H}_n}\ClassifRisk(\theta,h)$.

\begin{theorem}\label{thm:gap:classif:cluster}
When $J=2$ and $\theta\in\Theta^{\mathrm{ind}}$, $\inf_{h \in \mathcal{H}_n}\ClassifRisk(\theta,h) = 0$\quad if and only if \quad ~$\inf_{g \in \mathcal{G}_n}\ClusterRisk(\theta,g) = 0$. If $\inf_{h \in \mathcal{H}_n}\ClassifRisk(\theta,h) \neq 0$~ then the difference between the Bayes risks satisfies
    \begin{equation*}
        \inf_{h \in \mathcal{H}_n}\ClassifRisk(\theta,h) - \inf_{g \in \mathcal{G}_n}\ClusterRisk(\theta,g) \leq \min\Bigg(\frac{(1 - 4\varepsilon_{n,\theta}^2)^{\frac{n}{2}} }{\frac{n}{2}\log\Big(\frac{1 + 2 \varepsilon_{n, \theta}}{1 - 2\varepsilon_{n, \theta}}\Big)},\, \sqrt{\frac{\pi}{2 n}} \Bigg).
    \end{equation*}
Furthermore, there exists a universal constant $B > 0$ such that for all $n\geq 100$ and all $\theta \in \Theta^{\mathrm{ind}}$
    \begin{equation*}
      \inf_{h \in \mathcal{H}_n}\ClassifRisk(\theta,h) - \inf_{g \in \mathcal{G}_n}\ClusterRisk(\theta,g)%
      \geq B \min\Bigg(\frac{(1 - 4\varepsilon_{n,\theta}^2)^{\frac{n}{2}\big(1 + \frac{6.8}{1\vee \sqrt{n}\varepsilon_{n,\theta}}\big)} }{\frac{n}{2}\log\Big(\frac{1 + 2 \varepsilon_{n, \theta}}{1 - 2\varepsilon_{n, \theta}}\Big)}   ,\,\frac{1}{\sqrt{n}}\Bigg).
    \end{equation*}
\end{theorem}
The proof of Theorem \ref{thm:gap:classif:cluster} is given in \cite[Section~S2]{GKN2025SM}. 
As emphasized by the lower bound, the upper bound on the difference between the two Bayes risks in the previous theorem is essentially tight. The lower bound also shows that the Bayes risk of clustering is always strictly smaller than the Bayes risk of classification, unless they are both zero (which happens when the two emission densities have disjoint support). Also, the difference between the risks decays super-polynomially in $n$ as soon as $\varepsilon_{n,\theta} \gg \sqrt{\log(n)/n}$, and polynomially if $\varepsilon_{n,\theta} = O(\sqrt{\log(n)/n})$ with worst-case rate $\asymp n^{-1/2}$ when $\varepsilon_{n,\theta} = O(n^{-1/2})$.

A direct consequence of this result is that the Bayes risk of classification is equivalent to the Bayes risk of clustering for $n$ large enough as shown in the following corollary.
\begin{corollary}\label{cor:equiv:class:clust}
    In the case of independent labels with $J =2$, for all $\theta\in\Theta^{\mathrm{ind}}$ and all $n \geq 2$
    \begin{equation*}
        \left(1 - \alpha_n\right) \inf_{h \in \mathcal{H}_n}\ClassifRisk(\theta,h) \leq \inf_{g \in \mathcal{G}_n}\ClusterRisk(\theta,g) \leq \inf_{h \in \mathcal{H}_n}\ClassifRisk(\theta,h)
    \end{equation*}
    where $\alpha_n = 2\min\Big(2(1 + \varepsilon_{n, \theta})\frac{(1 - 4\varepsilon^{2}_{n, \theta})^{\frac{n-2}{2}}}{n \log\left(\frac{1+2\varepsilon_{n, \theta}}{1 - 2\varepsilon_{n, \theta}}\right)}, \frac{1}{1-2\varepsilon_{n, \theta}}\sqrt{\frac{\pi}{2n}}\Big)$.
\end{corollary}

The following proposition shows that, contrary to Corollary \ref{cor:equiv:class:clust}, the two Bayes risks are not equivalent in general.

\begin{proposition}
  \label{pro:bayes-risk:non-equivalence}
  Whenever $J > 2$ and $n \geq 1$ :
  \begin{equation*}
    \inf_{\theta \in \Theta^{\mathrm{ind}}} \frac{\inf_{g\in \mathcal{G}_n}\ClusterRisk(\theta,g)}{\inf_{h\in \mathcal{H}_n}\ClassifRisk(\theta,h)}%
    = 0.
  \end{equation*}
\end{proposition}

The proof of Proposition~\ref{pro:bayes-risk:non-equivalence} is given in \cite[Section~S9]{GKN2025SM}. It sheds light on why the previously established equivalence when there were only two classes no longer holds when the number of classes exceeds two. The proof is based on the following intuition. The observations which are misclustered or misclassified appear only in regions of overlap between the emission densities (cf. Theorem \ref{thm:bayes-risk:key_quantity} and Corollary \ref{cor:key_quantity}). Consider now the situation where $n$ observations are derived from a mixture of $J=3$ densities $F_1$, $F_2$ and $F_3$ with weights $\nu_1 = 1-2\eta$, $\nu_2 = \eta$ and $\nu_3 = \eta$. Assume that the support of $F_1$ is disjoint from that of $F_2$ and $F_3$ and that $F_2$ and $F_3$ have overlapping supports over a small region of the space. When the weight $\eta$ is small, two situations are possible: eventually only one observation belongs to the support of $F_2$ or $F_3$, in which case the clustering error is null, or more observations appear in this region of the space which happens with very small probability because $\eta$ is small. Combining these two insights, the magnitude of the risk of clustering is shown to be negligible with respect to the risk of classification. The use of compactly supported distributions is not essential to the previous argument. For instance, the argument still holds modulo supplementary technicalities if one takes $F_1 = \mathcal{N}\left(-\alpha x, 1\right)$, $F_2 = \mathcal{N}\left(-\alpha , 1\right)$ and $F_3 = \mathcal{N}\left(\alpha, 1\right)$ with $\alpha > 0$ and $x > 0$ large enough.


This dichotomy between $J=2$ and $J > 2$ concerns also the minimizers of the risks. Contrary to Theorem \ref{thm:bayes-risk:min:iid:J=2} (when $J=2$), in the i.i.d. setting with $J>2$, one can always find some model parameters for which the result of clustering using the Bayes classifier differs from that of the Bayes clusterer with positive probability, as shown in the next theorem.

\begin{theorem}
    \label{thm:bayes-risk:min:iid:J>2}
    If $J > 2$, then for all $\theta \in \Theta^{\mathrm{ind}}$, if
        \begin{equation*}
            \PP_{\theta}\left(\bigcup_{j=1}^{J}\left\{0 < \max_{l \neq j} \nu_l f_l(Y) < \nu_j f_j (Y) \leq \sum_{l\neq j}\nu_l f_l (Y) \right\} \right) > 0
        \end{equation*}
        then,
        \begin{equation*}
            \forall n \geq 2,\quad\PP_{\theta}\left(g_{\theta}^{\star}(Y_{1:n}) \neq \pi_n\circ h_{\theta}^{\star}(Y_{1:n})\right) > 0.
        \end{equation*}
\end{theorem}

The condition above can be ensured easily for many distributions $\left(f_j\right)_{j\in\mathbb{X}}$ such as multinomials, mixtures of Gaussians, etc. For example, it is valid when $J=3$, $(\nu_1, \nu_2, \nu_3) = (0.4, 0.4, 0.2)$, and the emission densities are Gaussians with variance $\sigma^{2} = 1$ and means $(\mu_1, \mu_2, \mu_3) = (1, 2, 3)$. The proof of Theorem~\ref{thm:bayes-risk:min:iid:J>2} is given in \cite[Section~S3]{GKN2025SM}.

Even though the Bayes risks of clustering and classification are not equivalent uniformly over $\Theta^{\mathrm{ind}}$ by Proposition~\ref{pro:bayes-risk:non-equivalence}, we show in the next theorem that when $J>2$, they become equivalent when $\inf_{g \in \mathcal{G}_n}\ClusterRisk(\theta,g) \gtrsim J^2 e^{-n\beta/8}$  where $\beta = \min_{j\ne k}(\nu_j + \nu_{k})$.

\begin{theorem}\label{thm:bayes-risk:lb:iid:J>2}
    For all $\theta \in \Theta^{\mathrm{ind}}$ and all $n\geq 1$ the following bounds hold
    \begin{align*}
    \inf_{g \in \mathcal{G}_n}\ClusterRisk(\theta,g)%
    &\geq \inf_{h \in \mathcal{H}_n}\ClassifRisk(\theta,h)%
      - \sqrt{\frac{\log(J!)}{2n}},\\
    \inf_{g \in \mathcal{G}_n}\ClusterRisk(\theta,g)%
    &\geq (1-\xi_n) \inf_{h\in \mathcal{H}_n}\ClassifRisk(\theta,h) - J^2e^{-n\beta/8},
  \end{align*}
  where $\beta = \min_{j\ne k}(\nu_j + \nu_{k})$ and $\xi_n = \frac{4e}{\beta}[\sqrt{\log(J!)/(2n)}]^{1- 4/(n\beta)}$.
\end{theorem}
 The proof of Theorem~\ref{thm:bayes-risk:lb:iid:J>2} is given in \cite[Section~S5]{GKN2025SM}. Even if the second lower bound obtained does not establish an equivalence between the Bayes risks of clustering and classification, it remains useful for analyzing phase transitions. Specifically, under a mild assumption on $\beta$, this bound results in similar phase transitions for both risks.

Even if clustering using the Bayes classifier differs sometimes from that of the Bayes clusterer (as shown in Theorem \ref{thm:bayes-risk:min:iid:J>2}), Theorem \ref{thm:bayes-risk:lb:iid:J>2} provides guarantees for the risk of clustering using the Bayes classifier as shown by the following corollary.

\begin{corollary}\label{cor:Bayes_classif:iid}
    For all $\theta \in \Theta^{\mathrm{ind}}$ and all $n\geq 1$ the following bounds hold
    \begin{equation*}
    \inf_{g \in \mathcal{G}_n} \ClusterRisk(\theta, g) \leq \ClusterRisk\left(\theta, \pi_n\circ h^{\star}_{\theta}\right) \leq \frac{1}{1 - \xi_{n}} \left(\inf_{g \in \mathcal{G}_n}\ClusterRisk(\theta,g) + J^{2}e^{-n\beta/8}\right)
\end{equation*}
with $\xi_n$ and $\beta$ as defined in Theorem~\ref{thm:bayes-risk:lb:iid:J>2}. 
When there are only two classes
\begin{equation*}
    \inf_{g \in \mathcal{G}_n} \ClusterRisk(\theta, g) \leq \ClusterRisk\left(\theta, \pi_n\circ h^{\star}_{\theta}\right) \leq \frac{1}{1 - \alpha_{n}} \inf_{g \in \mathcal{G}_n}\ClusterRisk(\theta,g) 
\end{equation*}
where $\alpha_n$ is defined in Corollary \ref{cor:equiv:class:clust}.
\end{corollary}

\subsection{HMM case}\label{sec:comp-clust-class-hmm}
The behavior of the Bayes risks under the HMM modeling exhibits similarities and differences with the i.i.d. case. As we will see, the Bayes risks keep having the same behavior under the HMM setting, while their minimizers, on the other hand, can be different. First, contrary to the i.i.d. case where the result of clustering using the Bayes classifier matches that of the Bayes clusterer, it is always possible to find a set of parameters for which they differ under the HMM setting, as shown in the next theorem.

\begin{theorem}
    \label{thm:bayes-risk:min:hmm:J=2}
    In the case of dependent labels and when $J = 2$, there exists a subset $\Tilde{\Theta} \subset \Theta^{\mathrm{dep}}$ such that
        \begin{equation*}
            \left(\forall \theta\in\Tilde{\Theta}\right) \left(\forall n \geq 2\right),\quad\PP_{\theta}\left(g_{\theta}^{\star}(Y_{1:n}) \neq \pi_n\circ h_{\theta}^{\star}(Y_{1:n})\right) > 0.
        \end{equation*}
\end{theorem}

We illustrate this in the simple situation of $n=2$. Assume one observes two consecutive observations $Y_1$ and $Y_2$ of a HMM with transition matrix $$Q = \begin{pmatrix}
1-p & p \\
q & 1-q\\
\end{pmatrix}$$
and emission densities $f_1$ and $f_2$. Denoting for simplicity $\bar{p}=1-p$ and $\bar{q}= 1-q$, one can easily check that the Bayes clusterer puts the two observations in the same cluster when 
\begin{equation}
\label{eq:n2:bayesclust:cond1}
    q \bar{p} f_1 (Y_1) f_1(Y_2) + p \bar{q} f_2 (Y_1) f_2(Y_2) \geq p q \left(f_2 (Y_1) f_1(Y_2) + f_1 (Y_1) f_2(Y_2)\right)
\end{equation}
On the other hand, the Bayes classifier puts the two observations in the same class when
\begin{equation*}
    \left(\PP_{\theta}\left(X_1 = 2 \mid Y_{1:2}\right) - \PP_{\theta}\left(X_1 = 1 \mid Y_{1:2}\right)\right) \left(\PP_{\theta}\left(X_2 = 2 \mid Y_{1:2}\right) - \PP_{\theta}\left(X_2 = 1 \mid Y_{1:2}\right)\right) \geq 0,
\end{equation*}
or equivalently,
\begin{multline}
\label{eq:n2:bayesclust:cond2}
    \left(p q f_{2}(Y_1) f_{1} (Y_2) + p \bar{q} f_{2}(Y_1) f_{2} (Y_2) - q \bar{p} f_{1}(Y_1) f_{1} (Y_2) - q p f_{1}(Y_1) f_{2} (Y_2)\right)\\
    \times \left(q p f_{1}(Y_1) f_{2} (Y_2) + p \bar{q} f_{2}(Y_1) f_{2} (Y_2) - q \bar{p} f_{1}(Y_1) f_{1} (Y_2) - p q f_{2}(Y_1) f_{1} (Y_2)\right)\geq 0.
\end{multline}
The two conditions \eqref{eq:n2:bayesclust:cond1} and \eqref{eq:n2:bayesclust:cond2} are not equivalent. In the simple situation of Bernoulli emissions $\left(f_1, f_2\right) = \left(\mathcal{B}(\alpha_1), \mathcal{B}(\alpha_2)\right)$, and $Y_{1:2} = (1, 1)$, the Bayes clustering puts the two observations in the same cluster when 
\begin{equation*}
    q\bar{p} \alpha_1^{2} + p\bar{q}\alpha_2^{2} > 2 p q \alpha_1 \alpha_2
\end{equation*}
while the Bayes classifier puts them always in the same cluster since the corresponding  condition becomes
\begin{equation*}
    \left(p\bar{q}\alpha^{2}_2 - q\bar{p}\alpha_{1}^{2}\right)^{2} \geq 0.
\end{equation*}
Note that the first condition is not always ensured when $p$ and $q$ are chosen near $1$. Note also that in the situation where $p+q = 1$, the dependence structure is lost and the two conditions become equivalent to
    \begin{equation*}
        (qf_{1}(Y_1) - p f_{2}(Y_1)) (qf_{1}(Y_2) - p f_{2}(Y_2)) \geq 0
    \end{equation*}
which is coherent with Theorem \ref{thm:bayes-risk:min:iid:J=2} in the i.i.d. case. This highlights the strong difference between the dependent and independent setting. The proof of Theorem \ref{thm:bayes-risk:min:hmm:J=2} can be found in \cite[Section~S7]{GKN2025SM}.

We now establish the equivalence between the two risks under the Markovian dependence of the labels. We consider the following assumption.
\begin{assumption}\label{Assumption_mixing} 
$\delta = \min_{x,x'}Q_{x,x'} > 0$ and $\min_{x}\nu_{x}\geq \delta$.
\end{assumption}
The positive lower bound $\delta$ introduced in Assumption  \ref{Assumption_mixing} makes the hidden Markov chain irreducible. It will be used in proving deviation inequalities or when forgetting properties of the chain are needed. Even if the minimizers of the risk of classification and the risk of clustering might differ, we are still able to prove the equivalence between the two risks. The following theorem shows this is the case for $n$ large enough.

\begin{theorem}
  \label{thm:bayes-risk:lb:hmm:J=2}
  If $J = 2$, then for all $\theta \in \Theta^{\mathrm{dep}}$ such that  Assumption \ref{Assumption_mixing} holds  and all $n \geq 1$ 
  \begin{align*}
      (1 - \tilde{\alpha}_n) \inf_{h \in \mathcal{H}_n}\ClassifRisk(\theta,h) \leq \inf_{g \in \mathcal{G}_n}\ClusterRisk(\theta,g) \leq \inf_{h \in \mathcal{H}_n}\ClassifRisk(\theta,h),
  \end{align*}
  where $\tilde{\alpha}_n = 2e\big(\frac{1-\delta}{\delta} \big)^{4} [\frac{1-\delta}{\delta}\sqrt{\log(2)/2n}]^{1 - 2/n}$.
\end{theorem}
Thanks to Theorem \ref{thm:bayes-risk:lb:hmm:J=2}, it suffices to study the Bayes risk of classification, since any bound on this risk can be extrapolated to the risk of clustering, regardless of the magnitude of the Bayes risk of clustering. The proof of Theorem~\ref{thm:bayes-risk:lb:hmm:J=2} is given in \cite[Section~S6]{GKN2025main}.

Let us now consider the case $J > 2$. As in the i.i.d. setting where the minimizers of the risks do not always coincide, it is always possible to find a set of parameters for which the HMM observations are dependent, but the results of clustering using the Bayes clusterer and the partition induced by the Bayes classifier are not the same, as shown in the next theorem.
\begin{theorem}
    \label{thm:bayes-risk:min:hmm:J>2}
    In the case of dependent labels, for all $J > 2$ and all $n\geq 2$, there exists a subset $\Tilde{\Theta}_{n, J} \subset \Theta^{\mathrm{dep}}$ such that
        \begin{equation*}
            \forall \theta\in\Tilde{\Theta}_{n, J}, \quad\PP_{\theta}\left(g_{\theta}^{\star}(Y_{1:n}) \neq \pi_n\circ h_{\theta}^{\star}(Y_{1:n})\right) > 0.
        \end{equation*}
\end{theorem}
The proof of Theorem \ref{thm:bayes-risk:min:hmm:J>2} can be found in \cite[Section~S8]{GKN2025SM}.

Finally, we establish a result similar to Theorem \ref{thm:bayes-risk:lb:iid:J>2} under the HMM setting.
\begin{theorem}
  \label{thm:bayes-risk:lb:hmm:J>2}
  For all $\theta \in \Theta^{\mathrm{dep}}$ such that  Assumption \ref{Assumption_mixing} holds  and all $n \geq 1$, the following bounds are true
  \begin{align*}
    \inf_{g \in \mathcal{G}_n}\ClusterRisk(\theta,g)%
    &\geq\inf_{h \in \mathcal{H}_n}\ClassifRisk(\theta,h) - \frac{1}{1-\rho_{0}} \sqrt{\frac{\log(J!)}{2n}},\\
    \inf_{g \in \mathcal{G}_n}\ClusterRisk(\theta,g)
    &\geq (1 - \tilde{\xi}_n)\inf_{h \in \mathcal{H}_n}\ClassifRisk(\theta,h) - (J^{2}+1)e^{-2n(1-\rho_0)^{2}\beta^{2}/25},
  \end{align*}
  where $\beta = \min_{i, j\neq k}\PP_{\theta}\left(X_{i}\in\left\{j, k\right\}\right)$, $\rho_0 = \frac{1 - J\delta}{1 - (J-1)\delta}$, and $\tilde{\xi}_n = \frac{5}{\beta(1 - \rho_0)}\sqrt{\log(J!)/(2n)}$.
\end{theorem}
The proof of Theorem~\ref{thm:bayes-risk:lb:hmm:J>2} is given in \cite[Section~S6]{GKN2025SM}. Note that when the HMM is stationary, $\beta = \min_{j\ne k}(\nu_j + \nu_{k})$ as in the i.i.d. case. Notice also that when $\delta= 1/J$, the observations are i.i.d. with uniform distribution over the set $\mathbb{X}$, and we recover the first inequality for i.i.d. observations. However, we do not recover the inequality for i.i.d. observations in general from that of HMM  observations.


As for the i.i.d setting, even if clustering using the Bayes classifier differs sometimes from that of the Bayes clusterer (as shown in Theorem \ref{thm:bayes-risk:min:hmm:J=2} and Theorem \ref{thm:bayes-risk:min:hmm:J>2}), Theorem \ref{thm:bayes-risk:lb:hmm:J>2} provides guarantees for the risk of clustering using the Bayes classifier as shown by the following corollary.

\begin{corollary}\label{cor:Bayes_classif:hmm}
    For all $\theta \in \Theta^{\mathrm{dep}}$ and all $n\geq 1$ the following bounds hold
    \begin{equation*}
    \inf_{g \in \mathcal{G}_n} \ClusterRisk(\theta, g) \leq \ClusterRisk\left(\theta, \pi_n\circ h^{\star}_{\theta}\right) \leq \frac{1}{1 - \tilde\xi_{n}} \left(\inf_{g \in \mathcal{G}_n}\ClusterRisk(\theta,g) + (J^{2}+1)e^{-2n(1-\rho_0)^{2}\beta^{2}/25}\right)
\end{equation*}
where $\tilde{\xi}_n$, $\beta$ and $\rho_0$ are as in Theorem~\ref{thm:bayes-risk:lb:hmm:J>2}. When there are only two classes
\begin{equation*}
    \inf_{g \in \mathcal{G}_n} \ClusterRisk(\theta, g) \leq \ClusterRisk\left(\theta, \pi_n\circ h^{\star}_{\theta}\right) \leq \frac{1}{1 - \tilde\alpha_{n}} \inf_{g \in \mathcal{G}_n}\ClusterRisk(\theta,g) 
\end{equation*}
where $\tilde\alpha_{n}$ is as in Theorem~\ref{thm:bayes-risk:lb:hmm:J=2}.
\end{corollary}

\subsection{A key quantity for the Bayes risk of clustering for both I.I.D. and HMM}
\label{sec:estimates-bayes-risk}
We now state our main result which proves upper and lower bounds on the  Bayes risks in function of a quantity measuring the separation between the emission densities up to constants depending on the transition matrix. These bounds translate into bounds on the risk of clustering thanks to the result above.
Let $\Lambda \coloneqq \int_{\mathbb{Y}}\min_{x_{0}\in\mathbb{X}}\left[\sum_{x\neq x_0}f_{x}(y)\right]d\mathcal{L}(y)$.
\begin{theorem}\label{thm:bayes-risk:key_quantity}
Under Assumption \ref{Assumption_mixing}, for all $n\geq 1$, the  Bayes risk of classification satisfies:
\begin{equation*}
    \begin{split}
        &\forall \theta \in \Theta^{\mathrm{ind}},\quad \delta\Lambda \leq \inf_{h\in \mathcal{H}_n}\ClassifRisk(\theta,h) \leq (1 - (J-1)\delta)\Lambda\\
        &\forall \theta \in \Theta^{\mathrm{dep}},\quad\frac{\delta^{2}}{1 - (J-1)\delta}\Lambda \leq \inf_{h\in \mathcal{H}_n}\ClassifRisk(\theta,h)  \leq (1 - (J-1)\delta)\Lambda
    \end{split}
\end{equation*}
\end{theorem}
The proof of Theorem~\ref{thm:bayes-risk:key_quantity} is given in \cite[Section~S10]{GKN2025SM}. Note that the bounds vanish when all the emission densities have disjoint supports. Note also that upper and lower bounds match when $\delta = \frac{1}{J}$ and the risk corresponds to $\Lambda/J$ in this case. This situation corresponds to i.i.d. observations derived from a mixture with $J$ components of equal weights. This proves in particular the tightless of the bounds which can not be improved by any absolute multiplicative constant without restricting the parameter space. Thanks to the results comparing the Bayes risks, $\Lambda$ is the appropriate measure of the difficulty of clustering in many regimes as shown in the following corollary. Recall the definition of $\alpha_n$ from Corollary~\ref{cor:equiv:class:clust}, of $\xi_n$ from Theorem~\ref{thm:bayes-risk:lb:iid:J>2} of $\tilde{\alpha}_n$ from Theorem~\ref{thm:bayes-risk:lb:hmm:J=2}, of $\tilde{\xi}_n$ from Theorem~\ref{thm:bayes-risk:lb:hmm:J>2}, and of $\beta = \min_{i\in[n], j\neq k\in\mathbb{X}}\PP_{\theta}\left(X_{i}\in\left\{j, k\right\}\right)$.

\begin{corollary}\label{cor:key_quantity}
    Under Assumption \ref{Assumption_mixing}, the following holds.

    \begin{itemize}
    \item When $J=2$:
    \begin{align*}
        &\forall \theta \in \Theta^{\mathrm{ind}}, \quad  (1-\alpha_n) \delta \Lambda \leq \inf_{g \in \mathcal{G}_n} \ClusterRisk(\theta, g) \leq (1 - \delta) \Lambda, \\
        &\forall \theta \in \Theta^{\mathrm{dep}}, \quad \frac{\delta^2(1-\tilde{\alpha}_n)}{1 - \delta} \Lambda \leq \inf_{g \in \mathcal{G}_n} \ClusterRisk(\theta, g) \leq (1 - \delta) \Lambda.
    \end{align*}
    
    \item When $J>2$ and $\theta = (\nu, Q, \left(f_x\right)_{x\in\mathbb{X}}) \in \Theta^{\mathrm{ind}}$ is such that $\delta \Lambda \geq 4J^2 e^{-n\beta/8}$ and $n$ is sufficiently large to have $\xi_n \leq \frac{1}{2}$:
    \begin{equation*}
        \frac{\delta}{4} \Lambda \leq \inf_{g \in \mathcal{G}_n} \ClusterRisk(\theta, g) \leq (1 - (J-1)\delta) \Lambda.
    \end{equation*}
        
    \item When $J > 2$ and $\theta = (\nu, Q, \left(f_x\right)_{x\in\mathbb{X}}) \in \Theta^{\mathrm{dep}}$ is such that $\delta^2\Lambda \geq 4(1-(J-1)\delta)(J^2+1)e^{-2n(1-\rho_0)^2\beta^2/15}$ and $n$ is sufficiently large to have $\tilde{\xi}_n \leq \frac{1}{2}$:
    \begin{equation*}
        \frac{\delta^2}{4(1-\delta)}\Lambda \leq \inf_{g \in \mathcal{G}_n} \ClusterRisk(\theta, g) \leq (1 - (J-1)\delta) \Lambda.
    \end{equation*}
    \end{itemize}
\end{corollary}
In other words, when there are only two classes, we obtain a tight characterization of the Bayes risk of clustering in terms of the separation between the two emission densities (\rom{1}) covering all the regimes (\rom{2}) valid in the parametric and non-parametric setting (\rom{3}) without imposing any separation between the emission densities. In both i.i.d. and HMM settings, when there are only two classes, there exists a positive constant $\alpha(\delta)$ depending only on $\delta$ such that
\begin{equation*}
    \alpha(\delta)\int_{\mathbb{Y}}[f_{1}\wedge f_{2}](y)d\mathcal{L}(y) \leq \inf_{g \in \mathcal{G}_n}\ClusterRisk(\theta,g)\leq (1 - \delta)\int_{\mathbb{Y}}[f_{1}\wedge f_{2}](y)d\mathcal{L}(y).
\end{equation*}
For example, in the case of Gaussian emission distributions with two hidden states, this translates in a clear identification of the signal-to-noise ratio which drives the Bayes risk of clustering. For two Gaussian emission densities with means $\mu_0$ and $\mu_1$ and the same covariance matrix $\Sigma$, the Bayes risk of clustering ensures:
\begin{equation*}
    \frac{\alpha(\delta)}{2}\exp\left(-\frac{\text{SNR}}{4}\right) \leq \inf_{g \in \mathcal{G}_n}\ClusterRisk(\theta,g)\leq (1-\delta)\exp\left(-\frac{\text{SNR}}{8}\right)
\end{equation*}
where $\text{SNR} = (\mu_{0} - \mu_{1})^{\intercal}\Sigma^{-1}(\mu_{0} - \mu_{1})$. 

\subsection{Reaching the Bayes risk}
\label{sec:reaching}

\subsubsection{I.I.D. setting}
While there is no formal proof of this fact, it seems that there is no way of reaching the Bayes risk of clustering or classification without strong assumptions on the mixture components under the non-parametric i.i.d. mixture model. Without structural assumptions, the associated Bayes classifier can never be learnt nor approximated from the data in the i.i.d. case because the model is not identifiable and thus, the mixture components can not be estimated.  The algorithms proposed in the literature perform well in clustering i.i.d. observations only when the clusters are assumed to be separated. This is the case for the $k$-means and its variants \cite{GV19, MVW17} and the spectral algorithms \cite{SBY09, KSH05} to cite a few. We refer to \cite{MR3889693} and Chapter 12 of \cite{G21} for a review of model-based clustering techniques.
\subsubsection{HMM setting}\label{sec:slow-rates-plugin}
Unlike the i.i.d. hypothesis, the HMM hypothesis allows us to identify the model without any assumptions about the emission distributions, apart from the fact that they are distinct \cite{GHA16}. This allows the construction of simple clustering procedures with risk comparable to the Bayes risk of clustering. Given that the transition matrix of the hidden Markov chain is non-singular and ergodic, all parameters -- including the number \( J \) of hidden states and the emission distributions -- can be identified from the distribution of \( K \) consecutive observations \( (Y_1, \ldots, Y_K) \), provided that the emission distributions are distinct and \( K \) is sufficiently large compared to \( J \) \cite{MR3509896}. This allows complete flexibility on the emission distributions, provided a Markovian dependence of labels. Within the framework of this model, non-parametric estimation of the mixture components becomes possible without any restrictions on the population densities, apart from the fact that they are distinct, and various estimation procedures have been proposed \cite{YEC16, YES17, MR3862446, MR4347381, ACG21, AGN23}. Although we have shown in Theorems \ref{thm:bayes-risk:min:hmm:J=2} and \ref{thm:bayes-risk:min:hmm:J>2} that the Bayes clusterer $g^{\star}_{\theta}$ does not necessarily coincide with $\pi_n\circ h^{\star}_{\theta}$, this distinction matters only to establish strong decision-theoretic foundations for clustering. In practice, the Bayes classifier still exhibits a competitive performance in clustering. Corollary $\ref{cor:Bayes_classif:hmm}$ shows that the risk of clustering of the Bayes classifier approximates closely the Bayes risk of clustering, proving thus its near-optimality. The widely used method of clustering when the observations are drawn from a HMM consists in approaching the behavior of the Bayes classifier $h_{\theta}^\star$ by plugging-in estimated parameters $\hat\theta$ and using the induced clusterer $\pi_n \circ h_{\hat\theta}^\star$. It is not worth the effort of computing $g_\theta^\star$ and use $g_{\hat\theta}^\star$ because the price to pay in the excess risk for trading the true $\theta$ by its estimate $\hat\theta$ is most likely of several order of magnitude larger than the price to pay in using $\pi_n \circ h_\theta^\star$ in place of $g_\theta^\star$.
Thus in this section, we focus on the clustering rule
\begin{equation*}
    g_{\hat{\theta}}^{\star}(Y_{1:n}) \coloneqq \pi_{n} \circ h_{\hat{\theta}}^{\star}(Y_{1:n}) = \pi_{n}\left( \left(\argmax_{x\in\mathbb{X}}\PP_{\hat{\theta}}\left(X_{i} = x\mid Y_{1:n}\right)\right)_{1\leq i \leq n}\right)
\end{equation*}
where $\hat{\theta} = \hat{\theta}(Y_{1:n})$ is an estimator constructed using the celebrated tensor method \cite{anandkumar2014tensor,ACG21}. The main advantage in using $\pi_n\circ h_{\hat\theta}^\star$ is that in contrast with $g_{\hat\theta}^\star$ the classifier $h_{\hat\theta}^\star$ is easily computed thanks to the recurrence formulas ensured by the Forward-Backward algorithm \cite{CMT05}.
Theorem~\ref{thm:excess_risk:clust:rate} below controls the excess risk of this clustering procedure. We will make the following assumption:
\begin{assumption}
    \label{Assumption_stat}
    The initial distribution $\nu$ is the stationary distribution of $\mathbf{X}$.
\end{assumption}

Notice however that under Assumption \ref{Assumption_stat}, the second part of Assumption \ref{Assumption_mixing} follows directly from the first part.

\begin{assumption}\label{Assumption_4}
    $Q$ is full-rank and aperiodic.
\end{assumption}

Under Assumptions \ref{Assumption_mixing}, \ref{Assumption_stat} and  \ref{Assumption_4}, the hidden Markov chain is stationary ergodic. The following assumption is sufficient to build estimators using the empirical distribution of the distribution of three consecutive observations.

\begin{assumption}\label{Assumption_3}
    The emission densities $(f_{x})_{x\in\mathbb{X}}$ are compactly supported and $C^{\star} = \int\frac{dy}{\sum_{x\in\mathbb{X}}f_x (y)} < \infty$, the integral is over the union of the supports of the emission densities.
\end{assumption}

\begin{assumption}\label{Assumption_5}
$\left(f_{x}\right)_{x\in\mathbb{X}}$ are linearly independent and belong to $C^{s}(\mathbb{R})$ the space of locally H\"{o}lder continuous functions.
\end{assumption}

We now state the theorem. Its proof is given in \cite[Section~S11]{GKN2025SM}. 

\begin{theorem}\label{thm:excess_risk:clust:rate}
    Let $D_{n} \xrightarrow[]{}+\infty$ arbitrarily slowly. There exists a sequence of randomized estimators $(\hat{\theta}_n)_{n\geq 1}$ such that for all $\theta\in\Theta^{\mathrm{dep}}$ satisfying Assumption~\ref{Assumption_mixing} to~\ref{Assumption_5}
\begin{equation*}
        \EE[\ClusterRisk(\theta,\pi_n\circ h_{\hat{\theta}}^{\star})] - \inf_{g \in \mathcal{G}_n}\ClusterRisk(\theta,g)= \mathcal{O}\left(D_{n}^{5/2}\left(\frac{\log(n)}{n}\right)^{\frac{s}{2s + 1}}\right)
\end{equation*}
where the expectation $\EE[\cdot]$ is understood with respect to the randomness of the algorithm. 
\end{theorem}

 The cornerstone of the proof is the analysis of how errors in estimation of the parameters propagate to errors in the filtering and smoothing distributions proved in \cite{YES17}. To achieve this rate for the excess risk, we must specify an estimator. Here we make the choice of using a modified version of the spectral algorithm of \cite{ACG21} which relies on the tensor method developped in \cite{anandkumar2014tensor}. The full algorithm is given in details in \cite[Section~S11]{GKN2025main}. The main difference with the original algorithm of \cite{ACG21} is a modification guaranteeing that the algorithm outputs the same permutation for the estimation of the transition matrix $\hat Q$ and the emission densities $\hat f_1,\dots,\hat f_J$. Indeed, in the previous works, the aim was more to get upper bounds on the risks up to label-switching, and the analysis was done for $Q$ and the emission distributions separately which does not guarantee that the permutation used in the control of the error of estimation of $Q$ and the emission densities is the same. Albeit the modification is rather natural, proof that it works require a substantial effort.
 
 Notice that the rate on the excess risk is exactly the estimation rate of H\"older regular functions in sup-norm. However, since we do not observe realizations of each emission density separately, this rate is not a straightforward consequence of density estimation theory. See \cite{ACG21} where the usual non-parametric rate for the estimation of densities in sup-norm is obtained in the HMM context.
 
 Though we analyse the question for the spectral estimator, we believe that results similar to Theorem \ref{thm:excess_risk:clust:rate} hold for most estimation procedures previously proposed in the literature, putting in the upper bounds the bounds on the estimation risk up to label-switching obtained in those works.
This is why we shall use in Section \ref{sec:simu} the least-squares estimation method for which a public and efficient code exists \cite{YohannGH}.

\section{Numerical simulations}
\label{sec:simu}
We present here the results of numerical simulations which leverage the added value of non-parametric clustering under hidden Markov modelling. We will consider two examples in which the hidden states will be generated through the same transition matrix $$Q = \begin{pmatrix}
  0.8 & 0.2\\ 
  0.3 & 0.7
\end{pmatrix}.$$
\begin{example}\label{example_1}
     A sample of $n = 5.10^{4}$ observations of a HMM with transition matrix $Q$ and emissions: $F_{1} = \frac{1}{2}\left(\mathcal{N}(1.7, 0.2) + \mathcal{N}(7, 0.15)\right)$ and $F_{2} = \frac{1}{2}\left(\mathcal{N}(3.5, 0.2) + \mathcal{N}(5, 0.4)\right)$.
\end{example}
\begin{example}\label{example_2}
     A sample of $n = 10^{5}$ observations of a HMM with transition matrix $Q$ and emissions: $F_{1} = \frac{1}{2}\left(\mathcal{N}(3, 0.6) + \mathcal{N}(7, 0.4)\right)$ and $F_{2} = \frac{1}{2}\left(\mathcal{N}(5, 0.3) + \mathcal{N}(9, 0.4)\right)$.
\end{example}
On these examples, we use the plug-in classifier whose clustering risk has been controlled in Theorem \ref{thm:excess_risk:clust:rate}.
Recall our procedure works as follows:
\begin{itemize}
    \item First, the emission densities are estimated using the observations. We use the penalized least squares estimator proposed in \cite{YEC16} whose code has been made public in \cite{YohannGH}. We use the histogram basis for the estimation.
    \item Second, the Forward-Backward algorithm is used to compute the a posteriori distributions of the hidden states under the estimated model parameters and given the observations.
    \item Third, the hidden states are estimated by maximizing the a posteriori distributions. 
\end{itemize}
The results of clustering using the Forward-Backward algorithm will be compared to those using the $k$-means algorithm. Since we have access to the hidden states, the error of clustering can be estimated by choosing the best permutation. Figures~\ref{plots:estimation}, \ref{plots:clustering:Ex:1}, and~\ref{plots:clustering:Ex:2} display the results of estimation and clustering. Performance of Bayes classifier, plug-in classifier and $k$-means are reported in Table~\ref{tab:clust_results}.
\begin{table}[!htb]
    \centering
    \begin{tabular}{|c|c|c|c|c|}
        \cline{2-5}
        \multicolumn{1}{c|}{} & Bayes classifier & Plug-in classifier & $k$-means algorithm & $\Lambda$ \\
        \hline
        Example \ref{example_1} & 1.56\% & 1.61\% & 46.7\% & 0.046 \\ 
        Example \ref{example_2} & 6.42\% & 6.51\% & 47.3\% & 0.165\\ 
        \hline
    \end{tabular} 
    \caption{Errors of clustering using three clustering rules: the Bayes classifier (using the true model parameters), the plug-in classifier (using the estimated parameters) and the $k$-means algorithm.}
    \label{tab:clust_results}
\end{table}
\begin{figure}[!htb]
  \begin{subfigure}[b]{1\textwidth}
    \centering
    \includegraphics[width=1\linewidth]{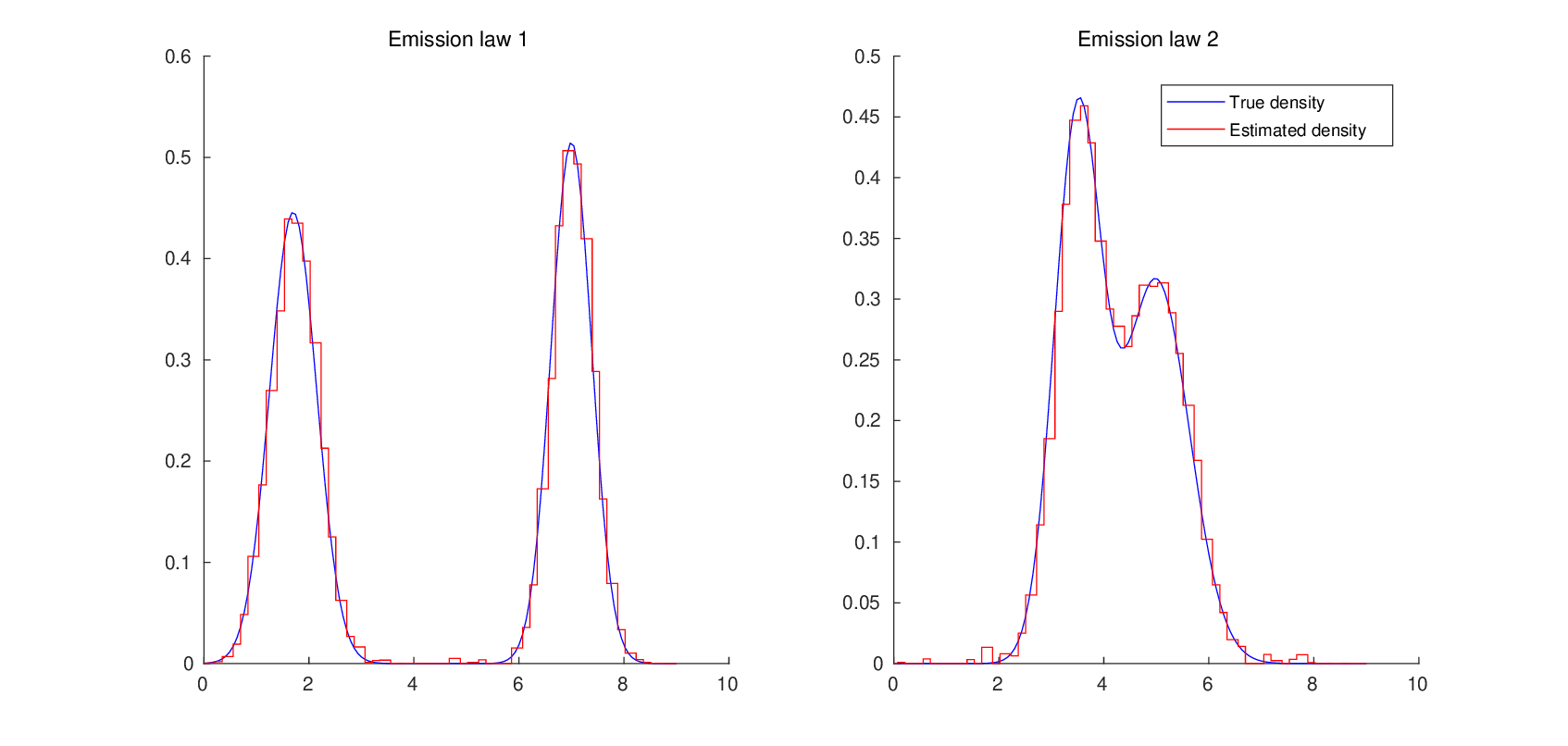} 
    \caption{Estimation results for Example \ref{example_1}}
    \label{plots:estimation:1}
  \end{subfigure}
  \begin{subfigure}[b]{1\textwidth}
    \centering
    \includegraphics[width=1\linewidth]{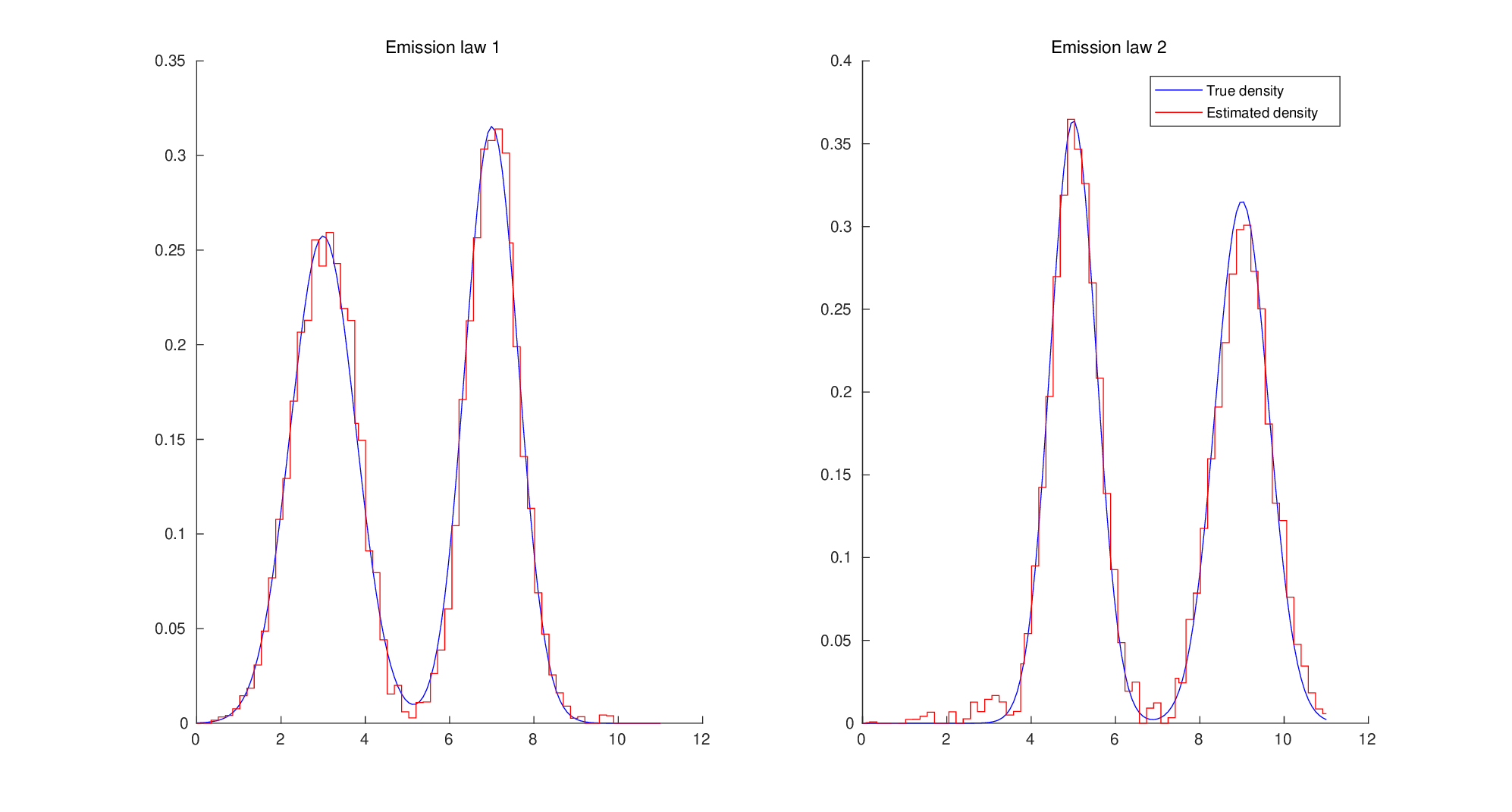} 
    \caption{Estimation results for Example \ref{example_2}} 
    \label{plots:estimation:2}
  \end{subfigure} 
  \caption{Non-parametric penalized least squares density estimation using the histogram basis for Example \ref{example_1} and Example \ref{example_2}} 
  \label{plots:estimation}
\end{figure}
\begin{figure}[!htb]
  \begin{subfigure}{0.5\linewidth}
    \centering
    \includegraphics[scale=0.53]{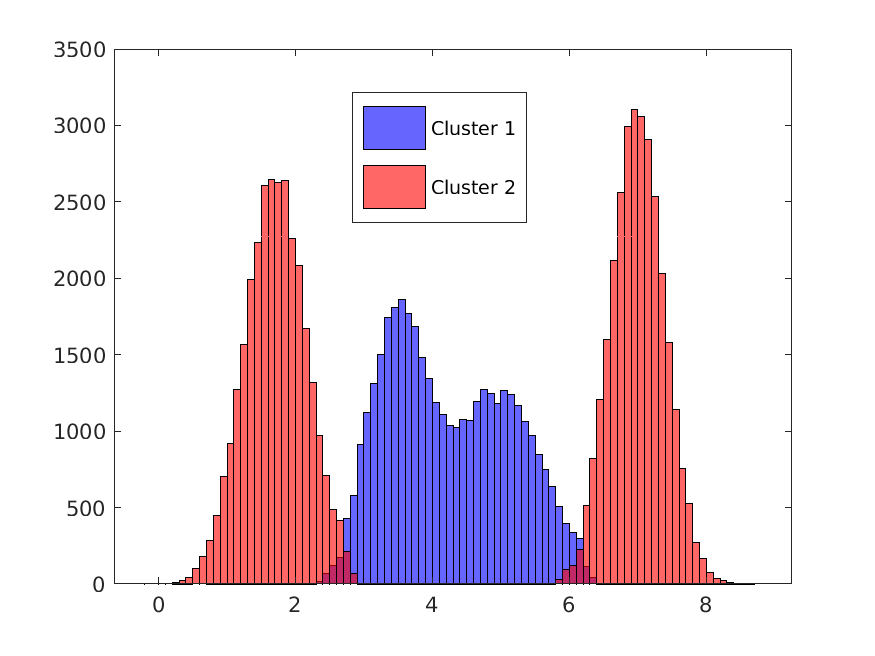} 
    \caption{Clustering using plug-in clusterer} 
    \label{plots:clustering:Ex:1:1}
    \vspace{4ex}
  \end{subfigure}
  \begin{subfigure}{0.5\linewidth}
    \centering
    \includegraphics[scale=0.53]{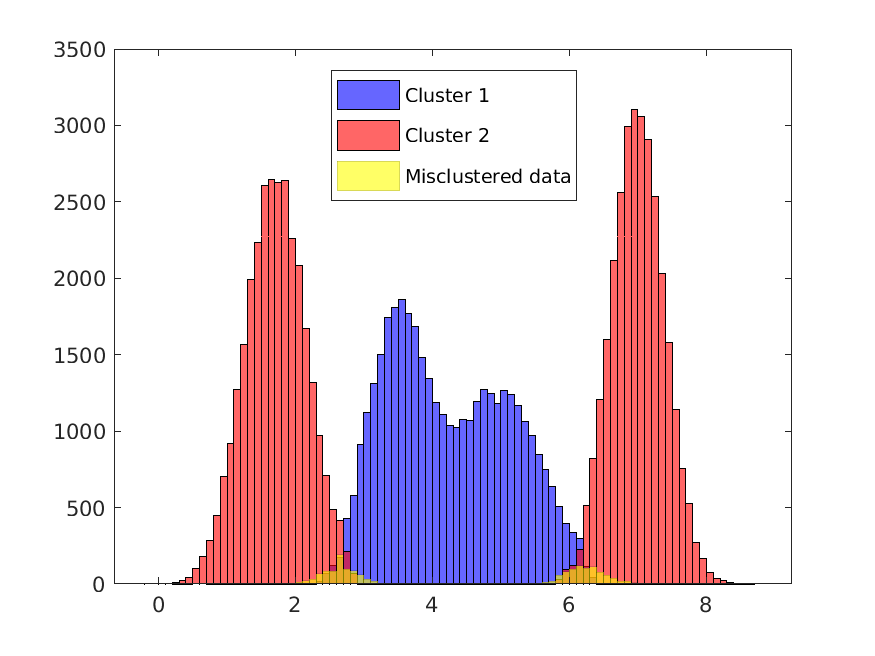} 
    \caption{Misclustered observations for plug-in clusterer} 
    \label{plots:clustering:Ex:1:2}
    \vspace{4ex}
  \end{subfigure} 
  \begin{subfigure}{0.5\linewidth}
    \centering
    \includegraphics[scale=0.53]{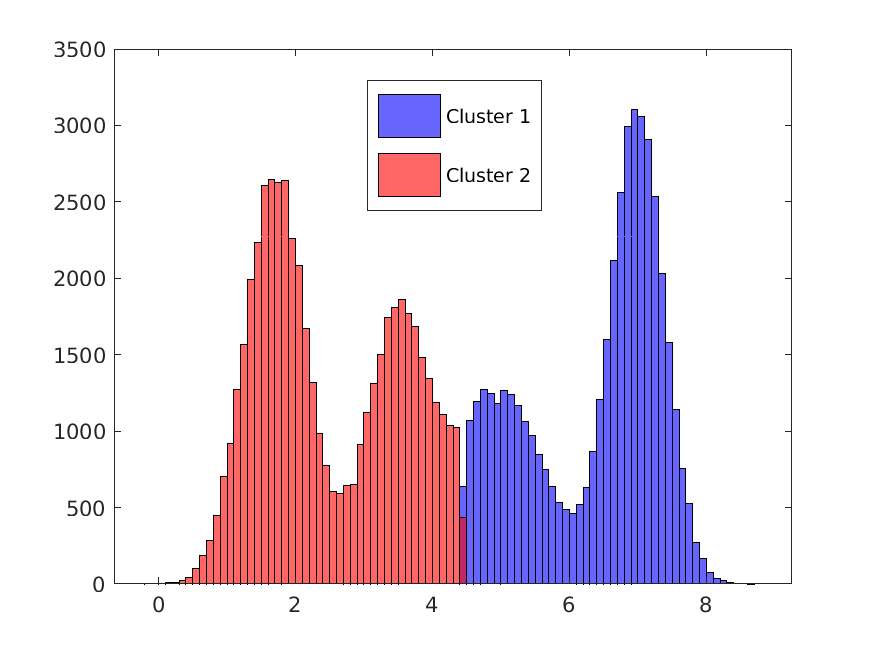} 
    \caption{Clustering using $k$-means} 
    \label{plots:clustering:Ex:1:3}
  \end{subfigure}
  \begin{subfigure}{0.5\linewidth}
    \centering
    \includegraphics[scale=0.53]{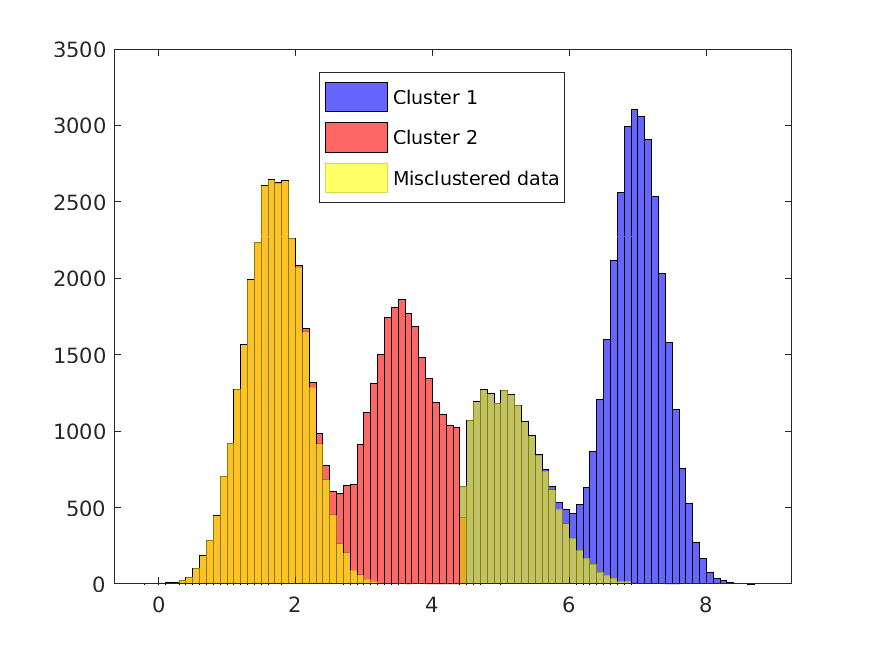} 
    \caption{Misclustered observations for $k$-means} 
    \label{plots:clustering:Ex:1:4}
  \end{subfigure} 
  \caption{Histograms of clusters and clustering errors for Example \ref{example_1}}
  \label{plots:clustering:Ex:1}
\end{figure}
\begin{figure}[!htb] 
  \begin{subfigure}{0.5\linewidth}
    \centering
    \includegraphics[scale=0.53]{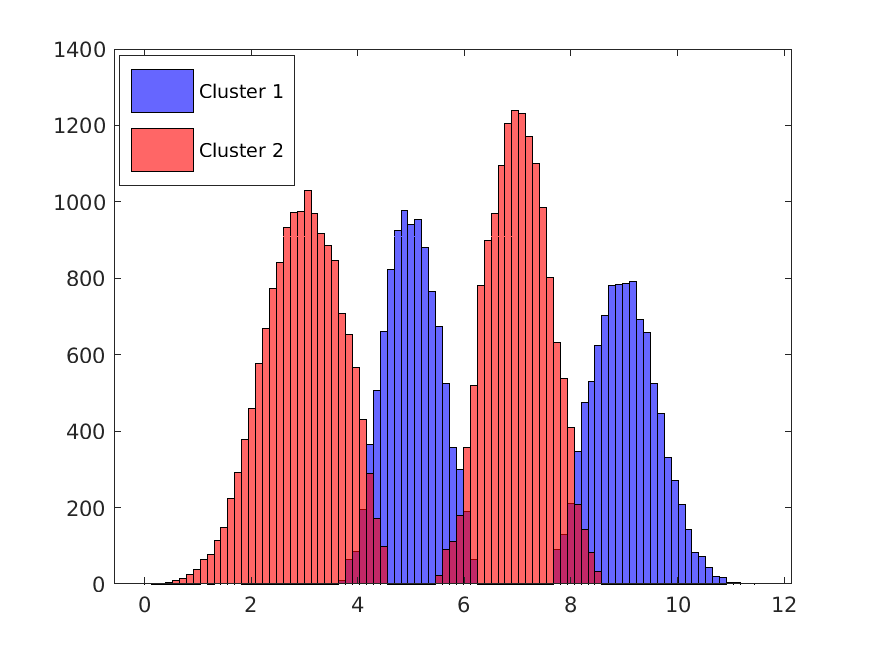} 
    \caption{Clustering using plug-in clusterer} 
    \label{plots:clustering:Ex:2:1}
    \vspace{4ex}
  \end{subfigure}
  \begin{subfigure}{0.5\linewidth}
    \centering
    \includegraphics[scale=0.53]{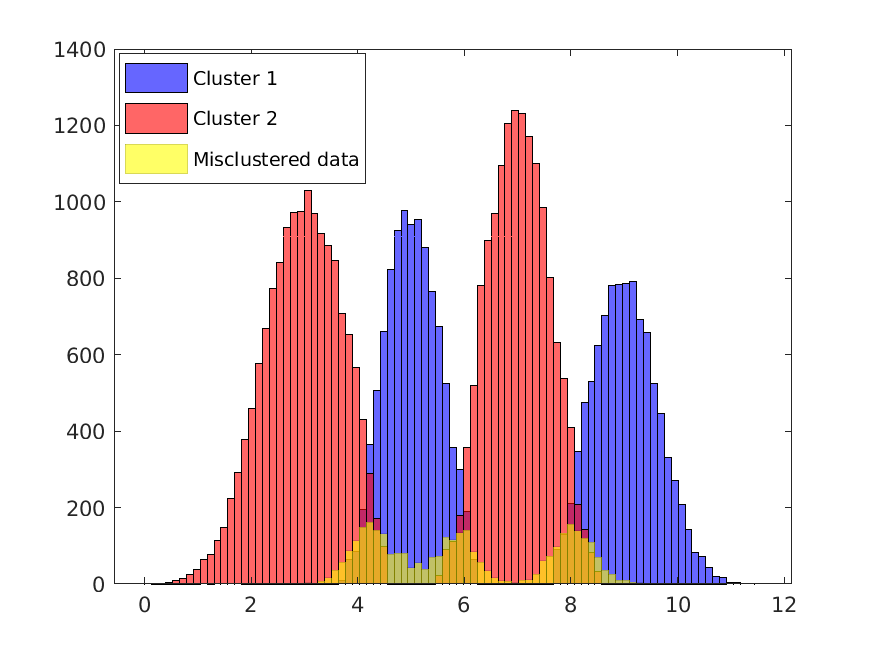} 
    \caption{Misclustered observations for plug-in clusterer} 
    \label{plots:clustering:Ex:2:2}
    \vspace{4ex}
  \end{subfigure} 
  \begin{subfigure}{0.5\linewidth}
    \centering
    \includegraphics[scale=0.53]{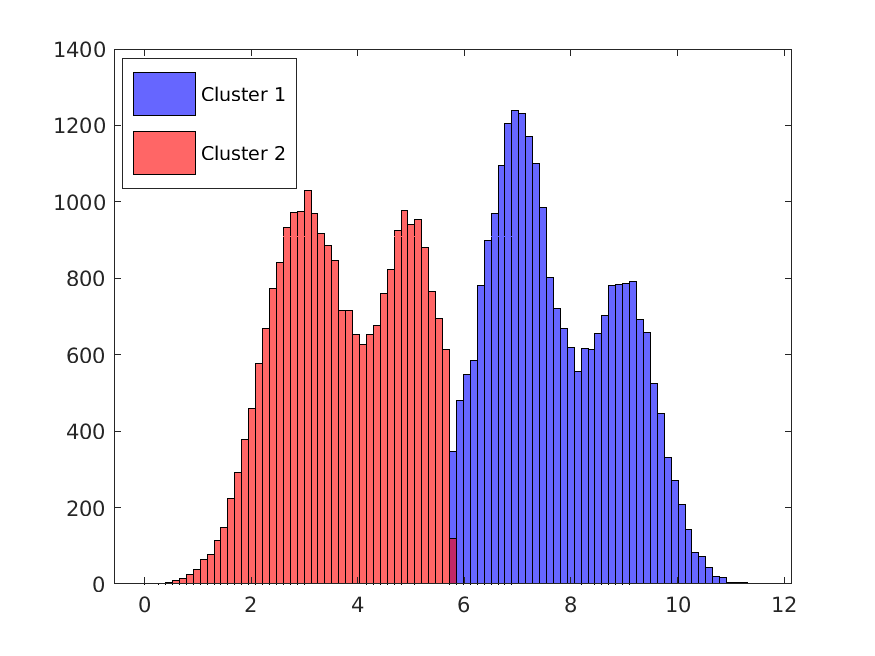} 
    \caption{Clustering using $k$-means} 
    \label{plots:clustering:Ex:2:3}
  \end{subfigure}
  \begin{subfigure}{0.5\linewidth}
    \centering
    \includegraphics[scale=0.53]{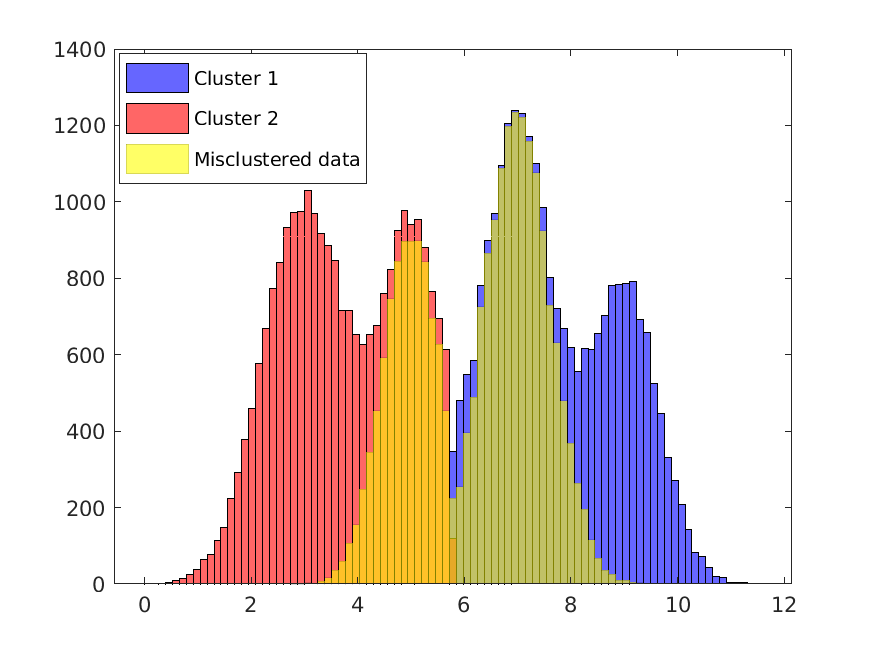} 
    \caption{Misclustered observations for $k$-means} 
    \label{plots:clustering:Ex:2:4}
  \end{subfigure} 
  \caption{Histograms of clusters and clustering errors for Example \ref{example_2}}
  \label{plots:clustering:Ex:2}
\end{figure}
If the observations were independent and the emissions modelled non-parametrically, the unique quantity that could have been estimated consistently would be the stationary distribution. However, under the HMM assumption, the estimation of each emission density with the minimax rate is possible. This is due to the identifiability of the model which holds even if no assumption is made on the emission densities. This is not possible in the independent case. Figure \ref{plots:estimation} shows the estimation results of the emission densities and confirms the theoretical properties of the estimator. On the other hand, Corollary \ref{cor:key_quantity} proves that in the case of a HMM with two classes, clustering errors could appear only in zones where the two emission densities overlap. In Figures \ref{plots:clustering:Ex:1:2}, and \ref{plots:clustering:Ex:2:2}, misclustered observations appear only in the overlaps between the emission densities. 
 Compared to the $k$-means algorithm which is purely geometric and does not exploit the distribution of the observations, the plug-in procedure allows combining together observations even if they are geometrically distant from each other, which is not possible with the $k$-means algorithm. In this context of Gaussian mixtures, the performance of the $k$-means algorithm is mediocre as depicted in Table \ref{tab:clust_results} and does not improve significantly when the overlap between the emission densities is small. However, for the plug-in procedure, the more separated are the emission densities, the better are the results of clustering. 
\section{Discussions and Perspectives}\label{sec:perspectives}
This work focuses on an in-depth study of Bayes risks of clustering and classification. This analysis has led us to prove a form of equivalence between both risks. After identifying the key quantity which measures the difficulty of the classification task, it was extrapolated to the Bayes risk of clustering in several regimes. Finally, the excess risk of the plug-in procedure was studied. Although the analysis is sufficiently detailed to ensure a thorough understanding of both problems, there are still some interesting questions which were not covered by this work. We give a small overview below.
\paragraph*{Lower-bound on entries of $Q$}
Throughout our analysis of the Bayes risk of clustering, we have used Assumption \ref{Assumption_mixing} which was crucial in obtaining the lower-bounds. In the absence of such an assumption, the lower-bound of Theorem \ref{thm:bayes-risk:key_quantity} no more matches the upper-bound and the magnitude of the Bayes risk of classification can not be precisely understood. The same thing applies to the Bayes risk of clustering. In addition, the control of the excess risk of the plug-in procedure is no more guaranteed. 
\paragraph*{Approaching the frontier to independence}
This situation happens when the emission distributions are nearly similar or when the transition matrix has almost equal lines. In this case, one can hope to improve the coefficient $\frac{\delta^{2}}{1 - (J-1)\delta}$ which appears in the lower-bound on the risk of classification (Theorem \ref{thm:bayes-risk:key_quantity}) to $\delta$ as in the independent case. In fact, in this situation, the dependence between the observations is so weak and the effect of future and past observations on the current classification rule is so negligible that the magnitude of the Bayes risk of classification is the same as in the i.i.d. setting. On the other hand, as shown in \cite{AGN21, AGN23}, estimation of the model parameters no more becomes possible when approaching the frontier. The plug-in procedure should not work as well.

\paragraph*{Lower bounds on the Bayes risk of clustering when it is very small}

Theorems~\ref{thm:gap:classif:cluster},  \ref{thm:bayes-risk:lb:iid:J>2}, \ref{thm:bayes-risk:lb:hmm:J=2} and ~\ref{thm:bayes-risk:lb:hmm:J>2} establish lower bounds on the Bayes risk of clustering in terms of the Bayes risk of classification. When $J=2$ these bounds are meaningful regardless of how small is the Bayes risk of clustering. When $J>2$, however, these bounds can be vacuous if the Bayes risk of classification gets too small. This is not an artifact of our bounds since we have shown in Proposition~\ref{pro:bayes-risk:non-equivalence} that the two Bayes risks are not uniformly comparable when $J > 2$. Whence from the current work we only know that the Bayes risk of clustering is driven by $\Lambda$ in the region of parameters for which it is not exponentially small in $n$ and that it can not be driven by $\Lambda$ otherwise. If it was the case, then it would be equivalent to the Bayes risk of classification in contradiction with the Proposition~\ref{pro:bayes-risk:non-equivalence}. Understanding the Bayes risk of clustering in the region of extreme parameters is still an open question.


\paragraph*{Fast rates}

Theorem~\ref{thm:excess_risk:clust:rate} has interest mainly in the situation where $\inf_{g \in \mathcal{G}_n}\ClusterRisk(\theta,g)  \gtrsim (\log(n)/n)^{\frac{s}{2s+1}}$. In this regime, Theorem~\ref{thm:excess_risk:clust:rate} tells us that the plug-in procedure has a risk of the same magnitude. However, in the situation where the magnitude of the Bayes risk of clustering is much smaller, that is when the emission distributions are very separated, one can hope to obtain faster rates. For example, when $\inf_{g \in \mathcal{G}_n}\ClusterRisk(\theta,g) = \mathcal{O}\left(e^{-c n}\right)$ for $c$ a positive absolute constant, one can hope to show that the risk of the plug-in procedure is exponentially small in $n$. The following lemma represents a first step for the proof of such a result. 
    
\begin{lemma}
\label{lem:fastrates}
    For all $0<\gamma<1 / 2$ and all $\theta \in \Theta$
    \begin{equation}
        \label{eq:lem:fastrates}
        \ClassifRisk(\theta, h_{\hat{\theta}}) \leq \frac{\inf _{h \in \mathcal{H}_n} \ClassifRisk(\theta, h)}{1 / 2-\gamma}+\frac{1}{n} \sum_{i=1}^n \mathbb{P}_\theta\left(\left\|\phi_{\theta, i \mid n}-\phi_{\hat{\theta}, i \mid n}\right\|_{\mathrm{TV}}>\gamma\right) .
    \end{equation}
    where $\phi_{\theta, i \mid n} = \PP_{\theta}\left(X_{i} \in . \mid Y_{1:n}\right)$ and $h_{\hat{\theta}}$ is the plug-in classifier defined in Section~\ref{sec:slow-rates-plugin}.
\end{lemma}
The proof of Lemma~\ref{lem:fastrates} is given in \cite[Section~S13]{GKN2025SM}.

Observe that the second term of the rhs of \eqref{eq:lem:fastrates} is a large deviation term which may eventually decrease exponentially fast in $n$. The only price to pay to obtain large deviation type of decay is a constant factor of at least 2 in front of $\inf _{h \in \mathcal{H}_n} \ClassifRisk(\theta, h)$. In situations where $\inf _{h \in \mathcal{H}_n} \ClassifRisk(\theta, h)$ is small, this might be advantageous compared to the bound in Theorem~\ref{thm:excess_risk:clust:rate}. 
In Proposition 2.2 of \cite{YES17}, the authors prove the inequality 
\begin{multline}
\label{eq:smoothing-control}
    \lVert \phi_{\theta, i|n}(.,Y_{1:n}) - \phi_{\hat{\theta}, i|n}(.,Y_{1:n})\rVert_{\mathrm{TV}} \leq
    \frac{4(1 - \delta)}{\delta^{2}}\Biggl(\rho^{i-1}\lVert \nu - \hat{\nu}\rVert_{2}\\ + \left(1/(1 - \rho) + 1/(1 - \hat{\rho}) \right)\lVert Q - \hat{Q}\rVert_{\mathrm{F}} + \sum_{l=1}^{n}\delta\frac{(\hat{\rho}\vee \rho)^{|l-i|}}{c^{\star}(Y_l)}\max_{x\in\mathbb{X}}\left|f_{x}(Y_{l}) - \hat{f}_{x}(Y_{l})\right|\Biggr)
\end{multline}
where $c^{\star}(y) = \min_{x\in\mathbb{X}}\sum_{x^{\prime}\in\mathbb{X}}Q(x, x^{\prime})f_{x^{\prime}}(y)$, $\rho=1-\delta/(1-\delta)$ and $\hat{\rho}=1-\hat{\delta}/(1-\hat{\delta})$ where $\hat{\delta}$ is an estimator of $\delta$. \eqref{eq:smoothing-control} can be used to control the second term in the rhs of \eqref{eq:lem:fastrates}. This would require to derive large deviation inequalities for all the terms involved in \eqref{eq:smoothing-control}, which turns out to be a rather challenging problem.

\paragraph*{Optimal excess risk}

Although we obtain upper bounds on the excess risk of clustering of the plug-in procedure (Theorem~\ref{thm:excess_risk:clust:rate}), we do not know the optimal rate of decay of the excess risk. In particular, it is unknown if the plug-in procedure achieves optimal excess risk. The Lemma~\ref{lem:fastrates} suggests that when the Bayes risk is smaller than $\mathcal{O}(n^{-s/(2s+1)})$ then our upper bounds on the excess risk of the plug-in could be improved. Yet without optimality guarantees. Determining the optimal excess risk is an open and interesting question to investigate.

\paragraph*{Alternatives to plug-in}

Under the hidden Markov modeling, the most straightforward way to take advantage of the identifiability of the model is to estimate the model parameters and use them for clustering through the plug-in procedure. Unlike the i.i.d. case where algorithms such as $k$-means can be used as an alternative, we do not know of any alternative to the plug-in in the HMM case. It would be very interesting to find clustering procedures that leverage the nonparametric identifiability of HMM without relying on estimating the parameters first.

%% file: supplement-abstract.tex
This document is supplementary material for the article \cite{GKN2025main}. It contains the missing proofs. We refer to the main document \cite{GKN2025main} for all the definitions.

%% file: supplement-content.tex
\section{Proof of \texorpdfstring{\cite[Theorem~\ref*{main-thm:bayes-risk:min:iid:J=2}]{GKN2025main}}{[4, Theorem 4]} }
\label{proof:bayes-risk:min:iid:J=2}

Let $\theta\in\Theta^{\mathrm{ind}}$. Thanks to \cite[Equation~(\ref*{main-eq:risk-clustering:classif})]{GKN2025main}, the Bayes clusterer $g^{\star}_{\theta}$ can be defined as the partition $g^{\star}_{\theta} = \pi_n\circ \Tilde{h}_{\theta}$ where $\Tilde{h}_{\theta} = \left(\Tilde{h}_{\theta, i}\right)_{i\in[n]}$ is the classifier minimizing:
\begin{equation}
    h = \left(h_{i}\right)_{1\leq i \leq n} \mapsto%
\EE_{\theta}\Bigg[\min_{\tau \in \mathcal{S}_{2}}\frac{1}{n}\sum_{i=1}^n\1_{h_i(Y_{1:n}) \ne \tau(X_i)} \Bigg].
\end{equation}
Let $Y_{1:n}$ be $n$ i.i.d. observations of the mixture with parameters $\theta$. We have $g^{\star}_{\theta}(Y_{1:n}) = \pi_n \circ \Tilde{h}_{\theta}(Y_{1:n})$ a.e, where
\begin{equation*}
    \Tilde{h}_{\theta}(Y_{1:n})\in\argmin_{h = (h_i)_{i\in[n]}}\EE_{\theta}\Bigg[\min_{\tau \in \mathcal{S}_{2}}\frac{1}{n}\sum_{i=1}^n\1_{h_i(Y_{1:n}) \ne \tau(X_i)} \bigg| Y_{1:n} \Bigg].
\end{equation*}
Given that:
\begin{equation*}
    \begin{split}
        \EE_{\theta}\Bigg[\min_{\tau \in \mathcal{S}_{2}}\frac{1}{n}\sum_{i=1}^n\1_{h_i(Y_{1:n}) \ne \tau(X_i)} \bigg| Y_{1:n} \Bigg] &= \EE_{\theta}\Bigg[\min\left(\frac{1}{n}\sum_{i=1}^n\1_{h_i(Y_{1:n}) \ne X_i}, 1 - \frac{1}{n}\sum_{i=1}^n\1_{h_i(Y_{1:n}) \ne X_i}\right) \bigg| Y_{1:n} \Bigg]\\
        &= \frac{1}{2} - \frac{1}{n}\EE_{\theta}\Bigg[\left|\sum_{i=1}^n\1_{h_i(Y_{1:n}) \ne X_i} - \frac{n}{2}\right| \bigg| Y_{1:n} \Bigg]
    \end{split}
\end{equation*}
one gets $\Tilde{h}_{\theta}(Y_{1:n})\in\argmax_{h = (h_i)_{i\in[n]}}\EE_{\theta}\Bigg[\left|\sum_{i=1}^n\1_{h_i(Y_{1:n}) \ne X_i} - \frac{n}{2}\right| \bigg| Y_{1:n} \Bigg]$. We now consider the following lemma. Its proof is due to Ziv Scully and can be found in \cite{Ziv}. We will detail its proof in Section~\ref{sec:proof:max_dev} for the sake of completeness.
\begin{lemma}\label{lem:max_dev:Scully}
    Let $(Z_i)_{i\in[n]}$ be a sequence of independent Bernoulli random variables such that $Z_i\sim\mathcal{B}(p_i)$ where $p_i\in\left\{\alpha_i, 1-\alpha_i\right\}$. Then the maximum $\max_{\left(p_i\right)_{i\in[n]}}\EE\left[\left|\sum_{i\in[n]}Z_i - \frac{n}{2}\right|\right]$ is reached at $\left(p_i\right)_{i\in[n]} = \left(\alpha_i\wedge (1-\alpha_i)\right)_{i\in[n]}$ and $\left(p_i\right)_{i\in[n]} = \left(\alpha_i\vee (1-\alpha_i)\right)_{i\in[n]}$.
\end{lemma}
We apply this lemma to the random variables $\left(\1_{h_i(Y_{1:n}) \ne X_i}\right)_{i\in [n]}$ which are independent conditionally to $Y_{1:n}$ and ensure :
\begin{equation*}
    \1_{h_i(Y_{1:n}) \ne X_i} \mid Y_{1:n}\sim \mathcal{B}(p_i(Y_{i}))
\end{equation*}
such that $p_i(Y_{i})\in\left\{\PP_{\theta}\left(X_i = 1 \mid Y_i\right), \PP_{\theta}\left(X_i = 2 \mid Y_i\right)\right\}$. Two cases occur:
\begin{itemize}
    \item $\left(\forall i\in[n]\right)$ $\PP_{\theta}\left(X_i = 1 \mid Y_i\right)\neq 1/2$, then the Bayes classifier is unique and Lemma \ref{lem:max_dev:Scully} allows us to conclude that $\Tilde{h}_{\theta}(Y_{1:n}) = \left(\argmax_{x=1,2}\PP_{\theta}\left(X_i = x \mid Y_i\right)\right)_{i\in[n]} = h^{\star}_{\theta}\left(Y_{1:n}\right)$.
Consequently:
\begin{equation*}
    g^{\star}_{\theta}(Y_{1:n}) = \pi_n \circ h^{\star}_{\theta}(Y_{1:n}).
\end{equation*}
    \item $\left(\exists i\in[n]\right)$ $\PP_{\theta}\left(X_i = 1 \mid Y_i\right) = 1/2$, then the same argument yields
\begin{equation*}
    g^{\star}_{\theta}(Y_{1:n}) = \pi_n \circ h^{\star}_{\theta}(Y_{1:n}).
\end{equation*}
where we abuse the notation of $h^{\star}_{\theta}(Y_{1:n})$ to refer not only to a specific Bayes classifier but to the set of all the Bayes classifiers since it is not unique. The same is done for the Bayes clusterer $g^{\star}_{\theta}(Y_{1:n})$. The permutation $\pi_n$ is then applied to the set of all Bayes classifiers.
\end{itemize}

\section{Proof of \texorpdfstring{\cite[Theorem~\ref*{main-thm:gap:classif:cluster}]{GKN2025main}}{[4, Theorem 5]}}
\label{proof:thm:gap:classif:cluster}
 
We first prove the upper bound. Let define $h^{\star}_{\theta, i} (Y_i) = \argmax_{a\in\left\{1, 2\right\}}\PP_{\theta}\left(X_i = a\mid Y_i\right)$. Let $Z_n = \sum_{i=1}^{n}\1_{X_i \neq h^{\star}_{\theta, i} (Y_i)}$. One gets:
\begin{equation*}
    \begin{split}
        \inf_{h \in \mathcal{H}_n}\ClassifRisk(\theta,h) - \inf_{g \in \mathcal{G}_n}\ClusterRisk(\theta,g)&= \EE_{\theta}\left[\left(\frac{2}{n}Z_n -1\right)\1_{\frac{2}{n}Z_n > 1}\right]\\
        &= \int_{0}^{1}\PP_{\theta}\left(\left(\frac{2}{n}Z_n -1\right)\1_{\frac{2}{n}Z_n > 1}> x\right) dx.
    \end{split}
  \end{equation*}
  Namely,
  \begin{equation}
    \label{eq:gap:excact}
    \inf_{h \in \mathcal{H}_n}\ClassifRisk(\theta,h) - \inf_{g \in \mathcal{G}_n}\ClusterRisk(\theta,g)
    = \int_{0}^{1}\PP_{\theta}\left(Z_n >  \frac{n}{2}\left(x + 1\right)\right) dx%
    \eqqcolon J_n.
  \end{equation}
Chernoff bound yields:
\begin{equation*}
    \begin{split}
        \PP_{\theta}\left(Z_n > \frac{n}{2}\left(x + 1\right)\right)&\leq \inf_{\lambda}\left\{\exp\left(-\frac{n\lambda}{2}(x+1)\right)\EE_{\theta}\left[\exp\left(\lambda Z_n\right)\right]\right\}dx\\
        &\leq\left(\left(\frac{1+x}{1 - 2\varepsilon_{n, \theta}}\right)^{\frac{1+x}{2}}\left(\frac{1-x}{1 + 2\varepsilon_{n, \theta}}\right)^{\frac{1-x}{2}}\right)^{-n}.
    \end{split}
\end{equation*}
Let $g(t) = \frac{1+t}{2}\log\left(\frac{1+t}{1 - 2\varepsilon_{n, \theta}}\right) + \frac{1-t}{2}\log\left(\frac{1-t}{1 + 2\varepsilon_{n, \theta}}\right)$. Then, $g^{\prime}(t) = \frac{1}{2}\log\left(\frac{1+t}{1-t}\right) + \frac{1}{2}\log\left(\frac{1 + 2\varepsilon_{n, \theta}}{1 - 2\varepsilon_{n, \theta}}\right)$ and $g^{\prime\prime}(t) = \frac{1}{1 - t^{2}}$. Deduce that $g(t) \geq g(0) + \max(g'(0)t, \frac{t^2}{2} )$ for all $t \in (0,1)$. Then,
\begin{align*}
  J_n%
  &\leq e^{-ng(0)}\int_0^1e^{-n \max(g'(0)t, \frac{t^2}{2})}\intd t\\
  &\leq e^{-ng(0)}\min\Bigg(\int_0^1e^{-ng'(0)t}\intd t,\, \int_0^1e^{-nt^2/2}\intd t \Bigg)\\
  &\leq \min\Big(\frac{e^{-ng(0)}}{ng'(0)},\, \sqrt{\frac{\pi}{2n}} \Big)
\end{align*}

We now derive the more challenging lower bound. We assume throughout that $n\geq 100$. We also assume that  $\inf_{h\in \mathcal{H}_n}\ClassifRisk(\theta,h) > 0$ otherwise the lower bound is zero and holds trivially. Suppose first that $0 \leq \varepsilon_{n, \theta} \leq \frac{1}{10\sqrt{n}}$. Then,

\begin{align*}
  J_n%
  &\geq \int_0^{\frac{1}{10\sqrt{n}}}P\Big(S_n >  \frac{n(x+1)}{2}\Big)\intd x\\%
  &\geq \frac{1}{10\sqrt{n}}P\Big(S_n >  \frac{n}{2} + \frac{\sqrt{n}}{20}\Big)\\
  &= \frac{1}{10\sqrt{n}} P\Big( \frac{S_n - n(\frac{1}{2} - \varepsilon_{n, \theta})}{\sqrt{n(1 - \varepsilon_{n, \theta}^2)/4})} >   \frac{\frac{\sqrt{n}}{20} +  n \varepsilon_{n, \theta}}{\sqrt{n(1 - \varepsilon_{n, \theta}^2)/4}}\Big)\\
  &\geq \frac{1}{10 \sqrt{n}} P\Big( \frac{S_n - n(\frac{1}{2} - \varepsilon_{n, \theta})}{\sqrt{n(1 - \varepsilon_{n, \theta}^2)/4}} > \sqrt{\frac{10}{111}} \Big)
\end{align*}
because $n \geq 100$ and $\varepsilon_{n, \theta} \leq \frac{1}{10\sqrt{n}} \leq \frac{1}{100}$. By the theorem of Berry and Esseen,
\begin{align*}
  J_n%
  &\geq \frac{1}{\sqrt{n}}\Bigg(1 - \Phi\big(\sqrt{\frac{10}{111}}\big)%
  - \frac{0.4748}{\sqrt{n(1-\varepsilon_{n, \theta}^2)/4}}\Bigg)%
  \geq \frac{0.2870}{\sqrt{n}}
\end{align*}
since $n\geq 100$. Finally, because $\varepsilon_{n, \theta} \leq \frac{1}{10\sqrt{n}}$ it must be that
\begin{align*}
  \frac{\exp\big(-ng(0)\big[1 + \frac{6.8}{1\vee 10\sqrt{n\varepsilon_{n, \theta}^2}}\big] \big) }{ng'(0)}%
  &\geq\frac{1}{\sqrt{n}} \frac{\exp\big(-\frac{1}{2}\log(1 - 0.04/n)[1 + 6.8] \big)}{\frac{1}{2\sqrt{n}}\log(\frac{1+0.2/\sqrt{n}}{1-0.2/\sqrt{n}} )}\\
  &\geq \frac{500}{\sqrt{n}}.
\end{align*}
Deduce that for a universal constant $B > 0$
\begin{equation}
  \label{eq:mainlb}
  J_n \geq B\min\Bigg(\frac{\exp\big(-ng(0)\big[1 + \frac{6.8}{1\vee 10\sqrt{n\varepsilon_{n, \theta}^2}}\big] \big) }{ng'(0)}  ,\frac{1}{\sqrt{n}} \Bigg).
\end{equation}
Now suppose $\frac{1}{10\sqrt{n}} < \varepsilon_{n, \theta} < \frac{1}{2}$. We first lower bound,
\begin{align*}
  J_n 
  &\geq \int_0^{\frac{1}{10\sqrt{n}}}P\Big( S_n > \frac{n(x+1)}{2}\Big)\intd x.
\end{align*}
We lower bound the probability $P(S_n > \frac{n}{2}(x+1))$ using Cramér's technique. In the next $0 \leq x \leq \frac{1}{10\sqrt{n}}$. Then for all $\lambda > 0$ and all $0 < \delta < \frac{n}{2} - \frac{\sqrt{n}}{20}$ (which guarantees that $\frac{n(1+x)}{2} + \delta < n$] we have
\begin{align*}
  P\Big(S_n > \frac{n}{2}(x+1)\Big)%
  &= \sum_{\frac{n(x+1)}{2}<y\leq n} \binom{n}{y}r^y(1-r)^{n-y}\\%
  &\geq \sum_{\frac{n(x+1)}{2} < y < \frac{n(x+1)}{2}+ \delta} \binom{n}{y}e^{-\lambda y + n \psi_r(\lambda)} \frac{(re^{\lambda})^y(1-r)^{n-y}}{\exp(n \psi_r(\lambda))}w
\end{align*}
where $\psi_r(\lambda) = \log(1 - r +re^{\lambda})$ is the cumulant generating function of the Bernoulli distribution with parameter $r$; where $r \equiv \inf_{h\in \mathcal{H}_n}\ClassifRisk(\theta,h)$ for simplicity. Here one notice that $y \mapsto \binom{n}{y}\frac{(pe^{\lambda})^y(1-p)^{n-y}}{\exp(n \psi_r(\lambda))}$ is the density of the Binomial distribution with parameters $(n,q_{\lambda}) $ where $q_{\lambda} = \frac{r e^{\lambda}/(1-r)}{1 + re^{\lambda}/(1-r)}$. Letting $\tilde{S}_n \sim \mathrm{Binomial}(n,q_{\lambda})$, it is seen that
\begin{equation*}
  P\Big(S_n > \frac{n}{2}(x+1)\Big)%
  \geq e^{-\lambda( \frac{n(x+1)}{2} + \delta)  + n\psi_r(\lambda)}P\Bigg(\frac{n}{2}(x+1) < \tilde{S}_n < \frac{n}{2}(x+1)+\delta\Bigg).
\end{equation*}
Now we make the choice that $q_{\lambda} = \frac{1+x}{2} + \delta/n$, which corresponds to  $\lambda = -\log\big(1 - \frac{1+x + 2\delta/n}{2}\big) + \log\big(\frac{1-r}{r} \big)$. Observe that $\lambda$ is well-defined and positive since by assumption $0 < r \leq \frac{1}{2}$ and $1+x + 2\delta/n < 2$; this also guarantees that $0 < q_{\lambda} < 1$. Then,
\begin{align*}
  P\Big(S_n > \frac{n}{2}(x+1)\Big)%
  &\geq e^{-n I_r(\frac{1+x}{2} + \frac{\delta}{n} )}P\Big(\frac{n(1+x)}{2} < \tilde{S}_n < \frac{n(1+x)}{2}+ \delta \Big)w\\
  &= e^{-n g(x + 2\delta/n)}P\Big( -\frac{\delta}{\sqrt{nq_{\lambda}(1-q_{\lambda})}} < \frac{\tilde{S}_n - nq_{\lambda} }{\sqrt{nq_{\lambda}(1-q_{\lambda})}} < 0 \Big)%
\end{align*}
By the theorem of Berry and Esseen, we can conclude that
\begin{align*}
  P\Big(S_n > \frac{n}{2}(x+1)\Big)
  &\geq e^{-n g(x + 2\delta/n)}\Bigg(\Phi(0) - \Phi\Big(-\frac{\delta}{\sqrt{nq_{\lambda}(1-q_{\lambda})}} \Big)%
    - 2\frac{0.4748}{\sqrt{nq_{\lambda}(1-q_{\lambda})}} \Bigg)\\
  &\geq e^{-n g(x + 2\delta/n)}\Bigg(\frac{1}{2} -  \Phi\Big(-\frac{2\delta}{\sqrt{n}} \Big)%
    -  \frac{1.8992}{\sqrt{n(1-\kappa_x^2)}} \Bigg)
\end{align*}
where $\kappa_x \coloneqq x + \frac{2\delta}{n}$. We choose $\delta = - \frac{\sqrt{n}}{2}\Phi^{-1}(1/4)$. This implies that $\kappa_x \leq \frac{0.1 - \Phi^{-1}(1/4)}{\sqrt{n}} \leq \frac{0.7745}{\sqrt{n}}$ and $\Phi(-2\delta/\sqrt{n}) = \frac{1}{4}$. Consequently for $n \geq 100$,
\begin{equation*}
  P\Big(S_n > \frac{n}{2}(x+1)\Big)
  \geq 0.0595\cdot e^{-n g(x + 2\delta/n)},\ \mathrm{and},\quad%
  J_n \geq 0.0595\int_0^{\frac{1}{10\sqrt{n}}}e^{-ng(x + 2\delta/n)}\intd x.
\end{equation*}
A Taylor expansion of $g$ near zero yields the existence of $u \in (0,x+2\delta/n)$ such that $g(x + 2\delta/n) = g(0) + g'(0)(x + 2\delta/n) + g''(u)(x + 2\delta/n)^2/2$. But when $x \in (0,\frac{1}{10\sqrt{n}})$ it must be that $0 \leq x+2\delta/n \leq \frac{0.1 - \Phi^{-1}(1/4)}{\sqrt{n}} \leq \frac{0.7745}{\sqrt{n}}$ and thus $g''(u) = \frac{1}{1-u^2} \leq 1.0061$ when $n\geq 100$. Hence,
\begin{align*}
  J_n%
  &\geq 0.0595\cdot e^{-ng(0) - 2\delta g'(0) - \frac{1.0061\cdot 0.7745^2}{2}} \int_0^{\frac{1}{10\sqrt{n}}}e^{-ng'(0)x}\intd x\\
  &= 0.0440\cdot\Big(1 - e^{-\frac{g'(0)\sqrt{n}}{10}}\Big)\frac{e^{-ng(0)[1 + \frac{2\delta g'(0)}{n g(0)}] }}{n g'(0)}.
\end{align*}
Here recall that $g(0) = -\frac{1}{2}\log(1 - 4\varepsilon_{n, \theta}^2)$ and $g'(0) = \frac{1}{2}\log\big(\frac{1 + 2\varepsilon_{n, \theta}}{1-2\varepsilon_{n, \theta}} \big)$. It follows that the function $\varepsilon_{n, \theta} \mapsto g'(0) - \frac{g(0)}{\varepsilon_{n, \theta}}$ admits the derivative $\varepsilon_{n, \theta} \mapsto \frac{-(1 -4\varepsilon_{n, \theta}^2)\log(1-4\varepsilon_{n, \theta}^2) - 4\varepsilon_{n, \theta}^2}{\varepsilon_{n, \theta}^2(1-4\varepsilon_{n, \theta}^2)}$ which is negative for $\varepsilon_{n, \theta} > 0$. Deduce that $g'(0) - \frac{g(0)}{\varepsilon_{n, \theta}} \leq 0$, or equivalently $\frac{g'(0)}{g(0)} \leq \frac{1}{\varepsilon_{n, \theta}}$. Therefore,
\begin{align*}
  \frac{2\delta g'(0)}{n g(0)}%
  &\leq \frac{-\Phi^{-1}(1/4)}{\sqrt{n}\varepsilon_{n, \theta}}%
    \leq \frac{6.8}{\max(1,10 \sqrt{n\varepsilon_{n, \theta}^2}) }.
\end{align*}
Similarly since $\varepsilon_{n, \theta} \mapsto g'(0)$ is monotonically increasing, $\frac{g'(0)\sqrt{n}}{10} \geq \frac{\frac{1}{2}\log(\frac{1+ 0.4/\sqrt{n}}{1-0.4/\sqrt{n}})\sqrt{n}  }{10} \geq 0.04$. Hence,
\begin{equation*}
  J_n \geq%
  0.0017\cdot \frac{\exp\big(-ng(0)\big[1 + \frac{6.8}{1\vee 10\sqrt{n\varepsilon_{n, \theta}^2}}\big] \big) }{ng'(0)}
\end{equation*}
Finally, since $\varepsilon_{n, \theta} > \frac{1}{10\sqrt{n}}$ the above computations show that $n g'(0) \geq 0.4 \sqrt{n}$. Therefore there is a universal constant $B > 0$ such that Equation~\eqref{eq:mainlb} is also satisfied when $\frac{1}{10\sqrt{n}} < \varepsilon_{n, \theta} < \frac{1}{2}$.

\section{Proof of \texorpdfstring{\cite[Theorem~\ref*{main-thm:bayes-risk:min:iid:J>2}]{GKN2025main}}{[4, Theorem 6]}}
\label{proof:bayes-risk:min:iid:J>2}

Given two partitions $A$ and $B$, we recall the clustering loss defined in \cite[Equation~(\ref*{main-def:min_overlap})]{GKN2025main}:
\begin{equation*}
    \ell\left(A, B\right) = 
    1 - \frac{1}{n}\sup_{\substack{M \subseteq \mathcal{E}(A, B)\\M\, \textrm{is\, a matching}} 
    } \sum_{ \{C,C' \} \in M }\card{(C \cap C' )}
\end{equation*}
We define 
\[
\left(\forall i \in [n]\right)\left(\forall k\in \mathbb{X}\right) \quad \alpha^{(Y_i)}_{k} = \frac{\nu_{k}f_{k}(Y_i)}{\sum_{j=1}^{J} \nu_{j}f_{j}(Y_i)}
\]
and $h^{\star}_{\theta}(Y_{1:n}) = \left(h^{\star}_{\theta, i}(Y_{i})\right)_{i \in [n]}$ where $ \left(\forall i\in[n]\right)~ h^{\star}_{\theta, i}(Y_i) = \argmax_{k\in \mathbb{X}} \alpha^{(Y_i)}_{k}$. \\
Consider the event :
\begin{equation*}
    A_n = \bigcup_{j=1}^{J}\bigcap_{i=1}^{n}\left\{ \max_{k \neq j} \nu_k f_k (Y_i) < \nu_jf_j (Y_i)\right\}.
\end{equation*}
Then,
$$
A_n\subset \left\{\pi_{n}\left(\left(h^{\star}_{\theta, i}(Y_i)\right)_{i \in [n]}\right) = \pi_{n}\left(\left(1, \dots , 1\right)\right)\right\}.
$$
Let $\theta \in \Theta^{\mathrm{ind}}$, such that
        \begin{equation*}
            \PP_{\theta}\left(\bigcup_{j=1}^{J}~ \left\{0 < \max_{l \neq j} \nu_l f_l(Y) < \nu_j f_j (Y) \leq \sum_{l\neq j}\nu_l f_l (Y)\right\} \right) > 0.
        \end{equation*}
Then $\PP_{\theta}\left(A_n\right) >0$.
Since:
\begin{equation*}
    \begin{split}
        &A_n\cap\left\{\EE_{\theta}\Bigg[\ell(\pi_n(X_{1:n}), \pi_n(\left(1, \dots, 1)\right))\bigg| Y_{1:n}\Bigg] > \EE_{\theta}\Bigg[\ell(\pi_n(X_{1:n}), \pi_n(\left(1, \dots, 1, 2)\right))\bigg| Y_{1:n}\Bigg]\right\}\\
        &\subset \left\{g^{\star}_{\theta}(Y_{1:n}) \neq \pi_n\circ h^{\star}_{\theta}(Y_{1:n})\right\},
    \end{split}
\end{equation*}
it suffices then to show that
    \begin{equation*}
        \begin{split}
            \PP_{\theta}\left(\EE_{\theta}\Bigg[\ell(\pi_n(X_{1:n}), \pi_n(\left(1, \dots, 1)\right))\bigg| Y_{1:n}\Bigg] > \EE_{\theta}\Bigg[\ell(\pi_n(X_{1:n}), \pi_n(\left(1, \dots, 1, 2\right)))\bigg| Y_{1:n}\Bigg]\bigg| A_n\right) > 0.
        \end{split}
    \end{equation*}
    Let:
    \begin{equation*}
        \begin{split}
            k^{(1)}_{n}(x_{1:n}) &= \argmax_{i\in\mathbb{X}}\sum_{k=1}^{n} \1_{x_k = i}\\
            N^{(1)}_{n}(x_{1:n}) &= \max_{i\in\mathbb{X}}\sum_{k=1}^{n}\1_{x_k = i}\\
            N^{(2)}_{n}(x_{1:n}) &= \max_{i\neq k^{(1)}_{n}(x_{1:n})}\sum_{k=1}^{n} \1_{x_k = i}\\
            N_{n,i}(x_{1:n}) &= \sum_{k=1}^{n}\1_{x_k = i}
        \end{split}
    \end{equation*}
    First, note that for $x_{1:n}\in \mathbb{X}^{n}$:
    \begin{equation*}
        \begin{split}
            \ell(\pi_{n}\left(x_{1:n}\right), \pi_{n}\left(1, \dots, 1, 2\right)) &= \left\{
    \begin{array}{ll}
        n - N^{(1)}_{n}(x_{1:n}) -1 & \mbox{if } x_n \neq k^{(1)}_{n} (x_{1:n}) \\
        n - N^{(1)}_{n}(x_{1:n}) +1 & \mbox{if } N^{(2)}_{n}(x_{1:n}) < N^{(1)}_{n}(x_{1:n}) -1,  x_n = k^{(1)}_{n} (x_{1:n}) \\
        n - N^{(1)}_{n}(x_{1:n}) & \mbox{if } N^{(2)}_{n}(x_{1:n}) = N^{(1)}_{n}(x_{1:n}) - 1, x_n = k^{(1)}_{n} (x_{1:n}) \\
        n - N^{(1)}_{n}(x_{1:n}) -1 & \mbox{if } N^{(2)}_{n}(x_{1:n}) = N^{(1)}_{n}(x_{1:n}), x_n = k^{(1)}_{n} (x_{1:n}) \\
    \end{array}
\right.
        \end{split}
    \end{equation*}
    The inequality
    \begin{equation}\label{ineq:1}
        \EE_{\theta}\big[\ell(\pi_n(X_{1:n}), \pi_n(\left(1, \dots, 1)\right))\big| Y_{1:n}\big] > \EE_{\theta}\big[\ell(\pi_n(X_{1:n}), \pi_n(\left(1, \dots, 1, 2\right)))\big| Y_{1:n}\big]
    \end{equation}
    is equivalent to: 
    \begin{equation*}
        \sum_{x_{1:n}\in\mathbb{X}^{n}}\ell(\pi_n(x_{1:n}), \pi_n(\left(1, ...,1\right)))\prod_{i=1}^{n}\alpha^{(Y_i)}_{x_i}> 
        \sum_{x_{1:n}\in\mathbb{X}^{n}}\ell(\pi_n(x_{1:n}), \pi_n(\left(1, .., 1, 2\right)))\prod_{i=1}^{n}\alpha^{(Y_i)}_{x_i}
    \end{equation*}
    Conditionally to $Y_{1:n}$, $X_1, ..., X_n$ are independent multinomial random variables such that $X_i\bigg| Y_i \sim \left(\alpha^{(Y_i)}_k\right)_{k\in\mathbb{X}}$. Using the expression of $\ell(\pi_n(x_{1:n}), \pi_n(\left(1, \dots, 1, 2\right)))$ and the fact that $\ell(\pi_n(x_{1:n}), \pi_n(\left(1, \dots, 1\right))) = n - N^{(1)}_{n}(x_{1:n})$, one obtains:
    \begin{equation*}
        \begin{split}
            \eqref{ineq:1}&\Leftrightarrow\sum_{\substack{x_{1:n}\in\mathbb{X}^{n}\\ x_n \neq k^{(1)}_{n}(x_{1:n}) \text{ or }\\  x_n = k^{(1)}_{n}(x_{1:n}) \text{ and } N^{(2)}_{n}(x_{1:n}) = N^{(1)}_{n}(x_{1:n})}} \prod_{i=1}^{n}\alpha^{(Y_i)}_{x_i} > \sum_{\substack{x_{1:n}\in\mathbb{X}^{n}\\ x_n = k^{(1)}_{n}(x_{1:n}) \text{ and } N^{(2)}_{n}(x_{1:n}) < N^{(1)}_{n}(x_{1:n})-1}} \prod_{i=1}^{n}\alpha^{(Y_i)}_{x_i}\\
            &\Leftrightarrow\PP_{\theta}\left(X_n \neq k^{(1)}_{n}(X_{1:n})\bigg| Y_{1:n}\right) + \PP_{\theta}\left(X_n = k^{(1)}_{n}(X_{1:n}), N^{(2)}_{n}(X_{1:n}) = N^{(1)}_{n}(X_{1:n})\bigg| Y_{1:n}\right)\\
            &> \PP_{\theta}\left(X_n = k^{(1)}_{n}(X_{1:n}), N^{(2)}_{n}(X_{1:n}) < N^{(1)}_{n}(X_{1:n})-1\bigg| Y_{1:n}\right)\\
            &\Leftrightarrow\PP_{\theta}\left(X_n \neq k^{(1)}_{n}(X_{1:n})\bigg| Y_{1:n}\right) - \PP_{\theta}\left(X_n = k^{(1)}_{n}(X_{1:n}), N^{(2)}_{n}(X_{1:n}) < N^{(1)}_{n}(X_{1:n})-1\bigg| Y_{1:n}\right)\\
            &+ \PP_{\theta}\left(X_n = k^{(1)}_{n}(X_{1:n}), N^{(2)}_{n}(X_{1:n}) = N^{(1)}_{n}(X_{1:n})\bigg| Y_{1:n}\right) > 0\\
        \end{split}
    \end{equation*}
    By marginalization over the possible values of $X_n$, one gets:
    \begin{equation*}
        \begin{split}
        \eqref{ineq:1}&\Leftrightarrow\sum_{j=1}^{J}\biggl[\PP_{\theta}\left(X_n \neq j, N_{n, j}(X_{1:n}) \geq \max_{k\neq j} N_{n, k}(X_{1:n})\bigg| Y_{1:n}\right)\\
        &- \PP_{\theta}\left(X_n = j, N_{n, j}(X_{1:n}) > \max_{k\neq j}N_{n, k}(X_{1:n}) +1\bigg| Y_{1:n}\right)\\
        &+ \PP_{\theta}\left(X_n = j, N_{n, j}(X_{1:n}) = \max_{k\neq j}N_{n, k}(X_{1:n})\bigg| Y_{1:n}\right)\biggr] > 0\\
        \Leftrightarrow&\sum_{j=1}^{J}\sum_{l\neq j}\bigg[ \alpha^{(Y_n)}_l \PP_{\theta}\left(N_{n-1, j}(X_{1:n-1}) \geq \max_{k\neq j, l} N_{n-1, k}(X_{1:n-1})\vee (N_{n-1, l}(X_{1:n-1}) + 1)\bigg| Y_{1:n-1}\right) \\
        &- \frac{\alpha^{(Y_n)}_l \alpha^{(Y_n)}_j}{1-\alpha^{(Y_n)}_j}\PP_{\theta}\left(N_{n-1, j}(X_{1:n-1}) > \max_{k\neq j}N_{n-1, k}(X_{1:n-1})\bigg| Y_{1:n-1}\right)\\
        & \frac{\alpha^{(Y_n)}_l \alpha^{(Y_n)}_j}{1-\alpha^{(Y_n)}_j}\PP_{\theta}\left(N_{n-1, j}(X_{1:n-1}) = \max_{k\neq j}N_{n-1, k}(X_{1:n-1})-1\bigg| Y_{1:n-1}\right)\bigg] > 0\\
        \Leftrightarrow&\sum_{j=1}^{J}\sum_{l\neq j}\bigg[ \alpha^{(Y_n)}_l \left\{\PP_{\theta}\left(N_{n-1, j}(X_{1:n-1}) \geq \max_{k\neq j, l} N_{n-1, k}(X_{1:n-1})\vee (N_{n-1, l}(X_{1:n-1})+1)\bigg| Y_{1:n-1}\right)\right.\\
        &- \left.\frac{ \alpha^{(Y_n)}_j}{1-\alpha^{(Y_n)}_j}\PP_{\theta}\left(N_{n-1, j}(X_{1:n-1}) > \max_{k\neq j}N_{n-1, k}(X_{1:n-1})\bigg| Y_{1:n-1}\right)\right\}\bigg]\\
        & + \sum_{j=1}^{J}\alpha^{(Y_n)}_{j}\PP_{\theta}\left(N_{n-1, j}(X_{1:n-1}) = \max_{k\neq j}N_{n-1, k}(X_{1:n-1})-1\bigg| Y_{1:n-1}\right) > 0.
        \end{split}
    \end{equation*}
    On the one hand, since:
    \begin{equation*}
        \begin{split}
            &\left\{N_{n-1, j}(X_{1:n-1}) > \max_{k\neq j}N_{n-1, k}(X_{1:n-1})\right\}\\
            &\subset\left\{N_{n-1, j}(X_{1:n-1}) \geq \max_{k\neq j, l} N_{n-1, k}(X_{1:n-1})\vee (N_{n-1, l}(X_{1:n-1}) + 1)\right\}
        \end{split}
    \end{equation*}
    and $\alpha^{(Y_n)}_{k} \leq \frac{1}{2}\iff \frac{\alpha^{(Y_n)}_{k}}{1 - \alpha^{(Y_n)}_{k}} \leq 1$, one gets:
    \begin{equation*}
        \begin{split}
            &\bigcap_{k\in\mathbb{X}}\left\{\alpha^{(Y_n)}_{k} \leq \frac{1}{2}\right\}\\
            &\subset  \left\{\sum_{j=1}^{J}\sum_{l\neq j}\alpha^{(Y_n)}_l\bigg[ \PP_{\theta}\left(N_{n-1, j}(X_{1:n-1}) \geq \max_{k\neq j, l} N_{n-1, k}(X_{1:n-1})\vee (N_{n-1, l}(X_{1:n-1})+1)\bigg| Y_{1:n-1}\right)\right.\\
            &\left.- \frac{ \alpha^{(Y_n)}_j}{1-\alpha^{(Y_n)}_j}\PP_{\theta}\left(N_{n-1, j}(X_{1:n-1}) > \max_{k\neq j} N_{n-1, k}(X_{1:n-1})\bigg| Y_{1:n-1}\right)\bigg] \geq 0\right\}.
        \end{split}
    \end{equation*}
    On the other hand,
    \begin{multline*}
        \bigcup_{l\neq j\in \mathbb{X}}\bigcap_{i=1}^{n}\left\{ \alpha_l^{(Y_i)}\wedge\alpha_j^{(Y_i)} > 0\right\}\subset\\ \left\{\sum_{j=1}^{J}\alpha^{(Y_n)}_{j}\PP_{\theta}\left(N_{n-1, j}(X_{1:n-1}) = \max_{k\neq j} N_{n-1, k}(X_{1:n-1})-1\bigg| Y_{1:n-1}\right) > 0\right\}
    \end{multline*}
    because for $k\in\mathbb{X}$, $N_{n-1, k}(X_{1:n-1})$ is a sum of Bernoulli random variables. Consequently,
    \begin{equation*}
        \begin{split}
            \bigcup_{l\neq j\in \mathbb{X}}\bigcap_{i=1}^{n}\left\{\alpha_l^{(Y_i)}\wedge\alpha_j^{(Y_i)} > 0\right\}\bigcap\bigcap_{k\in\mathbb{X}}\left\{\alpha^{(Y_n)}_{k} \leq \frac{1}{2}\right\} \subset \left\{\eqref{ineq:1}\right\}
        \end{split}
    \end{equation*}
    Consequently, 
    \begin{equation*}
        \begin{split}
            &\PP_{\theta}\left(\bigcup_{l\neq j\in \mathbb{X}}\bigcap_{i=1}^{n}\left\{\alpha_l^{(Y_i)}\wedge\alpha_j^{(Y_i)} > 0\right\}\bigcap\bigcap_{k\in\mathbb{X}}\left\{\alpha^{(Y_n)}_{k} \leq \frac{1}{2}\right\} \bigg| A_n\right)\\
            & \leq \PP_{\theta}\left(\EE_{\theta}\Bigg[\ell(\pi_n(X_{1:n}), \pi_n(\left(1, \dots, 1\right)))\bigg| Y_{1:n}\Bigg] > \EE_{\theta}\Bigg[\ell(\pi_n(X_{1:n}), \pi_n (\left(1, \dots, 1, 2\right)))\bigg| Y_{1:n}\Bigg]\bigg| A_n\right)
        \end{split}
    \end{equation*}
    To conclude, we only need to prove that the conditional probability in the lower-bound is positive. Finally, 
    \begin{equation*}
        \begin{split}
            &\bigcup_{1\leq j\neq k\leq J}\bigcap_{i=1}^{n}\left\{\max_{l \neq j} \nu_l f_l (Y_i) < \nu_j f_j (Y_i) \leq \sum_{l\neq j} \nu_l f_l(Y_i), \nu_k f_k (Y_i) > 0\right\}\\
            &\subset \bigcup_{j=1}^{J}\bigcap_{i=1}^{n}\left\{ \max_{l \neq j} \nu_l f_l (Y_i) < \nu_jf_j (Y_i) \leq \frac{1}{2}\sum_{l=1}^{J}  \nu_l f_l(Y_i)\right\}\bigcap\bigcup_{1\leq l\neq j\leq J}\bigcap_{i=1}^{n}\left\{ \nu_l f_l(Y_i) > 0, \nu_j f_j (Y_i) > 0\right\}\\
            &\subset \bigcup_{j=1}^{J}\bigcap_{i=1}^{n}\left\{\max_{l \neq j} \nu_l f_l (Y_i) < \nu_jf_j (Y_i) \right\}\bigcap\bigcup_{1\leq l\neq j \leq J}\bigcap_{i=1}^{n}\left\{ \nu_{l}f_{l}(Y_i) > 0, \nu_j f_j (Y_i) > 0\right\}\bigcap\bigcap_{k = 1}^{J}\left\{\alpha^{(Y_n)}_k < \frac{1}{2}\right\}\\
            &\subset A_n \bigcap \bigcup_{1\leq l\neq j\leq J}\bigcap_{i=1}^{n}\left\{ \alpha_l^{(Y_i)}\wedge\alpha_j^{(Y_i)} > 0\right\} \bigcap\bigcap_{k = 1}^{J}\left\{\alpha^{(Y_n)}_k \leq \frac{1}{2}\right\}
        \end{split}
    \end{equation*}

Given that the observations $Y_{1:n}$ are i.i.d. following the stationary distribution $\sum_{k=1}^{J}\nu_k f_k$ and that \\
    \begin{equation*}
        \begin{split}
            &\bigcup_{j=1}^{J}\left\{0 < \max_{l \neq j} \nu_l f_l(Y) < \nu_j f_j (Y) \leq \sum_{l\neq j}\nu_l f_l (Y) \right\}\\
            &\subset \bigcup_{1\leq j\neq k\leq J}\left\{\max_{l \neq j} \nu_l f_l(Y) < \nu_j f_j (Y) \leq \sum_{l\neq j} \nu_l f_l (Y), \nu_k f_k (Y) > 0 \right\}
        \end{split}
        \end{equation*}
    and that by assumption,
    \begin{equation*}
        \PP_{\theta}\left(\bigcup_{j=1}^{J}\left\{0 < \max_{l \neq j} \nu_l f_l(Y) < \nu_j f_j (Y) \leq \sum_{l\neq j}\nu_l f_l (Y) \right\} \right) > 0,
    \end{equation*}
    the result follows.

\section{Common elements to the proof of \texorpdfstring{\cite[Theorems~\ref*{main-thm:bayes-risk:lb:iid:J>2}, \ref*{main-thm:bayes-risk:lb:hmm:J=2}, and~\ref*{main-thm:bayes-risk:lb:hmm:J>2}]{GKN2025main}}{[4, Theorems 7, 9, and 11]}}

Recall the definition of $\ArianneRisk$ in \cite[Remark~\ref*{main-rmk:arianne}]{GKN2025main}: for all $(\theta,h) \in \Theta \times \mathcal{H}_n$,
\begin{equation}
  \label{eq:recall-arianne}
  \ArianneRisk(\theta,h)%
  = \EE_{\theta}\Big[ \min_{\tau \in \mathcal{S}_{J}}\EE_{\theta}\big[ U_{n,\tau}(h)  \mid Y_{1:n}\big] \Big]
\end{equation}
where $U_{n, \tau}(h) \coloneqq \frac{1}{n}\sum_{i=1}^{n}\1_{\tau(X_{i})\neq h_{i}(Y_{1:n})}$. We also make use of the notation $\hat{p}_{\tau}(h) \coloneqq \EE_{\theta}\left[U_{n,\tau}(h) \mid Y_{1:n}\right]$. Let $\hat{\tau}_h$ denote a $Y_{1:n}$-measurable permutation satisfying:
\begin{equation*}
  \hat{p}_{\hat{\tau}_h}(h) = \EE_{\theta}\left[U_{n,\hat{\tau}_h}(h) \mid Y_{1:n}\right]%
  = \min_{\tau}\EE_{\theta}\left[U_{n,\tau}(h) \mid Y_{1:n}\right] = \min_{\tau}\hat{p}_{\tau}(h).
\end{equation*}
Instead of comparing $\ClusterRisk$ and $\ClassifRisk$, we compare $\ClusterRisk$ and $\ArianneRisk$, which is enough to obtain the result thanks to the following easy lemma.

\begin{lemma}
  \label{lem:clust-vs-arianne}
  For all $\theta \in \Theta$ and all $n\geq 1$
  \begin{equation*}
    \inf_{h\in \mathcal{H}_n}\ClassifRisk(\theta,h)%
    = \inf_{h\in \mathcal{H}_n}\ArianneRisk(\theta,h).
  \end{equation*}
\end{lemma}
\begin{proof}
  The optimal permutation $\hat{\tau}_{h}$ such that $\hat{p}_{\hat{\tau}_{h}}(h) = \min_{\tau \in \mathcal{S}_{J}}\EE_{\theta}\left[U_{n,\tau}(h) \mid Y_{1:n}\right]$ is a $Y_{1:n}$-measurable permutation valued random variable. Since any $h \in \mathcal{H}_n$ is also $Y_{1:n}$-measurable, the result is immediate.
\end{proof}

In the next we then focus on comparing $\ClusterRisk$ and $\ArianneRisk$. We let $N_j \coloneqq \sum_{i=1}^n\1_{ \{X_i = j \} }$ and $N_{(1)} \leq N_{(2)} \leq \dots \leq N_{(J)}$ denote the order statistics of $(N_1,\dots,N_J)$.

\begin{proposition}
  \label{pro:bayes-risk:lb:general}
  A generic lower bound that works for any latent model (i.i.d. or HMM or whatever). For all classifiers $h$, all $\varepsilon$, all $\eta$ and all $\theta\in\Theta$
  \begin{align*}
    \EE_{\theta}\Big[\min_{\tau}U_{n,\tau}(h) \Big]%
    &\geq \EE_{\theta}\Big[\min_{\tau}\EE_{\theta}[U_{n,\tau}(h) \mid Y_{1:n}] \Big]\\
    &\quad%
      - \EE_{\theta}\Big[\EE_{\theta}\Big[\max_{\tau}(-U_{n,\tau}(h) + \hat{p}_{\tau}(h)) \mid Y_{1:n} \Big]\1_{ \{ \hat{p}_{\hat\tau_h}(h) \geq \varepsilon\} } \Big]\\
    &\quad%
      - \EE_{\theta}\Big[ \PP_{\theta}(U_{n,\hat\tau_h}(h) > \eta \mid Y_{1:n})\1_{\{ \hat{p}_{\hat\tau_h}(h) < \varepsilon \} } \Big]\\
    &\quad%
      - \PP_{\theta}\big( N_{(1)} + N_{(2)} < 2n\eta \big).
  \end{align*}
\end{proposition}
\begin{proof}
  For $\varepsilon \in [0,1]$, we decompose
  \begin{equation}
    \label{eq:pro:bayes-risk:lb:general:1}
    \EE_{\theta}\Big[\min_{\tau}U_{n,\tau}(h) \Big]%
    = \EE_{\theta}\Big[\min_{\tau}U_{n,\tau}(h)\1_{\{ \hat{p}_{\hat\tau_h}(h) \geq \varepsilon \} } \Big]%
    + \EE_{\theta}\Big[\min_{\tau}U_{n,\tau}(h)\1_{\{ \hat{p}_{\hat\tau_h}(h) < \varepsilon \} } \Big].
  \end{equation}
  The first term in the rhs of \eqref{eq:pro:bayes-risk:lb:general:1} is handled via a small deviation principle, remarking that on the event $\{ \hat{p}_{\hat\tau(h)}(h) \geq \varepsilon \}$
  \begin{align*}
    \EE_{\theta}\Big[ \min_{\tau}U_{n,\tau}(h)  \mid Y_{1:n}\Big]%
    &=\EE_{\theta}\Big[ \min_{\tau}(\hat{p}_{\tau}(h) + U_{n,\tau}(h) - \hat{p}_{\tau}(h))  \mid Y_{1:n}\Big]\\
    &\geq \min_{\tau}\hat{p}_{\tau}(h)%
      - \EE_{\theta}\Big[\max_{\tau}(-U_{n,\tau}(h) + \hat{p}_{\tau}(h)) \mid Y_{1:n} \Big].
  \end{align*}
  The second term in \eqref{eq:pro:bayes-risk:lb:general:1} is handled via a large deviation principle. By Lemma~\ref{lem:bayes-risk:lb:largedev}, on the event $\{ \hat{p}_{\hat\tau_h}(h) < \varepsilon\}$
  \begin{align*}
    \EE_{\theta}\Big[ \min_{\tau}U_{n,\tau}(h)  \mid Y_{1:n}\Big]%
    &\geq \EE_{\theta}\Big[ \min_{\tau}U_{n,\tau}(h)\1_{\{U_{n,\hat\tau_h}(h) \leq (N_{(1)} + N_{(2)})/(2n)  \} }  \mid Y_{1:n}\Big]\\
    &= \EE_{\theta}\Big[U_{n,\hat\tau_h}(h)\1_{\{U_{n,\hat\tau_h}(h) \leq (N_{(1)} + N_{(2)})/(2n)  \} }  \mid Y_{1:n}\Big]\\
    &\geq \hat{p}_{\hat\tau_h}(h)%
      - \PP_{\theta}\Big(U_{n,\hat\tau_h}(h) > \frac{N_{(1)} + N_{(2)}}{2n}  \mid Y_{1:n}\Big).
  \end{align*}
  But,
  \begin{align*}
    \PP_{\theta}\Big(U_{n,\hat\tau_h}(h) > \frac{N_{(1)} + N_{(2)}}{2n}  \mid Y_{1:n}\Big)%
    &\leq \PP_{\theta}\big( U_{n,\hat\tau_h}(h) > \eta \mid Y_{1:n}\big)%
      + \PP_{\theta}\big( N_{(1)} + N_{(2)} < 2n\eta \mid Y_{1:n} \big).
  \end{align*}
  The generic lower bound follows.
\end{proof}

\begin{lemma}
  \label{lem:bayes-risk:lb:largedev}
  If $U_{n,\tau'}(h) \leq \frac{N_{(1)} + N_{(2)}}{2n}$ then $\min_{\tau} U_{n,\tau}(h) = U_{n,\tau'}(h)$.
\end{lemma}
\begin{proof}
  Let $U_{n,\tau'}(h) \leq \frac{N_{(1)} + N_{(2)}}{2n}$ and suppose $\min_{\tau} U_{n,\tau}(h) < U_{n,\tau'}(h)$. Then, there exists a permutation $\tau'' \ne \tau'$ such that $U_{n,\tau''}(h) < U_{n,\tau'}(h)$. But then, letting $I = \{i \in \{1,\dots,n\}\;:\; \tau'(X_i) = h_i(Y_{1:n})\}$:
  \begin{align*}
    n(U_{n,\tau'}(h) - U_{n,\tau''}(h) )%
    &= \sum_{i=1}^n\Big(\1_{ \{ \tau'(X_i) \ne h_i(Y_{1:n})\} } - \1_{ \{\tau''(X_i) \ne h_i(Y_{1:n}) \} } \Big)\\
    &= \sum_{i=1}^n\Big(\1_{ \{ \tau''(X_i) = h_i(Y_{1:n})\} } - \1_{ \{\tau'(X_i) = h_i(Y_{1:n})\} } \Big)\\
    &= - \sum_{i\in I}\1_{\{ \tau'(X_i) \ne \tau''(X_i) \}}%
      + \sum_{i\in I^c}\1_{ \{ \tau'(X_i) = h_i(Y_{1:n}) \} }\\
    &=- \sum_{i=1}^n\1_{ \{ \tau'(X_i) \ne \tau''(X_i) \} }%
      + \sum_{i\in  I^c}\Big(\1_{ \{ \tau'(X_i) = h_i(Y_{1:n}) \} } + \1_{ \{ \tau'(X_i) \ne \tau''(X_i) \} } \Big)\\
    &\leq -(N_{(1)} + N_{(2)}) + 2|I^c|\\
    &= -(N_{(1)} + N_{(2)}) + 2n U_{n,\tau'}(h)
  \end{align*}
  where we have used that since $\tau' \ne \tau''$, it must be that $\sum_{i=1}^n\1_{\tau'(X_i) \ne \tau''(X_i)} \geq N_{(1)} + N_{(2)}$. Rearranging the previous:
  \begin{align*}
    U_{n,\tau''}(h)%
    \geq \frac{N_{(1)} + N_{(2)}}{n} - U_{n,\tau'}(h)%
    \geq U_{n,\tau'}(h)
  \end{align*}
  which contradicts that $U_{n,\tau''}(h) < U_{n,\tau'}(h)$. Hence $\min_{\tau}U_{n,\tau}(h) \geq U_{n,\tau'}(h)$.
\end{proof}

\section{Proof of \texorpdfstring{\cite[Theorem \ref*{main-thm:bayes-risk:lb:iid:J>2}]{GKN2025main}}{[4, Theorem 7]} (independent scenario)}
\label{proof:bayes-risk:lb:iid}

Here we apply the result of Proposition~\ref{pro:bayes-risk:lb:general} to the i.i.d. case. 

When $J > 2$, the first trivial bound is obtained by choosing $\varepsilon = \eta = 0$. With this choice, Proposition~\ref{pro:bayes-risk:lb:general} gives for $J \geq 2$:
\begin{align*}
  \EE_{\theta}\Big[\min_{\tau}U_{n,\tau}(h) \Big]%
  &\geq \EE_{\theta}\Big[\min_{\tau}\EE_{\theta}[U_{n,\tau}(h) \mid Y_{1:n}] \Big] - \EE_{\theta}\Big[\EE_{\theta}\Big[\max_{\tau}(-U_{n,\tau}(h) + \hat{p}_{\tau}(h)) \mid Y_{1:n} \Big] \Big]\\
  &\geq \EE_{\theta}\Big[\min_{\tau}\EE_{\theta}[U_{n,\tau}(h) \mid Y_{1:n}] \Big] - \sqrt{\frac{\log(J!)}{2n}}
\end{align*}
by Lemma~\ref{lem:bayes-risk:lb:iid:1} below.\\

When $J > 2$, Lemma~\ref{lem:bayes-risk:lb:iid:3} can be used to find that
\begin{equation*}
  \PP_{\theta}\big(N_{(1)} + N_{(2)} < 2n\eta\big)%
  \leq J^2e^{-\frac{n(\beta - 2\eta)^2}{2\beta}}
\end{equation*}
and the bound is obtained by choosing $\eta = \beta/4$ and $\varepsilon = \frac{\beta}{4e}[\log(J!)/(2n) ]^{2/(n\beta)}$.

\begin{lemma}
  \label{lem:bayes-risk:lb:iid:1}
  For all $\theta\in\Theta^{\mathrm{ind}}$, $\PP_{\theta}$-almost-surely
  \begin{equation*}
    \EE_{\theta}\Big[\max_{\tau}(-U_{n,\tau}(h) + \hat{p}_{\tau}(h)) \mid Y_{1:n} \Big]%
    \leq \sqrt{\frac{\log(J!)}{2n}}.
  \end{equation*}
\end{lemma}
\begin{proof}
  For any $\lambda > 0$,
\begin{equation*}
    \begin{split}
\EE_{\theta}\left[\max_{\tau}\left\{-U_{n, \tau}(h) + \hat{p}_{\tau}(h)\right\}\right] &\leq \frac{1}{\lambda}\log\left(\EE_{\theta}\left[\exp\left(\sup_{\tau}\left\{-\lambda \left(U_{n, \tau}(h) - \hat{p}_{\tau}(h)\right)\right\}\right)\bigg| Y_{1:n}\right]\right)\\
        &=\frac{1}{\lambda}\log\left(\EE_{\theta}\left[\sup_{\tau}\exp\left(-\lambda \left(U_{n, \tau}(h) - \hat{p}_{\tau}(h)\right)\right)\bigg| Y_{1:n}\right]\right)\\
        &\leq \frac{1}{\lambda}\log\left(J! \sup_{\tau}\EE_{\theta}\left[\exp\left(-\lambda\left(U_{n, \tau}(h) - \hat{p}_{\tau}(h)\right)\right)\bigg| Y_{1:n}\right]\right)\\
        &\leq \frac{1}{\lambda}\log\left(J! \exp\left(\frac{\lambda^{2}}{8}\times n \times \left(\frac{1}{n}\right)^{2}\right)\right) \text{(Hoeffding's lemma)}\\
        &\leq \inf_{\lambda > 0} \left\{\frac{\log(J!)}{\lambda} + \frac{\lambda}{8n}\right\}\\
        &\leq \sqrt{\frac{\log(J!)}{2n}}.
    \end{split}
\end{equation*}
 Hoeffding's lemma applies because conditionally to the sequence of observations $Y_{1:n}$, the labels $X_{1:n}$ are still independent.
\end{proof}

\begin{lemma}
  \label{lem:bayes-risk:lb:iid:2}
  For all $\theta\in\Theta^{\mathrm{ind}}$, $\PP_{\theta}$-almost-surely
  \begin{equation*}
    \PP_{\theta}\big( U_{n,\hat\tau_h}(h) > \eta  \mid Y_{1:n}\big)%
    \leq  \hat{p}_{\hat\tau_h}(h) \cdot \frac{e}{\eta} \Big(\frac{e \hat{p}_{\hat\tau_h}(h)}{\eta} \Big)^{n\eta - 1}e^{-n \hat{p}_{\hat\tau_h}(h)}.
  \end{equation*}
\end{lemma}
\begin{proof}
  By Chernoff's bound (with $q_i(h) = \PP_{\theta}(h_i(Y_{1:n}) \ne \hat\tau_h(X_i) \mid Y_{1:n})$):
  \begin{align*}
    \PP_{\theta}\big(U_{n,\hat\tau_h}(h) > \eta  \mid Y_{1:n}\big)%
    &= \PP_{\theta}\Big(\sum_{i=1}^n\1_{h_i(Y_{1:n}) \ne \hat\tau_h(X_i)} > n\eta \mid Y_{1:n}\Big)\\
    &\leq \inf_{\lambda > 0}\exp\Big(-\lambda n\eta + \sum_{i=1}^n\log\Big( q_i(h)e^{\lambda} + 1 - q_i(h) \Big) \Big)\\
    &\leq \inf_{\lambda > 0}\exp\Big(-\lambda n\eta + n\hat{p}_{\hat\tau_h}(h)(e^{\lambda} -  1) \Big)\\
    &= \Big( \frac{e\hat{p}_{\hat\tau_h}(h)}{\eta} \Big)^{n\eta}e^{-n \hat{p}_{\hat\tau_h}(h) }\\
    &\leq  \hat{p}_{\hat\tau_h}(h) \cdot \frac{e}{\eta} \Big(\frac{e \hat{p}_{\hat\tau_h}(h)}{\eta} \Big)^{n\eta - 1}e^{-n \hat{p}_{\hat\tau_h}(h)}.
  \end{align*}
\end{proof}

\begin{lemma}
  \label{lem:bayes-risk:lb:iid:3}
  For $\theta\in\Theta^{\mathrm{ind}}$, let $\beta = \min_{j\ne k}(\nu_j + \nu_k)$. If $J > 2$, then
  \begin{equation*}
    \PP_{\theta}\big(N_{(1)} + N_{(2)} < 2n\eta\big)%
    \leq J^2e^{-\frac{n(\beta - 2\eta)^2}{2\beta}}.
  \end{equation*}
\end{lemma}
\begin{proof}
  If $J = 2$, remark that $N_{(1)} + N_{(2)} = n$. Now we assume that $J > 2$. It holds that
  \begin{align*}
    \PP_{\theta}\big(N_{(1)} + N_{(2)} < 2n\eta\big)%
    &=P_{\theta}\big(\exists\, j\ne k,\ N_j + N_k \leq 2n\eta\big)\\
    &\leq J^2\max_{j\ne k}\PP_{\theta}\big(N_j + N_k \leq 2n\eta\big).
  \end{align*}
  Then observe that $N_j + N_k = \sum_{i=1}^n(\1_{X_i = j} + \1_{X_i = k}) = \sum_{i=1}^n\1_{X_i \in \{j,k\} }$ whenever $j\ne k$. In other words, when $j \ne k$ the random variables $N_j + N_k$ has a Binomial distribution with parameters $(n, \nu_j + \nu_k)$ under $\PP_{\theta}$. The conclusion follows using Chernoff's bound on the Binomial distribution (recall the Binomial distribution is subGaussian on the left-tail).
\end{proof}

\section{Proof of \texorpdfstring{\cite[Theorems \ref*{main-thm:bayes-risk:lb:hmm:J=2} and \ref*{main-thm:bayes-risk:lb:hmm:J>2}]{GKN2025main}}{[4, Theorems 9 and 11]} (dependent scenario)}
\label{proof:bayes-risk:lb:hmm}

\subsection{Preliminary}

We first recall basic results for HMMs that can be found in \cite{CMT05} about the distribution of the hidden states given a set of observations. For any parameter $\theta$, any integers $k$, $i\leq j$, the distribution of $X_{k}$ given $Y_{i:j}$ under $\PP_{\theta}$ will be denoted $\phi_{\theta, k\vert i:j}(., Y_{i:j})$. For any integers $i\leq n$, we shall simplify the so-called filtering distribution $\phi_{\theta, n \vert i:n}(., Y_{i:n})$ to $\phi_{\theta, n}(., Y_{i:n})$.

Conditional on observations $Y_{i:n}$, the sequence of the hidden states is an inhomogeneous Markov chain, with transition matrices called {\it forward kernels}. 
For each $k\leq n-1$, the forward kernel is denoted $(F_{\theta, k \vert n}[Y_{k+1:n}])$ to emphasize that it only depends on $Y_{k+1:n}$. When $k\geq n$, the kernel does not depend on the observations and is equal to the transition matrix $Q$, so that $F_{\theta, k \vert n}[Y_{k+1:n}]:=Q$ for $k\geq n$.
In other words, for any $n\in\N$, for any index $i\leq n$ and $k\geq i$ and any real-valued function $f$ on $\mathbb{X}$ (understood as a vector in $\mathbb{R}^J$), 
$$
\EE_{\theta}[f(X_{k+1})\mid X_{i:k}, Y_{i:n}] = F_{\theta, k \vert n}[Y_{k+1:n}]f = \sum_{x\in \mathbb{X}} F_{\theta, k \vert n}[Y_{k+1:n}](X_{k}, x)f(x).
$$

Conditional on observations $Y_{i:n}$, the reverse time sequence of  hidden states 
is also an inhomogeneous Markov chain with transition matrices $(B_{\theta, k}[Y_{i:k}])_{k\leq n-1}$ called  {\it backward kernels}. In other words, for any $n\in\N$, $i\leq  k\leq n-1$ and any function $f$ on $\mathbb{X}$:
$$
\EE_{\theta}[f(X_{k})\mid X_{k+1:n}, Y_{i:n}] = B_{\theta, k}[Y_{i:k}]f = \sum_{x\in \mathbb{X}} B_{\theta, k}[Y_{i:k}](X_{k+1}, x)f(x).
$$
Here, the backward kernel $B_{\theta, k}[Y_{i:k}]$ depends only on the observations up to time $k$. It is given by:
\begin{equation}\label{Backward_kernel}
    B_{\theta, k}[Y_{i:k}](\Tilde{x}, x) = \frac{\phi_{\theta, k}(x,Y_{i:k})Q(x, \Tilde{x})}{\sum_{x^{\prime}\in\mathbb{X}}\phi_{\theta, k}(x^{\prime},Y_{i:k})Q(x^{\prime}, \Tilde{x})}.
\end{equation}
Note that the denominator is always positive thanks to \cite[Assumption \ref*{main-Assumption_mixing}]{GKN2025main}.

For any transition kernel $T$, we denote $\delta(T)$ is  the Dobrushin coefficient of $T$ defined by:
\begin{equation*}
    \begin{split}
        \delta(T) &= 
        \sup_{(x, x^{\prime})\in \mathbb{X}\times \mathbb{X}}\lVert T(x, .) - T(x^{\prime}, .)\rVert_{\mathrm{TV}}
    \end{split}
\end{equation*}
where $\|\cdot\|_{\mathrm{TV}}$ is the total variation norm. We recall the following two lemmas which can be found in \cite{CMT05}. To end with, notice that under \cite[Assumption \ref*{main-Assumption_mixing}]{GKN2025main}, for any subset $A$ of $\mathbb X$, 
$$
J\delta\gamma(A) \leq \sum_{x'\in A}Q(x, x') \leq J (1-(J-1)\delta)\gamma(A)
 $$
 with $\gamma$ the uniform distribution over $\mathbb X$. Using Lemma~4.3.13 in \cite{CMT05},  this leads to the following lemma.

\begin{lemma}\label{Lemma_3}
Under \cite[Assumption \ref*{main-Assumption_mixing}]{GKN2025main}, for any integers $k$ and $n$, the Dobrushin coefficient of the forward 
kernel $F_{\theta, k \vert n}$ satisfies:
\begin{equation*}
     \delta(F_{\theta, k \vert n}) \leq \begin{cases}
       \rho_{0} &\quad k < n \\ \rho_{1} &\quad k \geq n
     \end{cases}
\end{equation*}
with  $\rho_{0} = 1 - \frac{J\delta}{J(1 - (J-1)\delta)} = \frac{1 - J\delta}{1 -(J-1)\delta}$ and $\rho_{1} = 1 - J\delta$.
\end{lemma}
Using Equation \eqref{Backward_kernel} we get that, under \cite[Assumption~\ref*{main-Assumption_mixing}]{GKN2025main}, for any  (possibly non positive) integers $i\leq k$,
$$
\forall x\in\mathbb{X},\qquad B_{\theta, k}[Y_{i:k}](x, .) \geq \frac{\delta}{1 - (J-1)\delta}\phi_{\theta, k}(., Y_{i:k}),
$$
so that applying Lemma~4.3.13 in \cite{CMT05}, one gets
\begin{lemma}\label{Lemma_B}
Under \cite[Assumption~\ref*{main-Assumption_mixing}]{GKN2025main}, for any  (possibly non positive) integers $i\leq k \leq n-1$, the Dobrushin coefficient of the backward 
kernel $B_{\theta, k}[Y_{i:k}]$ satisfies:
$$
\delta(B_{\theta, k}[Y_{i:k}]) \leq 1 - \frac{\delta}{1 - (J-1)\delta} = \rho_{0}.
$$
\end{lemma}
\subsection{Proofs}
We apply the result of Proposition~\ref{pro:bayes-risk:lb:general} to the HMM case. As in the i.i.d. case, when $J=2$ it must be that $\PP_{\theta}(N_{(1)} + N_{(2)} < n) = 0$. On the one hand, using Lemma~\ref{lem:bayes-risk:lb:HMM:1} and Markov's inequality as in the independent case, one gets for $\theta\in\Theta^{\mathrm{dep}}$
\begin{multline*}
\EE_{\theta}\Big[\EE_{\theta}\Big[\max_{\tau}(-U_{n,\tau}(h) + \hat{p}_{\tau}(h)) \mid Y_{1:n} \Big]\1_{ \{ \hat{p}_{\hat\tau_h}(h) \geq \varepsilon\} } \Big]\\
\begin{aligned}
  &\leq \frac{1}{1-\rho_0}\sqrt{\frac{\log(J!)}{2n}}\PP_{\theta}\big( \hat{p}_{\hat\tau_h}(h) \geq \varepsilon \big)\\
  &\leq \frac{1}{\varepsilon (1 - \rho_0)}\sqrt{\frac{\log(J!)}{2n}}\EE_{\theta}\Big[\min_{\tau}\EE_{\theta}[U_{n,\tau}(h) \mid Y_{1:n}]\1_{ \{ \hat{p}_{\hat\tau_h}(h) \geq \varepsilon \} } \Big]
\end{aligned}
\end{multline*}
and then using Lemma~\ref{lem:bayes-risk:lb:HMM:2}, we establish that for all $\varepsilon,\eta > 0$
\begin{multline*}
\EE_{\theta}\Big[ \PP_{\theta}(U_{n,\hat\tau_h}(h) > \eta \mid Y_{1:n})\1_{\{ \hat{p}_{\hat\tau_h}(h) < \varepsilon \} } \Big]\\%
\begin{aligned}
&\leq \EE_{\theta}\Big[\hat{p}_{\hat\tau_h}(h) \cdot \frac{e}{\eta} \left(\frac{1-(J-1)\delta}{\delta}\right)^{2n}\Big(\frac{e \hat{p}_{\hat\tau_h}(h)}{\eta} \Big)^{n\eta - 1}e^{-n \hat{p}_{\hat\tau_h}(h) }\1_{ \{ \hat{p}_{\hat\tau_h}(h) < \varepsilon \} } \Big]\\
  &\leq \frac{e}{\eta}\left(\frac{1-(J-1)\delta}{\delta}\right)^{2n}\Big(\frac{e \varepsilon}{\eta}\Big)^{n\eta - 1}\EE_{\theta}\Big[\min_{\tau}\EE_{\theta}[U_{n,\tau}(h) \mid Y_{1:n}]\1_{ \{ \hat{p}_{\hat\tau_h}(h) < \varepsilon \} } \Big].
\end{aligned}
\end{multline*}
Then, the bound in this case follows by taking $\eta \to \frac{1}{2}$ (by below) and 
$\varepsilon = \frac{1}{2e}\left(\frac{\delta}{1 -\delta}\right)^{4}[\log(J!)/(2n) ]^{1/n}$.

When $J > 2$, the first trivial bound is obtained by choosing $\varepsilon = \eta = 0$. Proposition~\ref{pro:bayes-risk:lb:general} yields:
\begin{align*}
  \EE_{\theta}\Big[\min_{\tau}U_{n,\tau}(h) \Big]%
  &\geq \EE_{\theta}\Big[\min_{\tau}\EE_{\theta}(U_{n,\tau}(h) \mid Y_{1:n}) \Big]
  - \EE_{\theta}\Big[\EE_{\theta}\Big[\max_{\tau}(\hat{p}_{\tau}(h) - U_{n,\tau}(h)) \mid Y_{1:n} \Big] \Big]\\
  &\geq \EE_{\theta}\Big[\min_{\tau}\EE_{\theta}(U_{n,\tau}(h) \mid Y_{1:n}) \Big] -  \frac{1}{1-\rho_0}\sqrt{\frac{\log(J!)}{2n}}
\end{align*}
by Lemma~\ref{lem:bayes-risk:lb:HMM:1} below.  The inequality follows by taking the infimum over h on both sides.

For the remaining inequality, when $J > 2$, Lemma~\ref{lem:bayes-risk:lb:HMM:3} ensures:
$$
\PP_{\theta}\big(N_{(1)} + N_{(2)} < 2n\eta\big)%
\leq J^2e^{-2n(1 - \rho_1)^{2}(\beta - 2\eta)^{2}}.
$$
On the other hand, we have by Lemma~\ref{lem:bayes-risk:lb:HMM:1}:
$$
\EE_{\theta}\left[\max_{\tau}(\hat{p}_{\tau}(h) - U_{n, \tau}(h))\mid Y_{1:n}\right] \leq \frac{1}{1-\rho_0}\sqrt{\frac{\log(J!)}{2n}}
$$
On the event $\left\{\hat{p}_{\hat{\tau}_h}(h) < \varepsilon\right\}$, for $\varepsilon < \eta$:
    \begin{equation*}
        \begin{split}
            \PP_{\theta}\left(U_{n, \hat{\tau}_{h}}(h) > \eta\mid Y_{1:n}\right) &= \PP_{\theta}\left(U_{n, \hat{\tau}_{h}}(h) - \hat{p}_{\hat{\tau}_h}(h) > \eta - \hat{p}_{\hat{\tau}_h}(h) \mid Y_{1:n}\right)\\
            &\leq \PP_{\theta}\left(U_{n, \hat{\tau}_{h}}(h) - \hat{p}_{\hat{\tau}_h}(h) > \eta - \varepsilon \mid Y_{1:n}\right)\\
            &\leq e^{-\lambda(\eta - \varepsilon)}\EE_{\theta}\left[e^{\lambda\left\{U_{n, \hat{\tau}_{h}}(h) - \hat{p}_{\hat{\tau}_h}(h)\right\}}\mid Y_{1:n}\right]\\
            &\leq e^{-\lambda(\eta - \varepsilon)}e^{\frac{1}{8n}\left(\frac{\lambda}{1 - \rho_0}\right)^{2}}
        \end{split}
    \end{equation*}
where the last inequality is due to the argument using Marton coupling as shown in the proof of Lemma~\ref{lem:bayes-risk:lb:HMM:1}. Taking the minimum over $\lambda$, one gets
$$
\EE_{\theta}\left[\PP_{\theta}\left(U_{n, \hat{\tau}_{h}}(h) > \eta\mid Y_{1:n}\right)\1_{\hat{p}_{\hat{\tau}_h}(h) < \varepsilon}\right] \leq e^{-2n(1-\rho_0)^{2}(\eta - \varepsilon)^{2}}.
$$
Using Lemma~\ref{lem:bayes-risk:lb:HMM:3}, the final bound reads:
\begin{equation*}
    \begin{split}
        \EE_{\theta}\left[\min_{\tau}U_{n, \tau}(h)\right] &\geq \EE_{\theta}\left[\min_{\tau}\EE_{\theta}\left[U_{n, \tau}(h)\mid Y_{1:n}\right]\right]\\
        & -\frac{1}{\varepsilon(1 - \rho_0)}\sqrt{\frac{\log(J!)}{2n}}\EE_{\theta}\left[\hat{p}_{\hat{\tau}_h}(h)\1_{\hat{p}_{\hat{\tau}_h}(h) \geq \varepsilon}\right]\\
        & -e^{-2n(1-\rho_0)^{2}(\eta - \varepsilon)^{2}}-J^{2}e^{-2n(1-\rho_1)^{2}(\beta - 2\eta)^{2}}
    \end{split}
\end{equation*}
Choosing $\varepsilon = \frac{\eta}{2}$ and $\eta = \frac{2}{5}\beta$ and noting that $\rho_1 < \rho_0$ one obtains:
\begin{align*}
\EE_{\theta}\left[\min_{\tau}U_{n, \tau}(h)\right]
&\geq \left[1 - \frac{5}{\beta(1 - \rho_0)}\sqrt{\frac{\log(J!)}{2n}}\right]\EE_{\theta}\left[\min_{\tau}\EE_{\theta}\left[U_{n, \tau}(h)\mid Y_{1:n}\right]\right]\\
&\quad- (J^{2}+1)e^{-cn(1-\rho_0)^{2}\beta^{2}}.
\end{align*}
The result follows by taking the infimum over $h$.

\begin{lemma}
    \label{lem:bayes-risk:lb:HMM:1}
    Under \cite[Assumptions $\ref*{main-Assumption_mixing}$ and \ref*{main-Assumption_stat}]{GKN2025main}, for all $\theta\in\Theta^{\mathrm{dep}}$, $\PP_\theta$-almost-surely
$$
\EE_{\theta}\left[\max_{\tau}(\hat{p}_{\tau}(h) - U_{n, \tau}(h))\mid Y_{1:n}\right] \leq \frac{1}{1-\rho_0}\sqrt{\frac{\log(J!)}{2n}}
$$
where $\rho_0 = \frac{1-J\delta}{1-(J-1)\delta}$.
\end{lemma}
\begin{proof}
    Given that for any $\lambda>0$, 
$$
\EE_{\theta}\left[\max_{\tau}\left(\hat{p}_{\tau}(h)-(U_{n, \tau}(h) \right)\right] \leq \frac{1}{\lambda}\log\left(J! \max_{\tau}\EE_{\theta}\left[\exp\left(-\lambda (U_{n, \tau}(h) - \hat{p}_{\tau}(h))\right) \bigg| Y_{1:n}\right]\right)
$$
we shall exhibit an upper bound of the rhs term by applying Theorem~2.9 of \cite{Paulin15}, conditional on $Y_{1:n}$. So that for now we consider $Y = Y_{1:n}$ as fixed.
Define $f_{Y}^{h, \tau}$ for any  $x = x_{1:n}$ by:
$$
f_{Y}^{h, \tau}(x) =- \frac{1}{n}\sum_{i=1}^{n}\1_{\tau(x_{i}) \neq h_{i}(Y_{1:n})}.
$$
Then, for any $x = x_{1:n}$ and $x^{\prime} = x^{\prime}_{1:n}$,
$$f_{Y}^{h, \tau}(x) - f_{Y}^{h, \tau}(x^{\prime}) \leq \frac{1}{n}\sum_{i=1}^{n}\1_{x_i \neq x_i^{\prime}}.
$$
We thus may apply (2.5) in Theorem~2.9 of \cite{Paulin15} to get
\begin{equation}
\label{eq:Paulin}
\EE_{\theta}\left[e^{\lambda (f_{Y}^{h, \tau}(X) - \EE_{\theta}\left[f_{Y}^{h, \tau}(X) \mid Y_{1:n}\right])} \bigg| Y_{1:n}\right] \leq e^{\frac{\lambda^{2}}{8n^{2}}\sum_{i=1}^{n}\left(\sum_{j=i}^{n}\Gamma_{i, j}\right)^{2}},
\end{equation}
where $\Gamma$ comes from a Marton coupling (see Definition 2.1 in \cite{Paulin15}) and is given by:
$$
\Gamma_{j, i}:=0,\qquad \Gamma_{i, j}:=\sup _{x_1, \ldots, x_i, x_i^{\prime}\in\mathbb{X}} \mathbb{P}_{\theta}\left(X_{1, j}^{\left(x_1, \ldots, x_i, x_i^{\prime}, Y_{1:n}\right)} \neq X_{2, j}^{\left(x_1, \ldots, x_i, x_i^{\prime}, Y_{1:n}\right)}\bigg| Y_{1:n}\right)
$$
for $1 \leq i < j \leq n$. Now,
\begin{multline*}
\Gamma_{i,j} \leq \sup_{x_{1:i-1}\in\mathbb{X}^{i-1}}
       \Bigg\|\PP_{\theta}\left(X_{1,j}^{(x_{1:i-1}, \Tilde{a}, \Tilde{b}, Y_{1:n})}\in \cdot\bigg| X_{1:i-1} = x_{1:i-1}, Y_{1:n}\right)\\ - \PP_{\theta}\left(X_{2,j}^{(x_{1:i-1}, \Tilde{a}, \Tilde{b}, Y_{1:n})}\in \cdot\bigg|X_{1:i-1} = x_{1:i-1}, Y_{1:n}\right)\Bigg\|_{\mathrm{TV}},
\end{multline*}
ie.,
\begin{multline*}
\Gamma_{i,j} \leq \sup_{x_{1:i-1}\in\mathbb{X}^{i-1}}  \Bigg\|\PP_{\theta}\left(X_{j}\in \cdot \bigg| X_{1:i-1} = x_{1:i-1}, X_i = \Tilde{a}, Y_{1:n}\right)\\-\PP_{\theta}\left(X_{j}\in \cdot\bigg| X_{1:i-1} = x_{1:i-1}, X_i = \Tilde{b}, Y_{1:n}\right)\Bigg\|_{\mathrm{TV}},
\end{multline*}
ie.,
\begin{equation*}
    \Gamma_{i, j} \leq \Bigg\|\PP_{\theta}\left(X_{j}\in \cdot \bigg| X_i = \Tilde{a}, Y_{1:n}\right)- \PP_{\theta}\left(X_{j}\in \cdot \bigg| X_i = \Tilde{b}, Y_{1:n}\right)\Bigg\|_{\mathrm{TV}}
\end{equation*}
since conditional on $Y_{1:n}$, the hidden states form a inhomogeneous Markov chain with transition kernels $(F_{k|n}[Y_{k+1:n}])$. 
Exponential forgetting of the smoothing distributions in HMMs (Proposition 4.3.26 in \cite{CMT05}) allows to conclude
that
$$
\Bigg\|\PP_{\theta}\left(X_{j}\in \cdot \bigg| X_i = \Tilde{a}, Y_{1:n}\right)- \PP_{\theta}\left(X_{j}\in \cdot \bigg| X_i = \Tilde{b}, Y_{1:n}\right)\Bigg\|_{\mathrm{TV}}\leq \rho_{0}^{j-i}
$$
where $\rho_0 =\frac{1 - J\delta}{1 - (J-1)\delta} $ (see also Lemma~\ref{Lemma_3}).
By inequality \eqref{eq:Paulin}:
$$
\max_{\tau}\EE_{\theta}\left[e^{\lambda (f_{Y}^{h, \tau}(X) - \EE_{\theta}\left[f_{Y}^{h, \tau}(X) \mid Y_{1:n}\right])} \bigg| Y_{1:n}\right] \leq e^{\frac{1}{8n}\left(\frac{\lambda}{1 - \rho_0}\right)^{2}}.
$$
Thus,
\begin{eqnarray*}
\EE_{\theta}\left[\max_{\tau}\left(\hat{p}_{\tau}(h)-U_{n, \tau}(h)\right) \bigg| Y_{1:n}\right]
&\leq& \inf_{\lambda>0}\left\{\frac{\log (J!)}{\lambda}+\frac{\lambda}{8n(1-\rho_0)^2} \right\}\\
&=& \frac{1}{1-\rho_{0}} \sqrt{\frac{\log(J!)}{2n}}
\end{eqnarray*}
\end{proof}
\begin{lemma}
  \label{lem:bayes-risk:lb:HMM:2}
  For all $\theta\in\Theta^{\mathrm{dep}}$, $\PP_{\theta}$-almost-surely
  \begin{equation*}
    \PP_{\theta}\big( U_{n,\hat\tau_h}(h) > \eta  \mid Y_{1:n}\big)%
    \leq  \hat{p}_{\hat\tau_h}(h) \cdot \frac{e}{\eta}\left(\frac{1-(J-1)\delta}{\delta}\right)^{2n}\Big(\frac{e \hat{p}_{\hat\tau_h}(h)}{\eta} \Big)^{n\eta - 1}e^{-n \hat{p}_{\hat\tau_h}(h)}.
  \end{equation*}
\end{lemma}
\begin{proof}
Let $S = \sum_{i=1}^{n}\1_{\hat{\tau}_{h}(X_i)\neq h_{i}(Y_{1:n})}$. We consider the following operators defined on $L^{\infty}(\left\{0, 1\right\})$, for $i\in[n]$:
$$
(M_{\theta, i}.f)(x) \coloneqq f(0)e^{\lambda \1_{\hat{\tau}_{h}(0)\neq h_{i}(Y_{1:n})} } B_{\theta, i}(x, 0) + f(1) e^{\lambda \1_{\hat{\tau}_{h}(1)\neq h_{i}(Y_{1:n})}} B_{\theta, i}(x, 1)
$$
where $B_{\theta, i}$ is the Backward kernel defined by:
\begin{equation*}
    B_{\theta, i}(x, y) = \frac{\phi_{\theta, i}(y) Q(y, x)}{\sum_{y^{\prime}}\phi_{\theta, i}(y^{\prime})Q(y^{\prime}, x)}.
\end{equation*}
Then observe that,
\begin{equation*}
    \begin{split}
        \EE_{\theta}[e^{\lambda S}\mid Y_{1:n}] &= \EE_{\theta}\left[e^{\lambda\sum_{i=2}^{n}\1_{\hat{\tau}_{h}(X_i)\neq h_{i}(Y_{1:n})} }
        \EE_{\theta}\left[e^{\lambda\1_{\hat{\tau}_{h}(X_1)\neq h_{1}(Y_{1:n})} } \mid X_{2:n}, Y_{1:n}\right]\mid Y_{1:n}\right]\\
        &= \EE_{\theta}\left[e^{\lambda\sum_{i=2}^{n}\1_{\hat{\tau}_{h}(X_i)\neq h_{i}(Y_{1:n})} }(M_{\theta, 1}.\1)(X_2)\mid Y_{1:n}\right]
    \end{split}
\end{equation*}
Repeating inductively the same trick leads to
\begin{equation*}
\EE_{\theta}[e^{\lambda S}\mid Y_{1:n}]%
= \EE_{\theta}\left[(M_{\theta, n}...M_{\theta, 1}.\1)(X_{n+1})\mid Y_{1:n}\right]
\end{equation*}
Hence
\begin{equation*}
\EE_{\theta}[e^{\lambda S}\mid Y_{1:n}]%
\leq \lVert (M_{\theta, n}...M_{\theta, 1}.\1) \rVert_{\infty}%
\leq \prod_{i=1}^{n}|||M_{\theta, i}|||_{\infty}
\end{equation*}
where
\begin{equation*}
    \begin{split}
        |||M_{\theta, i}|||_{\infty} &\coloneqq \sup_{f, ||f||_{\infty} = 1}||M_{\theta, i}.f||_{\infty}\\
        &=\max\left((M_{\theta, i}.f)(0), (M_{\theta, i}.f)(1)\right)\\
        &\leq \max\left(\sum_{z\in\left\{0, 1\right\}}e^{\lambda\1_{\hat{\tau}_{h}(z)\neq h_{i}(Y_{1:n})}}B_{\theta, i}(0, z), \sum_{z\in\left\{0, 1\right\}}e^{\lambda\1_{\hat{\tau}_{h}(z)\neq h_{i}(Y_{1:n})}}B_{\theta, i}(1, z)\right).
    \end{split}
\end{equation*}
But, given that
\begin{equation*}
    B_{\theta, i}(x, y) = \frac{\phi_{\theta, i}(y) Q(y, x)}{\sum_{y^{\prime}}\phi_{\theta, i}(y^{\prime})Q(y^{\prime}, x)}\leq \frac{1-(J-1)\delta}{\delta}\phi_{\theta, i}(y)
\end{equation*}
and that
\begin{equation*}
    \begin{split}
        \phi_{\theta, i|n}(y) &= \phi_{\theta, i+1|n}B_{\theta, i}(y) = \sum_{x}\frac{\phi_{\theta, i+1|n}(x)\phi_{\theta, i}(y) Q(y, x)}{\sum_{y^{\prime}}\phi_{\theta, i}(y^{\prime})Q(y^{\prime}, x)} \geq \frac{\delta \phi_{\theta, i}(y)}{1-(J-1)\delta}
    \end{split}
\end{equation*}
one obtains
\begin{equation*}
B_{\theta, i}(x, y) \leq \left(\frac{1-(J-1)\delta}{\delta}\right)^{2}\phi_{\theta, i|n}(y).
\end{equation*}
Thus,
\begin{equation*}
    \begin{split}
        |||M_{\theta, i}|||_{\infty}&\leq \left(\frac{1-(J-1)\delta}{\delta}\right)^{2}\left(e^{\lambda\1_{\hat{\tau}_{h}(0)\neq h_{i}(Y_{1:n})}}\phi_{i|n}(0) + e^{\lambda\1_{\hat{\tau}_{h}(1)\neq h_{i}(Y_{1:n})}}\phi_{i|n}(1)\right)\\
        &= \left(\frac{1-(J-1)\delta}{\delta}\right)^{2}\EE_{\theta}\left[e^{\lambda\1_{\hat{\tau}_{h}(X_i)\neq h_{i}(Y_{1:n})}}\mid Y_{1:n}\right].
    \end{split}
\end{equation*}
Finally,
\begin{equation*}
    \EE_{\theta}[\exp(\lambda S)\mid Y_{1:n}] \leq \left(\frac{1-(J-1)\delta}{\delta}\right)^{2n}\prod_{i=1}^{n}\EE_{\theta}\left[e^{\lambda\1_{\hat{\tau}_{h}(X_i)\neq h_{i}(Y_{1:n})}}\mid Y_{1:n}\right].
\end{equation*}
One can then use Chernoff's bound (with $q_i(h) = \PP_{\theta}(h_i(Y_{1:n}) \ne \hat\tau_h(X_i) \mid Y_{1:n})$):
\begin{equation*}
    \begin{split}
    \PP_{\theta}\big(U_{n,\hat\tau_h}(h) > \eta  \mid Y_{1:n}\big) &= \PP_{\theta}\Big(\sum_{i=1}^n 1_{h_{i}(Y_{1:n}) \ne \hat{\tau}_{h}(X_i)} > n\eta \mid Y_{1:n}\Big)\\
    &\leq \inf_{\lambda > 0}e^{-\lambda n\eta}\EE_{\theta}\left[e^{\lambda S}\mid Y_{1:n}\right]\\
    &\leq \inf_{\lambda > 0}e^{-\lambda n\eta}\left(\frac{1-(J-1)\delta}{\delta}\right)^{2n}\prod_{i=1}^{n}\EE_{\theta}\left[e^{\lambda\1_{\hat{\tau}_{h}(X_i)\neq h_{i}(Y_{1:n})}}\mid Y_{1:n}\right]\\
    &\leq \inf_{\lambda > 0}e^{-\lambda n\eta}\left(\frac{1-(J-1)\delta}{\delta}\right)^{2n}e^{\sum_{i=1}^{n}\log\left(q_{i}(h)e^{\lambda} + 1 - q_{i}(h)\right)}\\
    &\leq \inf_{\lambda > 0}e^{-\lambda n\eta}\left(\frac{1-(J-1)\delta}{\delta}\right)^{2n}e^{n\hat{p}_{\hat\tau_h}(h)(e^{\lambda}-1)}\\
    &\leq \hat{p}_{\hat\tau_h}(h) \cdot \frac{e}{\eta}\left(\frac{1-(J-1)\delta}{\delta}\right)^{2n} \Big(\frac{e \hat{p}_{\hat\tau_h}(h)}{\eta} \Big)^{n\eta - 1}e^{-n \hat{p}_{\hat\tau_h}(h)}
    \end{split}
\end{equation*}
\end{proof}
\begin{lemma}
  \label{lem:bayes-risk:lb:HMM:3}
  Let  $\beta = \min_{i, j\neq k}\PP_{\theta}\left(X_{i}\in\left\{j, k\right\}\right)$ and $\rho_1 = 1 - J\delta$. If $J \geq 3$, then for $\eta < \frac{\beta}{2}$ and $\theta\in\Theta^{\mathrm{dep}}$
  \begin{equation*}
    \PP_{\theta}\big(N_{(1)} + N_{(2)} < 2n\eta\big)%
    \leq J^2e^{-2n(1 - \rho_1)^{2}(\beta - 2\eta)^{2}}.
  \end{equation*}
\end{lemma}
\begin{proof}
    Let $\lambda > 0$, $\eta < \frac{\beta}{2}$, $j\neq k$ , $\beta_{i}(j, k) = \PP_{\theta}\left(X_i\in\left\{j, k\right\}\right)$.
\begin{equation*}
    \begin{split}
        \PP_{\theta}\left(N_j + N_k < 2n\eta\right) & = \PP_{\theta}\left(\sum_{i=1}^{n}(\1_{X_i\in\left\{j, k\right\}} - \beta_{i}(j, k)) < 2n\eta - \sum_{i=1}^{n}\beta_{i}(j, k)\right)\\
        &= \PP_{\theta}\left(e^{\lambda \sum_{i=1}^{n}\left(\beta_{i}(j, k) - \1_{X_i\in\left\{j, k\right\}}\right)} > e^{\lambda(\sum_{i=1}^{n}\beta_{i}(j, k) - 2n\eta)}\right)\\
        &\leq e^{-\lambda(\sum_{i=1}^{n}\beta_{i}(j, k) - 2n\eta)}\EE_{\theta}\left[e^{\lambda \sum_{i=1}^{n}\left(\beta_{i}(j, k) - \1_{X_i\in\left\{j, k\right\}}\right)}\right]
    \end{split}
\end{equation*}
$\EE_{\theta}\left[e^{\lambda \sum_{i=1}^{n}\left(\beta_{i}(j, k) - \1_{X_i\in\left\{j, k\right\}}\right)}\right]$ can be controlled by the same technique using the Marton coupling that was used in the proof of Lemma~\ref{lem:bayes-risk:lb:HMM:1}:
$$
\EE_{\theta}\left[e^{\lambda \sum_{i=1}^{n}\left(\beta_{i}(j, k) - \1_{X_i\in\left\{j, k\right\}}\right)}\right] \leq e^{\frac{(n\lambda)^{2}}{8n(1-\rho_1)^{2}}} = e^{\frac{n\lambda^{2}}{8(1-\rho_1)^{2}}}
$$
It follows that
\begin{align*}
\PP_{\theta}\left(N_j + N_k < 2n\eta\right) &\leq \inf_{\lambda > 0}e^{-\lambda(\sum_{i=1}^{n}\beta_{i}(j, k) - 2n\eta) + \frac{n\lambda^{2}}{8(1-\rho_1)^{2}}}\\
&\leq e^{-2n(1 - \rho_1)^{2}\left(\frac{\sum_{i=1}^{n}\beta_{i}(j, k)}{n} - 2\eta\right)^{2}}\\
&\leq e^{-2n(1 - \rho_1)^{2}(\beta - 2\eta)^{2}}
\end{align*}
Thus,
\begin{align*}
\PP_{\theta}\left(N_{(1)} + N_{(2)} < 2n\eta\right)
&\leq J^{2}\max_{j\neq k}\PP_{\theta}\left(N_j + N_k < 2n\eta\right)\\
&\leq J^{2}e^{-2n(1 - \rho_1)^{2}(\beta - 2\eta)^{2}}.
\end{align*}
\end{proof}

\section{Proof of \texorpdfstring{\cite[Theorem~\ref*{main-thm:bayes-risk:min:hmm:J=2}]{GKN2025main}}{[4, Theorem 8]}}
\label{proof:bayes-risk:min:hmm:J=2}

Let $a\in\left\{1, 2\right\}$ and $n\in\mathbb{N}$. Let $Q = \begin{pmatrix}
1-p & p \\
q & 1-q
\end{pmatrix}$ be the transition matrix and assume $0 <q < \frac{1}{2} < p$ and $p+q < 1$. Assume also that the initial distribution is the stationary distribution, that is $\nu = \left(\frac{q}{p+q}, \frac{p}{p+q}\right)$.
\begin{equation*}
    \begin{split}
        \PP_{\theta}\left(X_1 = a \mid Y_{1:n}\right) &= \sum_{x_{2:n}\in\left\{1, 2\right\}}\PP_{\theta}\left(X_1 = a, X_{2:n} = x_{2:n} \mid Y_{1:n}\right)\\
        &\propto \sum_{x_{2:n}\in\left\{1, 2\right\}} \nu(a) Q_{a, x_2}...Q_{x_{n-1}, x_n}f_{a}(Y_1)..f_{x_n}(Y_n)\\
        \PP_{\theta}\left(X_2 = a \mid Y_{1:n}\right) &= \sum_{x_1, x_{3:n}\in\left\{1, 2\right\}}\PP_{\theta}\left(X_1 = x_1, X_2 = a, X_{3:n} = x_{3:n} \mid Y_{1:n}\right)\\
        &\propto \sum_{x_1, x_{3:n}\in\left\{1, 2\right\}} \nu(x_1) Q_{x_1, a}...Q_{x_{n-1}, x_n}f_{x_1}(Y_1)f_{a}(Y_2)...f_{x_n}(Y_n)
    \end{split}
\end{equation*}
The Bayes classifier puts the two first observations in the same cluster when:
\begin{equation*}
    \begin{split}
        &\biggl[\PP_{\theta}\left(X_1 = 2 \mid Y_{1:n}\right) - \PP_{\theta}\left(X_1 = 1 \mid Y_{1:n}\right)\biggr]\times\biggl[\PP_{\theta}\left(X_2 = 2 \mid Y_{1:n}\right) - \PP_{\theta}\left(X_2 = 1 \mid Y_{1:n}\right)\biggr] \geq 0\\
        &\iff \biggl[\sum_{x_{2:n}\in\left\{1, 2\right\}}\left(\nu(2)Q_{2, x_2}f_{2}(Y_1) - \nu(1)Q_{1, x_2}f_{1}(Y_1)\right)Q_{x_{2}, x_3}...Q_{x_{n-1}, x_n}f_{x_2}(Y_2)...f_{x_n}(Y_n)\biggr]\\
        &\times\biggl[\sum_{x_1, x_{3:n}\in\left\{1, 2\right\}}\left(Q_{x_1, 2}Q_{2, x_3}f_{2}(Y_2) - Q_{x_1, 1}Q_{1, x_3}f_{1}(Y_2)\right)\nu(x_1)Q_{x_{3}, x_4}...Q_{x_{n-1}, x_n}f_{x_1}(Y_1)f_{x_3}(Y_3)...f_{x_n}(Y_n)\biggr] \geq 0
    \end{split}
\end{equation*}
A sufficient condition for this to be ensured is:
\begin{equation*}
    \begin{split}
        & \left(\frac{f_{1}}{f_{2}}(Y_1) < \frac{p\min(q, 1-q)}{q\max(p, 1-p)} \text{ and } \frac{f_{1}}{f_{2}}(Y_2) < \frac{\min(p, 1-q)\min(q, 1-q)}{\max(q,1-p)\max(p,1-p)}\right) \\
        &\text{ or }\\
        & \left(\frac{f_{2}}{f_{1}}(Y_1) < \frac{q\min(p, 1-p)}{p\max(q, 1-q)} \text{ and } \frac{f_{2}}{f_{1}}(Y_2) < \frac{\min(q,1-p)\min(p, 1-p)}{\max(p,1-q)\max(q,1-q)}\right)
    \end{split}
\end{equation*}

Since $q < \frac{1}{2} < p$ and $p+q < 1$, the condition simplifies to:
\begin{equation*}
    \left(\frac{f_{1}}{f_{2}}(Y_1) < 1 \text{ and } \frac{f_{1}}{f_{2}}(Y_2) < \frac{q}{1-p}\right)\text{ or } \left( \frac{f_{2}}{f_{1}}(Y_1) < \frac{q (1-p)}{p (1-q)} \text{ and } \frac{f_{2}}{f_{1}}(Y_2) < \frac{q (1-p)}{(1-q)^{2}}\right).
\end{equation*}

We consider the event:
\begin{equation*}
    A = \left\{\frac{f_1}{f_2}(Y_1) < 1, \frac{f_1}{f_2}(Y_2) < \frac{q}{1-p}\right\}\bigcup\left\{\frac{f_2}{f_1}(Y_1) < \frac{q(1-p)}{p(1-q)}, \frac{f_2}{f_1}(Y_2) < \frac{q(1-p)}{(1-q)^{2}}\right\}
\end{equation*}
In what follows, we seek a sufficient condition under which the Bayes clusterer puts the two first observations in two different clusters. The Bayes clusterer is a partition $F_n^{\star}$ that minimizes $\EE_{\theta}\left[\ell\left(\pi_n\left(X_{1:n}\right), F_n\right)\mid Y_{1:n}\right]$.
Let $L(Y_{1:n})$ be the likelihood of the observations $Y_{1:n}$. Consider the event 
\begin{equation*}
    B_n = \left\{\left(\forall i\in\llbracket 3, n\rrbracket\right)\quad Y_i\notin \text{Supp}\left(f_2\right) \right\}
\end{equation*}
Assume $B_n$ has positive probability. Since the hidden Markov chain is mixing, this happens for example when $f_1$ and $f_2$ do not have the same support. On this event,
\begin{equation*}
    \begin{split}
        &\EE_{\theta}\left[\ell\left(\pi_{n}(X_{1:n}), F_n\right) \mid Y_{1:n}\right] = \sum_{x_{1:n}\in\left\{1, 2\right\}^{n}}\ell\left(\pi_{n}(x_{1:n}), F_n\right) \PP_{\theta}\left(X_{1:n} = x_{1:n} \mid Y_{1:n}\right)\\
        &= \sum_{x_{1:2}\in\left\{1, 2\right\}^{2}}\ell\left(\pi_{n}((x_{1}, x_{2}, 1,\dots, 1), F_n\right)\PP_{\theta}\left(X_{1:2} = x_{1:2}, X_{3:n} = 1\mid Y_{1:n}\right)\\
        &= \frac{1}{L(Y_{1:n})}\sum_{x_{1:2}\in\left\{1, 2\right\}^{2}}\ell\left(\pi_{n}((x_{1}, x_{2}, 1,\dots, 1)), F_n\right)\nu(x_1)Q_{x_1, x_2} Q_{x_2, 1} Q_{1, 1}^{n-3}f_{x_1}(Y_1)f_{x_2}(Y_2)\prod_{i=3}^{n}f_{1}(Y_i)\\
        &\propto\Big(\ell\left(\pi_{n}((1,\dots, 1)), F_n\right)\nu(1)Q_{1, 1}^{2}f_{1}(Y_1)f_{1}(Y_2) + \ell\left(\pi_{n}((2, 2, 1, \dots, 1), F_n\right)\nu(2)Q_{2, 2} Q_{2, 1}f_{2}(Y_1)f_{2}(Y_2)\\
        & + \ell\left(\pi_{n}((1, 2, 1,\dots,1)), F_n\right)\nu(1)Q_{1, 2}Q_{2, 1}f_{1}(Y_1)f_{2}(Y_2) +\ell\left(\pi_{n}((2, 1, \dots, 1)), F_n\right)\nu(2)Q_{2, 1}Q_{1, 1}f_{2}(Y_1)f_{1}(Y_2)\Big)\\
        &\propto\Big(\ell\left(\pi_{n}((1,\dots, 1)), F_n\right)q(1-p)^{2}f_{1}(Y_1)f_{1}(Y_2) + \ell\left(\pi_{n}((2, 2, 1, \dots, 1), F_n\right)pq(1-q)f_{2}(Y_1)f_{2}(Y_2)\\
        & + \ell\left(\pi_{n}((1, 2, 1,\dots,1)), F_n\right)pq^{2}f_{1}(Y_1)f_{2}(Y_2) +\ell\left(\pi_{n}((2, 1, \dots, 1)), F_n\right)pq(1-p)f_{2}(Y_1)f_{1}(Y_2)\Big)
    \end{split}
\end{equation*}
For $n \geq 5$, $F^{\star}_n$ is necessarily of the form $F^{\star}_n = \pi_n(\left(y^{\star}_1, y^{\star}_2, 1, .., 1\right))$ with $y^{\star}_1$ and $y^{\star}_2$ in $\left\{1, 2\right\}$. Thus, 
\begin{equation*}
    \begin{split}
        F^{\star}_n \in \argmin\EE_{\theta}\left[\ell\left(\pi_{n}(X_{1:n}), F_n\right) \mid Y_{1:n}\right] &\iff (y^{\star}_1, y^{\star}_2)\in \argmin H(y_1, y_2)
    \end{split}
\end{equation*}
where 
\begin{equation*}
    \begin{split}
        H(y_1, y_2) &=  (y_1 + y_2 - 2) (1-p)^{2}f_{1}(Y_1)f_{1}(Y_2) + (4 - y_1 - y_2)p(1-q) f_{2}(Y_1)f_{2}(Y_2) \\
        & + (1 + y_1 - y_2)pq f_{1}(Y_1) f_{2}(Y_2) + (1 - y_1 + y_2)p(1-p)f_{2}(Y_1) f_{1} (Y_2)\\
        &= y_1 \left[(1-p)^{2}f_{1}(Y_1)f_{1}(Y_2) + pq f_{1}(Y_1) f_{2}(Y_2) - p(1-q) f_{2}(Y_1)f_{2}(Y_2) - p(1-p)f_{2}(Y_1) f_{1} (Y_2)\right]\\
        &+ y_2 \left[(1-p)^{2}f_{1}(Y_1)f_{1}(Y_2) - pq f_{1}(Y_1) f_{2}(Y_2) - p(1-q) f_{2}(Y_1)f_{2}(Y_2) + p(1-p)f_{2}(Y_1) f_{1} (Y_2)\right]\\
        & -2 (1-p)^{2} f_{1}(Y_1)f_{1}(Y_2) + 4p(1-q)f_{2}(Y_1)f_{2}(Y_2) + pq f_{1}(Y_1) f_{2}(Y_2) + p(1-p)f_{2}(Y_1) f_{1} (Y_2)
    \end{split}
\end{equation*}
\begin{equation*}
    \begin{split}
        y^{\star}_1 \neq y^{\star}_2 &\iff\left[(1-p)^{2}f_{1}(Y_1)f_{1}(Y_2) + pq f_{1}(Y_1) f_{2}(Y_2) - p(1-q) f_{2}(Y_1)f_{2}(Y_2) - p(1-p)f_{2}(Y_1) f_{1}(Y_2)\right]\\
        &\times\left[(1-p)^{2}f_{1}(Y_1)f_{1}(Y_2) - pq f_{1}(Y_1) f_{2}(Y_2) - p(1-q) f_{2}(Y_1)f_{2}(Y_2) + p(1-p)f_{2}(Y_1) f_{1} (Y_2)\right] < 0\\
        &\iff \left|(1-p)^{2}f_{1}(Y_1)f_{1}(Y_2) - p(1-q) f_{2}(Y_1)f_{2}(Y_2) \right| < p\left|q f_{1}(Y_1) f_{2}(Y_2) - (1-p) f_{2}(Y_1)f_{1}(Y_2)\right|
    \end{split}
\end{equation*}
Finally, consider the event:
\begin{equation*}
    C_n = \left\{\left|(1-p)^{2}f_{1}(Y_1)f_{1}(Y_2) - p(1-q) f_{2}(Y_1)f_{2}(Y_2) \right| < p\left|q f_{1}(Y_1) f_{2}(Y_2) - (1-p) f_{2}(Y_1)f_{1}(Y_2)\right|\right\}
\end{equation*}
We finally have: 
\begin{equation*}
    A \cap B_n \cap C_n \subset \left\{g_{\theta}^{\star}(Y_{1:n}) \neq  \pi_n\circ h^{\star}_{\theta}(Y_{1:n})\right\}.
\end{equation*}
By choosing appropriately the parameters $p$ and $q$, one can ensure that $\PP_{\theta}\left(A \cap B_n \cap C_n\right)$ is positive for many emission densities not having the same support. For example, one can ensure that for $p=0.58$, $q=0.35$, $f_{1}(Y_1)=5$, $f_{2}(Y_1)= 2.5$, $f_{1}(Y_2)=1.8$, $f_{2}(Y_2)=1.2$, both inequalities involved in the definition of event $C_n$ and $A$ are ensured. Choosing smooth densities having not exactly the same support, the event $A \cap B_n \cap C_n$ can be ensured to have positive probability. The counterexamples for $n\in\left\{3, 4\right\}$ can be proved as for the case $n=2$ presented in \cite[Section~\ref*{main-sec:comp-clust-class-hmm}]{GKN2025main}.

\section{Proof of \texorpdfstring{\cite[Theorem~\ref*{main-thm:bayes-risk:min:hmm:J>2}]{GKN2025main}}{[4, Theorem 10]}}
\label{proof:bayes-risk:min:hmm:J>2}

Let $n\in\N$ and $\theta^{\star} = \left(\nu^{\star}, Q^{\star}, \left(f^{\star}_{x}\right)_{x\in\mathbb{X}}\right)\in\Theta^{\mathrm{ind}}$ such that the assumption of \cite[Theorem~\ref*{main-thm:bayes-risk:min:iid:J>2}]{GKN2025main} is ensured. It follows that $\PP_{\theta^{\star}}\left(g_{\theta^{\star}}^{\star}(Y_{1:n}) \neq \pi_n\circ h_{\theta^{\star}}^{\star}(Y_{1:n})\right) > 0$. We also assume the emission densities $\left(f^{\star}_{x}\right)_{x\in\mathbb{X}}$ to be uniformly continuous. We denote $\Tilde{\Theta} \subset \Theta$ the subset of parameters of the form $\theta = \left(\nu, Q, f^{\star}_{1}, \dots, f^{\star}_{J}\right)$. Consider the two functions
\begin{align*}
H \colon \mathcal{H}_{n}\times \Theta \times \mathcal{Y}^{n} & \longrightarrow\left[0, 1\right]\\
\left(h, \theta, y_{1:n}\right)&\longmapsto \EE_{\theta}\left[\frac{1}{n}\sum_{i=1}^{n}\1_{h_{i}(Y_{1:n})\neq X_i}\bigg| Y_{1:n} = y_{1:n}\right]\\
G \colon \mathcal{G}_{n}\times \Theta \times \mathcal{Y}^{n} & \longrightarrow\left[0, 1\right]\\
\left(g, \theta, y_{1:n}\right)&\longmapsto \EE_{\theta}\left[\ell\left(\pi_{n}\left(X_{1:n}\right), g(Y_{1:n})\right)\bigg| Y_{1:n} = y_{1:n}\right]
\end{align*}
By uniform continuity of $\left(f^{\star}_{x}\right)_{x\in\mathbb{X}}$, it follows that for all $g\in\mathcal{G}_{n}$ and $h\in\mathcal{H}_{n}$, $H(h, ., .)$ and $G(g, ., .)$ are uniformly continuous. Since $\mathcal{H}_{n}$ and $\mathcal{G}_{n}$ are finite, there exists $\mathcal{V}(\theta^{\star})\subset \Tilde{\Theta}$ a neighborhood of $\theta^{\star}$ and $A_n$ an open subset of $\mathcal{Y}^{n}$ such that $\PP_{\theta^{\star}}\left(A_n\right) > 0$ and $\left(\forall \theta\in\mathcal{V}(\theta^{\star})\right) \left(\forall y_{1:n}\in A_n\right)$
\begin{equation*}
    \begin{split}
        \argmin_{h} H(h, \theta, y_{1:n}) = \argmin_{h} H(h, \theta^{\star}, y_{1:n}) \text{ and } \argmin_{g} G(g, \theta, y_{1:n}) = \argmin_{g} G(g, \theta^{\star}, y_{1:n})
    \end{split}
\end{equation*}
Or equivalently $g^{\star}_{\theta^{\star}}(y_{1:n}) = g^{\star}_{\theta}(y_{1:n})$ and $h^{\star}_{\theta^{\star}}(y_{1:n}) = h^{\star}_{\theta}(y_{1:n})$. On the other hand, using exactly the same arguments as those of \cite[Theorem~\ref*{main-thm:bayes-risk:min:iid:J>2}]{GKN2025main}, one could also have chosen $\theta^{\star}\in\Theta^{\mathrm{ind}}$ such that not only $\PP_{\theta^{\star}}\left(g_{\theta^{\star}}^{\star}(Y_{1:n}) \neq \pi_n\circ h_{\theta^{\star}}^{\star}(Y_{1:n})\right) > 0$ but also $\PP_{\theta^{\star}}\left(\left\{g_{\theta^{\star}}^{\star}(Y_{1:n}) \neq \pi_n\circ h_{\theta^{\star}}^{\star}(Y_{1:n})\right\}\cap A_{n}\right) > 0$. It follows that 
\begin{equation*}
    \left(\forall \theta \in\mathcal{V}(\theta^{\star})\right) \PP_{\theta}\left(\left\{g^{\star}_{\theta}(Y_{1:n}) \neq \pi_{n}\circ h^{\star}_{\theta}(Y_{1:n})\right\}\cap A_{n}\right) = \PP_{\theta}\left(\left\{g^{\star}_{\theta^{\star}}(Y_{1:n}) \neq \pi_{n}\circ h^{\star}_{\theta^{\star}}(Y_{1:n})\right\}\cap A_{n}\right)
\end{equation*}
which is positive when $\theta$ approches $\theta^{\star}$ by continuity of the map 
\begin{equation*}
    \theta\mapsto \PP_{\theta}\left(\left\{g^{\star}_{\theta^{\star}}(Y_{1:n}) \neq \pi_{n}\circ h^{\star}_{\theta^{\star}}(Y_{1:n})\right\}\cap A_{n}\right).
\end{equation*}
The result follows. 

\section{Proof of \texorpdfstring{\cite[Proposition~\ref*{main-pro:bayes-risk:non-equivalence}]{GKN2025main}}{[4, Proposition 1]}}
\label{sec:proof-prop-refpr}

The result is straightforward when $n = 1$. We assume in what follows $n \geq 2$. We prove the proposition by showing that when the probability of having small clusters is high, the two risks are not necessarily equivalent; and $\inf_{g \in \mathcal{G}_n}\ClusterRisk(\theta,g)$ may be much smaller than $\inf_{h \in \mathcal{H}_n}\ClassifRisk(\theta,h)$.

 Consider $J=3$ (similar examples can be constructed for any $J \geq 3$) with $\nu_1 = 1 - 2\eta$, $\nu_2 = \nu_3 = \eta$. We take $F_1 = U(0,1/2)$, $F_2= U(3/4,1)$ and $F_3 = U(3/4-\varepsilon,1-\varepsilon)$ for some $0 < \varepsilon < 1/4$ where $U(a, b)$ is the uniform distribution on the interval $(a, b)$. In this case,
\begin{align*}
  \PP_{\theta}(X_i \in \cdot \mid Y_i \in (0,1/2))%
  &= \delta_1(\cdot)\\
  \PP_{\theta}(X_i \in \cdot \mid Y_i \in (3/4-\varepsilon,3/4))
  &= \delta_3(\cdot)\\
  \PP_{\theta}(X_i \in \cdot \mid Y_i \in (1-\varepsilon,1))
  &= \delta_2(\cdot)\\
  \PP_{\theta}(X_i \in \cdot \mid Y_i \in (3/4,1-\varepsilon))
  &= \frac{1}{2}\delta_2(\cdot) + \frac{1}{2}\delta_3(\cdot).
\end{align*}
So in this case
\begin{align*}
  \EE_{\theta}\left[U_{n,\tau}(h) \mid Y_{1:n}\right]%
  &= \EE_{\theta}\Big[\frac{1}{n}\sum_{i=1}^n\1_{h_i(Y_{1:n}) \ne \tau(X_i)} \mid Y_{1:n} \Big]\\
  &= \frac{1}{n}\sum_{i=1}^n \PP_{\theta}(h_i(Y_{1:n}) \ne \tau(X_i) \mid Y_{1:n}).
\end{align*}
A Bayes ``classifier'' minimizing $h \mapsto \mathcal{R}_n^{\mathrm{class}}(\theta,h)$ in this case is given by $h^{\star}_{\theta} = \left(h^{\star}_{\theta, i}\right)_{i\in[n]}$ where
\begin{equation*}
  h_{\theta, i}^{\star}%
  =%
  \begin{cases}
    1 &\mathrm{if}\ Y_i\in (0,1/2)\\
    2 &\mathrm{if}\ Y_i \in (1-\varepsilon, 1)\\
    3 &\mathrm{if}\ Y_i \in (3/4-\varepsilon,1-\varepsilon).
  \end{cases}
\end{equation*}
With this choice, the optimal permutation is identity and
\begin{align*}
  \min_{\tau}\EE_{\theta}\left[U_{n,\tau}(h^{\star}_{\theta}) \mid Y_{1:n}\right]%
  &= \frac{1}{n}\sum_{i=1}^n\1_{Y_i \in (0,1/2)} \PP_{\theta}(1 \ne X_i \mid Y_i)\\%
  &\quad
    + \frac{1}{n}\sum_{i=1}^n\1_{Y_i \in (3/4-\varepsilon,3/4)} \PP_{\theta}(3 \ne X_i \mid Y_i)\\
  &\quad%
    +\frac{1}{n}\sum_{i=1}^n\1_{Y_i \in (1-\varepsilon,1)} \PP_{\theta}(2 \ne X_i \mid Y_i)\\%
    &\quad
    +\frac{1}{n}\sum_{i=1}^n\1_{Y_i \in (3/4,1-\varepsilon)} \PP_{\theta}(3 \ne X_i \mid Y_i),
\end{align*}
ie.,
\begin{equation*}
\min_{\tau}\EE_{\theta}\left[U_{n,\tau}(h^{\star}_{\theta}) \mid Y_{1:n}\right]%
  = \frac{1}{2n}\sum_{i=1}^n\1_{Y_i \in (3/4,1-\varepsilon)}.
\end{equation*}
Thus,
\begin{align*}
  \EE_{\theta}\Big[ \min_{\tau}\EE_{\theta}(U_{n,\tau}(h^{\star}_{\theta}) \mid Y_{1:n}) \Big]%
  &= \frac{1}{2}\PP_{\theta}(Y_1 \in (3/4,1-\varepsilon) )\\
  &= \frac{1}{2}\Big( \eta\cdot (1-4\varepsilon) + \eta\cdot (1-4\varepsilon)\Big)\\%
    &= \eta(1-4\varepsilon).
\end{align*}

We now investigate $\EE_{\theta}\left[\min_{\tau}U_{n,\tau}(h^{\star}_{\theta})\right]$, for the previous Bayes classifier $h^{\star}_{\theta}$ (which does not necessarily minimize $h \mapsto \EE_{\theta}\left[\min_{\tau}U_{n,\tau}(h)\right]$). We rewrite,
\begin{align*}
  \EE_{\theta}\left[\min_{\tau}U_{n,\tau}(h^{\star}_{\theta})\right]%
  &= \EE_{\theta}\left[\min_{\tau}U_{n,\tau}(h^{\star}_{\theta})\1_{\sum_{i=1}^n\1_{Y_i \in (3/4-\varepsilon,3/4)\cup (1-\varepsilon,1) } = 0}\right]\\%
   &\quad + \sum_{m=1}^n \EE_{\theta}\left[\min_{\tau}U_{n,\tau}(h^{\star}_{\theta})\1_{\sum_{i=1}^n\1_{Y_i \in (3/4-\varepsilon,3/4)\cup (1-\varepsilon,1) } = m}\right]
\end{align*}

Let first consider $\EE_{\theta}[\min_{\tau}U_{n,\tau}(h^{\star}_{\theta}) \mid Y_{1:n}]$ on the event that $\{ \sum_{i=1}^n\1_{Y_i \in (3/4-\varepsilon,3/4)\cup (1-\varepsilon,1) } = 0 \}$. Let define $N = \sum_{i=1}^n\1_{X_i\ne 1}$:
\begin{itemize}
  \item if $N = 0$ this means that all the $Y_i$ are in $(0,1/2)$ and our classifier $h^{\star}_{\theta}$ combined with the identity permutation will make zero error, \textit{i.e.} $\min_{\tau}U_{n,\tau}(h^{\star}_{\theta}) = 0$ on this event;

  \item if $N = 1$ this means that there is only one $j \in \{1,\dots,n\}$ such that $Y_j \in (3/4,1-\varepsilon)$ [by assumption it can not be in $(3/4-\varepsilon, 3/4)$ or $(1-\varepsilon,1)$]. Our classifier will predict $h^{\star}_{\theta, i} = 1$ for all $i\ne j$ and $h^{\star}_{\theta, j}=3$. Now, necessarily $X_i = 1$ for $i\ne j$. If $X_j = 3$ then $h^{\star}_{\theta} \circ \mathrm{Id}$ will have loss zero, and if $X_j = 2$ then $h^{\star}_{\theta} \circ
  \begin{pmatrix}
    1 & 2 & 3\\
    1 & 3 & 2\\
  \end{pmatrix}
  $ will have loss zero. So in the event that $\{N=1\}$ we also have that $\min_{\tau}U_{n,\tau}(h^{\star}_{\theta}) = 0$.
  \item If $N \geq 2$ our classifier will still classify perfectly all the $Y_i \in (0,1/2)$ so the loss can not exceed $\min_{\tau}U_{n,\tau}(h^{\star}_{\theta}) \leq N/n$ in this case.
\end{itemize}
So on the event $\{ \sum_{i=1}^n\1_{Y_i \in (3/4-\varepsilon,3/4)\cup (1-\varepsilon,1) } = 0 \}$:
\begin{align*}
  \EE_{\theta}\left[\min_{\tau}U_{n,\tau}(h^{\star}_{\theta}) \mid Y_{1:n}\right]
  &\leq \EE_{\theta}\Big[ \frac{N}{n}\1_{N\geq 2}\mid Y_{1:n}\Big]  \\
  &\leq \sum_{k = 2}^n \frac{k}{n} \PP_{\theta}\Big(N = k\mid Y_{1:n} \Big) 
\end{align*}
But under the law of  $X_{1:n} \mid Y_{1:n}$ in the considered event, we have that $N$ is almost-surely equal to the number of $Y_i \in (3/4,1-\varepsilon)$, so
\begin{align*}
  \EE_{\theta}\left[\min_{\tau}U_{n,\tau}(h^{\star}_{\theta}) \mid Y_{1:n}\right]%
  &\leq \sum_{k=2}^n\frac{k}{n}\1_{\sum_{i=1}^n\1_{Y_i \in (3/4,1-\varepsilon)} = k}\\
  &= \Big(\frac{1}{n}\sum_{i=1}^n\1_{Y_i\in (3/4,1-\varepsilon)}\Big)\1_{\sum_{i=1}^n\1_{Y_i \in (3/4,1-\varepsilon)} \geq 2}.
\end{align*}
Deduce that,
\begin{multline*}
  \EE_{\theta}\left[\min_{\tau}U_{n,\tau}(h^{\star}_{\theta})\1_{\sum_{i=1}^n\1_{Y_i \in (3/4-\varepsilon,3/4)\cup (1-\varepsilon,1) } = 0}\right]\\
  \leq \EE_{\theta}\Bigg[\Big(\frac{1}{n}\sum_{i=1}^n\1_{Y_i\in (3/4,1-\varepsilon)}\Big)\1_{\sum_{i=1}^n\1_{Y_i \in (3/4,1-\varepsilon)} \geq 2} \mid \sum_{i=1}^n\1_{Y_i \in (3/4-\varepsilon,3/4)\cup (1-\varepsilon,1) } = 0 \Bigg]
\end{multline*}
Conditional on $\{ \sum_{i=1}^n\1_{Y_i \in (3/4-\varepsilon,3/4)\cup (1-\varepsilon,1) } = 0 \}$, the random variable $\sum_{i=1}^n\1_{Y_i\in (3/4,1-\varepsilon)}$ has a $\mathrm{Binomial}(n,2\eta)$ distribution. Then,
\begin{align*}
  \EE_{\theta}\left[\min_{\tau}U_{n,\tau}(h^{\star}_{\theta})\1_{\sum_{i=1}^n\1_{Y_i \in (3/4-\varepsilon,3/4)\cup (1-\varepsilon,1) } = 0}\right]%
  &\leq \frac{1}{n}\sum_{k=2}^n\binom{n}{k} k \cdot (2\eta)^{k}(1-2\eta)^{n-k}\\
  &= \frac{2\eta(1 - (1-2\eta)^n - 2\eta)}{1-2\eta}\\
  &\asymp 4(n-1)\eta^2
\end{align*}
when $\eta \ll 1/n$.

Next (remark that this can be largely improved, but this is indeed for our purpose),
\begin{multline*}
\sum_{m=1}^n \EE_{\theta}\left[\min_{\tau}U_{n,\tau}(h^{\star}_{\theta})\1_{\sum_{i=1}^n\1_{Y_i \in (3/4-\varepsilon,3/4)\cup (1-\varepsilon,1) } = m}\right]\\
\begin{aligned}
&\leq \PP_{\theta}\Bigg(\sum_{i=1}^n\1_{Y_i \in (3/4-\varepsilon,3/4)\cup (1-\varepsilon,1) } \geq 1 \Bigg)\\
  &= 1 - \PP_{\theta}\Bigg(\sum_{i=1}^n\1_{Y_i \in (3/4-\varepsilon,3/4)\cup (1-\varepsilon,1) } = 0 \Bigg)\\
  &= 1 - \PP_{\theta}\Big(\forall i,\ Y_i \notin (3/4-\varepsilon,3/4)\cup (1-\varepsilon,1) \Big)\\
  &= 1 - \Big( (1-2\eta) + 2\eta( 1-4\varepsilon ) \Big)^n\\
  &= 1 - (1 - 8\eta\varepsilon)^n\\
  &\asymp 8n\eta\varepsilon
\end{aligned}
\end{multline*}
when $\eta \ll n$. So by choosing $\varepsilon \asymp \eta$, we have shown that whenever $\eta \ll 1/n$
\begin{equation*}
  \inf_{h}\EE_{\theta}\left[\min_{\tau}U_{n,\tau}(h)\right]\leq \EE_{\theta}\left[\min_{\tau}U_{n,\tau}(h^{\star}_{\theta})\right] \lesssim n\eta^2
\end{equation*}
but
\begin{equation*}
  \inf_{h}\EE_{\theta}\left[\min_{\tau}\EE_{\theta}\left[U_{n,\tau}(h) \mid Y_{1:n}\right]\right] = \EE_{\theta}\left[\min_{\tau}\EE_{\theta}\left[U_{n,\tau}(h^{\star}_{\theta}) \mid Y_{1:n}\right]\right] \sim \eta
\end{equation*}
so that
\begin{equation*}
  \frac{\inf_{h}\EE_{\theta}\left[\min_{\tau}U_{n,\tau}(h)\right]}{\inf_{h}\EE_{\theta}\left[\min_{\tau}\EE_{\theta}\left[U_{n,\tau}(h) \mid Y_{1:n}\right]\right]}%
  \lesssim n \eta
\end{equation*}
which goes to zero as $\eta \to 0$.

\section{Proof of \texorpdfstring{\cite[Theorem~\ref*{main-thm:bayes-risk:key_quantity}]{GKN2025main}}{[4, Theorem 12]}}
\label{sec:theo3}
Simple computations lead to the expression of the Bayes risk of classification:
\begin{equation*}
    \inf_{h \in \mathcal{H}_n}\ClassifRisk(\theta,h) = \frac{1}{n}\sum_{i=1}^{n}\EE_{\theta}\left[\min_{x_0\in\mathbb{X}}\PP_{\theta}\left(X_{i} \neq x_0 \mid Y_{1:n}\right)\right] =  \frac{1}{n}\sum_{i=1}^{n}\EE_{\theta}\left[\min_{x_0\in\mathbb{X}}\left(\sum_{x\neq x_{0}}\phi_{\theta, i|n} (x)\right)\right].
\end{equation*}

    \subsection{Bounds for the independent scenario}
First, for $\theta\in\Theta^{\mathrm{ind}}$,
\begin{eqnarray*}
\inf_{h \in \mathcal{H}_n}\ClassifRisk(\theta,h) &=& \EE_{\theta}\left[\min_{x_0\in\mathbb{X}}\PP_{\theta}\left(X_{1} \neq x_0 \mid Y_{1}\right)\right]\\
&=&\EE_{\theta}\left[\min_{x_0\in\mathbb{X}}\frac{\sum_{x\neq x_{0}}\nu_{x}f_{x}(Y_{1})}{\sum_{x}\nu_{x}f_{x}(Y_{1})}\right]\\
&=&\int \min_{x_0 \in \mathbb{X}} \sum_{x\neq x_{0}}\nu_{x}f_{x}(y)d\mathcal{L}(y)
\end{eqnarray*}
so that using \cite[Assumption~\ref*{main-Assumption_mixing}]{GKN2025main},
\begin{multline}
\label{eq:theo3:1}
\delta \int_{\mathbb{Y}}\min_{x_{0}\in\mathbb{X}}\left[\sum_{x\neq x_0}f_{x}(y)\right]d\mathcal{L}(y)
\leq \inf_{h \in \mathcal{H}_n}\ClassifRisk(\theta,h)\\
\leq (1 - (J-1)\delta)\int_{\mathbb{Y}}\min_{x_{0}\in\mathbb{X}}\left[\sum_{x\neq x_0}f_{x}(y)\right]d\mathcal{L}(y).
\end{multline}
\subsection{Bounds for the dependent scenario}
Let $\theta\in\Theta^{\mathrm{dep}}$. We first have 
$$
     \inf_{h\in \mathcal{H}_n}\ClassifRisk(\theta,h) 
     = \frac{1}{n}\sum_{i=1}^{n}\EE_{\theta}\left[\min_{x_0\in\mathbb{X}}\sum_{x\neq x_{0}}\phi_{\theta, i|1:n}(x)\right] \leq \frac{1}{n}\sum_{i=1}^{n}\EE_{\theta}\left[\min_{x_0\in\mathbb{X}}\sum_{x\neq x_{0}}\phi_{\theta, i}(x)\right].
$$
Now, for any $i\leq n-1$, using  the Backward recursions (see Proposition 3.3.9 in \cite{CMT05}),
$$
\EE_{\theta}\left[\min_{x_0\in\mathbb{X}}\sum_{x\neq x_{0}}\phi_{\theta, i|1:n}(x)\right] = \EE_{\theta}\left[\min_{x_0\in\mathbb{X}}\sum_{x\neq x_0}\sum_{x^{\prime}\in\mathbb{X}}\phi_{\theta, i+1|1:n}(x^{\prime})B_{\theta, i}[Y_{1:i}](x^{\prime}, x)\right].
$$
Now, using \cite[Assumption~\ref*{main-Assumption_mixing}]{GKN2025main} and Equation \eqref{Backward_kernel}, we get
$$
B_{\theta, i}[Y_{1:i}](x^{\prime}, x) = \frac{\phi_{\theta, i}(x)Q(x, x^{\prime})}{\sum_{\tilde{x}\in\mathbb{X}}\phi_{\theta, i}(\Tilde{x})Q(\Tilde{x}, x^{\prime})} \geq \frac{\phi_{\theta, i}(x)Q(x, x^{\prime})}{1 - (J-1)\delta}
$$
so that
\begin{equation*}
    \begin{split}
\EE_{\theta}\left[\min_{x_0\in\mathbb{X}}\sum_{x\neq x_{0}}\phi_{\theta, i|1:n}(x)\right] 
        &\geq \frac{\delta}{1 - (J-1)\delta}\EE_{\theta}\left[\min_{x_{0}\in\mathbb{X}}\sum_{x^{\prime}}\phi_{\theta, i+1}(x^{\prime})\sum_{x\neq x_0}\phi_{\theta, i}(x)\right]\\
               &\geq \frac{\delta}{1 - (J-1)\delta}\EE_{\theta}\left[\min_{x_{0}\in\mathbb{X}}\sum_{x\neq x_0}\phi_{\theta,i}(x)\right]\\
        &= \frac{\delta}{1 - (J-1)\delta}\EE_{\theta}\left[\min_{x_0\in\mathbb{X}}\sum_{x\neq x_{0}}\phi_{\theta, i}(x)\right].
    \end{split}
\end{equation*}
Now, for $i=n$, the same inequality obviously holds, so that we get
$$
\frac{\delta}{1 - (J-1) \delta}
\frac{1}{n}\sum_{i=1}^{n}\EE_{\theta}\left[\min_{x_0\in\mathbb{X}}\sum_{x\neq x_{0}}\phi_{\theta, i}(x)\right] \leq \inf_{h\in \mathcal{H}_n}\ClassifRisk(\theta,h) \leq \frac{1}{n}\sum_{i=1}^{n}\EE_{\theta}\left[\min_{x_0\in\mathbb{X}}\sum_{x\neq x_{0}}\phi_{\theta, i}(x)\right].
$$
It suffices then to exhibit upper and lower bounds on $\sum_{i=1}^{n}\EE_{\theta}\left[\min_{x_0\in\mathbb{X}}\sum_{x\neq x_{0}}\phi_{\theta, i}(x)\right]$. 
Using the Forward recursions (see Equation (3.22) in Proposition 3.2.5 of \cite{CMT05}), for any $i\geq 2$,
$$
\EE_{\theta}\left[\min_{x_0\in\mathbb{X}}\left(\sum_{x\neq x_{0}}\phi_{\theta, i} (x)\right)\right]=
 \EE_{\theta}\left[\min_{x_0\in\mathbb{X}}\left(\frac{\sum_{x\neq x_0}\sum_{x^{\prime}\in\mathbb{X}}Q_{x^{\prime}, x}\phi_{\theta, i-1}(x^{\prime})f_{x}(Y_{i})}{\sum_{x\in\mathbb{X}}\sum_{x^{\prime}\in\mathbb{X}}Q_{x^{\prime}, x}\phi_{\theta, i-1}(x^{\prime})f_{x}(Y_{i})}\right)\right].
 $$
Let $A\in\mathcal{Y}$. One has:
\begin{equation*}
    \begin{split}
  \PP_{\theta}\left(Y_{i}\in A \mid Y_{1:i-1}\right) &= \sum_{x\in\mathbb{X}}\PP_{\theta}\left(Y_{i}\in A, X_{i} = x \mid Y_{1:i-1}\right)\\
        &= \sum_{x\in\mathbb{X}} \PP_{\theta}\left(Y_{i}\in A \mid X_{i} = x, Y_{1:i-1}\right)\PP_{\theta}\left(X_{i} = x \mid Y_{1:i-1}\right)\\
        &= \sum_{x\in\mathbb{X}} \PP_{\theta}\left(X_{i} = x \mid Y_{1:i-1}\right)\int_{A}f_{x}(y)d\mathcal{L}(y)\\
        &= \sum_{x\in\mathbb{X}}\sum_{x^{\prime}\in\mathbb{X}} \PP_{\theta}\left(X_{i-1} = x^{\prime}, X_{i} = x \mid Y_{1:i-1}\right)\int_{A}f_{x}(y)d\mathcal{L}(y)\\
        &=\sum_{x\in\mathbb{X}}\sum_{x^{\prime}\in\mathbb{X}} \PP_{\theta}\left(X_{i} = x \mid X_{i-1} = x^{\prime}, Y_{1:i-1}\right)\phi_{\theta, i-1}(x^{\prime})\int_{A}f_{x}(y)d\mathcal{L}(y)\\
        &= \int_{A}\left(\sum_{x\in\mathbb{X}}\sum_{x^{\prime}\in\mathbb{X}} Q_{x^{\prime}, x}\phi_{\theta, i-1}(x^{\prime})f_{x}(y)\right)d\mathcal{L}(y),\\
    \end{split}
\end{equation*}
so that, conditionally on $Y_{1:i-1}$, $Y_i$ has density $\sum_{x\in\mathbb{X}}\sum_{x^{\prime}\in\mathbb{X}} Q_{x^{\prime}, x}\phi_{\theta, i-1}(x^{\prime})f_{x}$  with respect to the dominating measure $\mathcal{L}$. We thus get :
\begin{eqnarray*}
\EE_{\theta}\left[\min_{x_0\in\mathbb{X}}\left(\sum_{x\neq x_{0}}\phi_{\theta, i} (x)\right)\right]&=& \EE_{\theta}\left[\EE_{\theta}\left[\min_{x_0\in\mathbb{X}}\left(\frac{\sum_{x\neq x_0}\sum_{x^{\prime}\in\mathbb{X}}Q_{x^{\prime}, x}\phi_{\theta, i-1}(x^{\prime})f_{x}(Y_{i})}{\sum_{x\in\mathbb{X}}\sum_{x^{\prime}\in\mathbb{X}}Q_{x^{\prime}, x}\phi_{\theta, i-1}(x^{\prime})f_{x}(Y_{i})}\right)\bigg| Y_{1:i-1}\right]\right]\\
        &=& \EE_{\theta}\left[\int_{\mathbb{Y}}\min_{x_0\in\mathbb{X}}\left(\sum_{x\neq x_0}\sum_{x^{\prime}\in\mathbb{X}}Q_{x^{\prime}, x}\phi_{\theta, i-1}(x^{\prime})f_{x}(y)\right)d\mathcal{L}(y)\right].
\end{eqnarray*}
Then, under \cite[Assumption~\ref*{main-Assumption_mixing}]{GKN2025main}, for any $i\geq 2$,
\begin{multline}
\label{eq:theo3:2}
\delta \int_{\mathbb{Y}}\min_{x_{0}\in\mathbb{X}}\left[\sum_{x\neq x_0}f_{x}(y)\right]d\mathcal{L}(y) \leq \EE_{\theta}\left[\min_{x_0\in\mathbb{X}}\left(\sum_{x\neq x_{0}}\phi_{\theta, i} (x)\right)\right]\\
\leq (1 - (J-1)\delta)\int_{\mathbb{Y}}\min_{x_{0}\in\mathbb{X}}\left[\sum_{x\neq x_0}f_{x}(y)\right]d\mathcal{L}(y),
\end{multline}
The result follows.

\section{Proof of \texorpdfstring{\cite[Theorem~\ref*{main-thm:excess_risk:clust:rate}]{GKN2025main}}{[4, Theorem 13]}}
\label{sec:thm:clust:rate}

We first control the excess risk of classification by the errors made in the estimation of the parameters.
\begin{proposition}\label{prop:ub:excess:hmm}
    For all $\theta\in\Theta^{\mathrm{dep}}$ satisfying \cite[Assumptions~\ref*{main-Assumption_mixing}, \ref*{main-Assumption_stat} and~\ref*{main-Assumption_3}]{GKN2025main} and for all $n\geq 1$,
\begin{multline*}
    \ClassifRisk(\theta, \hat{h}) - \inf_{h\in \mathcal{H}_n}\ClassifRisk(\theta,h) \leq C\EE_{\theta}\left[\frac{1}{n}\lVert\nu - \hat{\nu}\rVert_{2} + \left(1 + \frac{\delta}{\hat{\delta}} \right)\lVert Q - \hat{Q} \rVert_{\mathrm{F}}\right]\\ + \frac{\delta C\sqrt{C^{\star}}}{n}\sum_{i=1}^{n}\sum_{l=1}^{n}\EE_{\theta}\left[(\hat{\rho}\vee \rho)^{2|l-i|}\|f_{x} - \hat{f}_{x}\|_{\infty}^{2}\right]^{1/2},
\end{multline*}
where $C = \frac{8(1-\delta)}{\delta^3}$, $\rho = \frac{1 - 2\delta}{1 - \delta}$, $\hat{\rho} = \frac{1 - 2\hat{\delta}}{1 - \hat{\delta}}$ and $\hat{\delta} = \min_{x,x'}\hat{Q}_{x,x'}$.
\end{proposition}
\begin{proof}
Recall that $\phi_{\theta,i\mid n}(\cdot) = \PP_{\theta}(X_i \in \cdot \mid Y_{1:n})$. Then,

\begin{multline*}
\ClassifRisk(\theta, \hat{h}) - \inf_{h\in \mathcal{H}_n}\ClassifRisk(\theta,h) \\
\begin{aligned}
&= \EE_{\theta}\Big[\frac{1}{n}\sum_{i=1}^{n} \1_{X_{i}\neq \hat{h}_{i}}\Big] - \EE_{\theta}\Big[\frac{1}{n}\sum_{i=1}^{n} \1_{X_{i}\neq h^{\star}_{\theta, i}}\Big]\\
&= \frac{1}{n}\sum_{i=1}^n\EE_{\theta}\Big[ \PP_{\theta}( X_i  = h^{\star}_{\theta, i} \mid Y_{1:n} ) - \PP_{\theta}(X_i = \hat{h}_{i} \mid Y_{1:n}) \Big]\\
&\leq \frac{1}{n}\sum_{i=1}^n\EE_{\theta}\Big[ \PP_{\theta}( X_i  = h^{\star}_{\theta, i} \mid Y_{1:n} ) - \PP_{\hat\theta}(X_i = \hat{h}_{i} \mid Y_{1:n}) + \|\phi_{\hat\theta,i\mid n} - \phi_{\theta,i\mid n}\|_{\mathrm{TV}} \Big]\\
&\leq \frac{1}{n}\sum_{i=1}^n\EE_{\theta}\Big[ \PP_{\theta}( X_i  = h^{\star}_{\theta, i} \mid Y_{1:n} ) - \PP_{\hat\theta}(X_i = h^{\star}_{\theta, i} \mid Y_{1:n}) + \|\phi_{\hat\theta,i\mid n} - \phi_{\theta,i\mid n}\|_{\mathrm{TV}} \Big]\\
&\leq \frac{2}{n}\sum_{i=1}^n\EE_{\theta}\Big[ \|\phi_{\hat\theta,i\mid n} - \phi_{\theta,i\mid n}\|_{\mathrm{TV}}\Big]
\end{aligned}
\end{multline*}
where the penultimate line follows because by definition $\hat{h}_i$ maximizes $x \mapsto \PP_{\hat\theta}(X_i = x \mid Y_{1:n})$. Then, under \cite[Assumption~\ref*{main-Assumption_3}]{GKN2025main} and by application of Proposition 2.2 of \cite{YES17} (see also \cite[Equation~(\ref*{main-eq:smoothing-control})]{GKN2025main}), one has

\begin{equation*}
    \begin{split}
    &\ClassifRisk(\theta, \hat{h}) - \inf_{h\in \mathcal{H}_n}\ClassifRisk(\theta,h) \\
    &\leq \frac{8(1 - \delta)}{\delta^{2}}\EE_{\theta}\Biggl[\frac{1}{n\delta}\lVert\nu - \hat{\nu}\rVert_{2} +\left(\frac{1}{\delta} + \frac{1}{\hat{\delta}} \right)\lVert Q - \hat{Q} \rVert_{\mathrm{F}} + \frac{1}{n}\sum_{i=1}^{n}\sum_{l=1}^{n}\delta\frac{(\hat{\rho}\vee \rho)^{|l-i|}}{c^{\star}(Y_l)}\max_{x\in\mathbb{X}}\|f_{x} - \hat{f}_{x}\|_{\infty}\Biggr]\\
    &\leq C\EE_{\theta}\left[\frac{1}{n}\lVert\nu - \hat{\nu}\rVert_{2} + \left(1 + \frac{\delta}{\hat{\delta}} \right)\lVert Q - \hat{Q} \rVert_{\mathrm{F}} + \frac{\delta^{2}}{n}\sum_{i=1}^{n}\sum_{l=1}^{n}\frac{(\hat{\rho}\vee \rho)^{|l-i|}}{c^{\star}(Y_l)}\max_{x\in\mathbb{X}}\|f_{x} - \hat{f}_{x}\|_{\infty}\right]\\
    &\leq C\EE_{\theta}\left[\frac{1}{n}\lVert\nu - \hat{\nu}\rVert_{2} + \left(1 + \frac{\delta}{\hat{\delta}} \right)\lVert Q - \hat{Q} \rVert_{\mathrm{F}}\right] + \frac{\delta^{2} C}{n}\sum_{i=1}^{n}\sum_{l=1}^{n}\EE_{\theta}\left[\frac{(\hat{\rho}\vee \rho)^{|l-i|}}{c^{\star}(Y_l)}\max_{x\in\mathbb{X}}\|f_{x} - \hat{f}_{x}\|_{\infty}\right]\\
    &\leq C\EE_{\theta}\left[\frac{1}{n}\lVert\nu - \hat{\nu}\rVert_{2} + \left(1 + \frac{\delta}{\hat{\delta}} \right)\lVert Q - \hat{Q} \rVert_{\mathrm{F}}\right] + \frac{\delta^{2} C}{n}\sum_{i=1}^{n}\sum_{l=1}^{n}\EE_{\theta}\left[\frac{1}{c^{\star}(Y_l)^{2}}\right]^{1/2}\EE_{\theta}\left[(\hat{\rho}\vee \rho)^{2|l-i|}\|f_{x} - \hat{f}_{x}\|_{\infty}^{2}\right]^{1/2}\\
    &\leq C\EE_{\theta}\left[\frac{1}{n}\lVert\nu - \hat{\nu}\rVert_{2} + \left(1 + \frac{\delta}{\hat{\delta}} \right)\lVert Q - \hat{Q} \rVert_{\mathrm{F}}\right] + \frac{\delta C\sqrt{C^{\star}}}{n}\sum_{i=1}^{n}\sum_{l=1}^{n}\EE_{\theta}\left[(\hat{\rho}\vee \rho)^{2|l-i|}\|f_{x} - \hat{f}_{x}\|_{\infty}^{2}\right]^{1/2}
    \end{split}
\end{equation*}
\end{proof}
In order to obtain a rate on the excess risk, we make use of Algorithm~\ref{Spectral_algorithm} which will yield the estimates used in the statement of the corollary. This algorithm merges the spectral algorithms of \cite{YES17, ACG21} with some slight modifications. Note that all the expectations and probabilities of this proof are with respect to the observations and the random unit matrices. Also note that the algorithm outputs estimates of the densities that are not necessarily \textit{bona-fide} densities. This is not problematic for the plug-in procedure as one typically uses the Forward-Backward algorithm \cite{CMT05} which works even if the emissions are not correctly normalized.


\begin{algorithm}
\caption{Non-parametric spectral estimation of the transition matrix and the emission laws}\label{Spectral_algorithm}
\begin{algorithmic}
  \Input
  \begin{itemize}
  \item Number of states $J$, integers $D$ and $r$.
  \item Data $(Y_i)_{i\leq n+2}$ drawn from a HMM with $J$ states.
  \item Functions $(\varphi_{d})_{d\in\mathbb{N}}$ \textbf{uniformly bounded} such that $O = (\EE_{\theta}[\varphi_{d}(Y_1)\mid X_1 = j])_{1\leq d\leq D, 1\leq j\leq J}$ is of rank $J$ with $\sigma_J(O)$ \textbf{bounded away from 0 uniformly} in $D$, at least for $D$ large enough.
  \item $K$ a Lipschitz-continuous kernel
  \end{itemize}
  \EndInput
  \Output
  \begin{itemize}
  \item Spectral estimators $\hat{Q}$ and $(\hat{f}_j)_{1\leq j\leq J}$
  \end{itemize}
  \EndOutput
  \Estimation
  \begin{itemize}
  \item[\textbf{[Step 1]}] For all $a, b, c \in\llbracket 1, D \rrbracket$, consider the following empirical estimators:
\begin{equation*}
\begin{split}
\hat{L} (a) &= \frac{1}{n}\sum_{s=1}^{n}\varphi_a (Y_s)\\
\hat{N}(a, b) &= \frac{1}{n}\sum_{s=1}^{n}\varphi_{a}(Y_s)\varphi_b(Y_{s+1})\\
\hat{P}(a, c) &= \frac{1}{n}\sum_{s=1}^{n}\varphi_a(Y_s)\varphi_c(Y_{s+2})\\
\hat{M}(a, b, c) &= \frac{1}{n}\sum_{s=1}^{n}\varphi_a (Y_s)\varphi_b (Y_{s+1})\varphi_c (Y_{s+2}) \\
\hat{M}^{x, L}(a, b) &= \frac{1}{n}\sum_{s=1}^{n}\varphi_a(Y_s)K_{L}(x, Y_{s+1})\varphi_b(Y_{s+2})
\end{split}
\end{equation*}  
 where $L$ is chosen such that: $2^{L} \asymp (n/\log(n))^{1/(2s+1)}$ and $s$ is the smoothness of the emissions.
   \item[\textbf{[Step 2]}] Let $\hat{V}$ be the $D\times J$ matrix of orthonormal right singular vectors of $\hat{P}$ corresponding to its top $J$ singular values.
   \item[\textbf{[Step 3]}] For all $d\in\llbracket 1, D \rrbracket$, set $\hat{B}(d) = (\hat{V}^{\top}\hat{P}\hat{V})^{-1}\hat{V}^{\top}\hat{M}(., d, .)\hat{V}$
   \item[\textbf{[Step 4]}] Generate $\Omega$ a $J\times J$ unit matrix uniformly drawn, for all $x\in\llbracket 1, J\rrbracket$, $\hat{C}(x) = \sum_{d=1}^{D}(\hat{V}\Omega)(d, x)\hat{B}(d)$
   \item[\textbf{[Step 5]}] Compute $\hat{R}_{1}$ a $J\times J$ unit Euclidean norm columns matrix that diagonalizes the matrix $\hat{C}(1)$:
   \begin{equation*}
   \hat{R}_{1}^{-1}\hat{C}(1)\hat{R}_{1} = \mathcal{D}iag[(\hat{\Lambda}(1,1), ..., \hat{\Lambda}(1, J))]
   \end{equation*}
   \item[\textbf{[Step 6]}] For all $x, x^{\prime}\in\llbracket 1, J \rrbracket$, $\hat{\Lambda}(x, x^{\prime}) =(\hat{R}_{1}^{-1}\hat{C}(x)\hat{R}_{1})_{x^{\prime}, x^{\prime}}$.
   \item[\textbf{[Step 7]}] Repeat steps 4 to 6 r times and take $\Omega_{r}$ maximizing $i\mapsto\min_{k\leq J}\min_{k_1 \neq k_2}|\hat{\Lambda}_{i}(k, k_1) - \hat{\Lambda}_{i}(k, k_2)|$
   \item[\textbf{[Step 8]}] Set $\hat{O} = \hat{V}\Omega_{r}\hat{\Lambda}$, $\tilde{\nu} = (\hat{V}^{\top}\hat{O})^{-1}\hat{V}^{\top}\hat{L}$ and  
$\hat{Q} = \Pi_{TM}\left((\hat{V}^{\top}\hat{O}\mathcal{D}iag[\tilde{\nu}])^{-1}\hat{V}^{\top}\hat{N}\hat{V}(\hat{O}^{\top}\hat{V})^{-1}\right)$ where $\Pi_{TM}$ denotes the projection (with respect to the scalar product given by the Frobenius norm) onto the convex set of transition matrices.
	\item[\textbf{[Step 9]}] For $x\in\mathbb{R}$, set $\hat{B}^{x} = \hat{B}^{x, D, L} = (\hat{V}^{\top}\hat{P}\hat{V})^{-1}\hat{V}^{\top}\hat{M}^{x, L}\hat{V}$
	\item[\textbf{[Step 10]}] Set $\hat{R}_{2} = \hat{Q}\hat{O}^{\top}\hat{V}$ and take $\tilde{f}_{j}(x) = (\hat{R}_{2}\hat{B}^{x}\hat{R}_{2}^{-1})_{j,j}$
	\item[\textbf{[Step 11]}] $\hat{f}_{j}(x) =  \left\{
    \begin{array}{ll}
        \tilde{f}_{j}(x) & \mbox{si } |\tilde{f}_{j}(x)|\leq n^{\beta} \\
        n^{\beta}sign(\tilde{f}_{j}(x)) & \mbox{otherwise}
    \end{array}
\right.$
for $\beta > 0$ fixed (but arbitrary). 
\end{itemize}
   \EndEstimation
\end{algorithmic}
\end{algorithm}


First, we start by controlling $\EE_{\theta}\left[\lVert Q - \hat{Q}\rVert_{\mathrm{F}}^{2}\right]$, $\EE_{\theta}\left[\frac{1}{\hat{\delta}^{2}}\right]$ and $\EE_{\theta}\left[\max_{x\in\mathbb{X}}\lVert f_{x} - \hat{f}_{x}\rVert_{\infty}^{2}\right]$ using the estimates yielded by the algorithm. Thanks to step 7 of the algorithm, one can obtain a slightly different version of Theorem~3.1 of \cite{YES17} (Note that this version is used in the proof of Corollary 3.2 in \cite{YES17}). It ensures the existence of positive constants $C, x_{0}, y_{0}, D_{0}$ and $n_{1}$ such that for all $D \geq D_{0}$ there exist a permutation $\tau_D\in\mathcal{S}_{J}$ such that for all $n \geq n_{1}\eta_{3}^{2}(\varphi_{D})x(y + \log(r))e^{y/r}$, $x \geq x_{0}$, $y \geq y_{0}$ and $r \geq 1$, with probability at least $1 - 4e^{-x} - 2e^{-y}$:
\begin{equation}\label{improved_control}
	\begin{split}
	\lVert Q^{\tau_D} - \hat{Q}\rVert_{\mathrm{F}}^{2} &\leq C\eta_{3}^{2}(\varphi_{D})x(y + \log(r))e^{y/r}/n\\
	\max_{x\in\mathbb{X}}\lVert f_{D,\tau_{D}(x)} - \hat{f}_{D, x}^{(r)}\rVert^{2}_{2} &\leq C\eta_{3}^{2}(\varphi_{D})x(y + \log(r))e^{y/r}/n
	\end{split}
\end{equation}
where:
\begin{itemize}
    \item $\left(\varphi_{k}\right)_{k\in\mathbb{N}}$ is the basis used in Algorithm \ref{Spectral_algorithm}
    \item $\eta_{3}^{2}(\varphi_D) = \sup_{y, y^{\prime}\in\mathcal{Y}^{3}}\sum_{a, b, c = 1}^{D}\left(\varphi_{a}(y_1)\varphi_{b}(y_2)\varphi_{c}(y_3) - \varphi_{a}(y_{1}^{\prime})\varphi_{b}(y_{2}^{\prime})\varphi_{c}(y_{3}^{\prime})\right)^{2}$
    \item $f_{D, x} = \sum_{d=1}^{D}\langle f_{x}, \varphi_{d}\rangle\varphi_{d}$ the projection of the density $f_x$ on the subspace spanned by the first $D$ components of the basis.
    \item $\hat{f}_{D, x}^{(r)} = \sum_{d=1}^{D}\hat{O}_{d, x}\varphi_{d}$ where $\hat{O}$ is the matrix constructed at step 8 of Algorithm \ref{Spectral_algorithm}.
\end{itemize}
As detailed in \cite{YES17}, when using a wavelet basis or trigonometric polynomials basis, $\eta_3(\varphi_{D})$ ensures for a constant $C_{\eta} > 0$: 
\begin{equation}\label{eta_bound}
    \eta_{3}(\varphi_{D})\leq C_{\eta}D^{3/2} \text{ and }  \max_{l\in\mathbb{N}}\lVert \varphi_l\rVert_{\infty} < \infty 
\end{equation}
We assume a similar basis is used.

It is important to note that the estimator $\hat{f}^{(r)}_{D, x}$ is not the one yielded by the Algorithm~\ref{Spectral_algorithm} but it is rather the one used in \cite{YES17}. We do not use it for the estimation because it does not allow obtaining the appropriate rate in infinite norm (see \cite{YES17} for more details). However, we will use it in our proof because $\lVert\hat{O}(., k) - O(., \tau_{D_n}(k))\rVert_{2} = \lVert \hat{f}_{D_{n}, k}^{(r)} - f_{D_{n},\tau_{D_n}(k)}\rVert_{2}$.

Assume that the parameters $x = x_n$, $y = y_n$, $D = D_n$ and $r = r_n$ are increasing with respect to $n$ and that $n \geq n_{1}C_{\eta}^{2}D_{n}^{3}x_n(y_n + \log(r_n))e^{y_n/r_n}$.

\paragraph*{Control of $\EE_{\theta}\left[\lVert Q^{\tau_{D_n}} - \hat{Q}\rVert^{2}_{\mathrm{F}}\right]$}

The control of $\lVert Q^{\tau_{D_n}} - \hat{Q}\rVert^{2}_{\mathrm{F}}$ in expectation is already proved in Corollary 3.2 of \cite{YES17} using Inequality \eqref{improved_control}. The proof chooses $r_{n} \propto \log(n)$ and $\eta_{3}(\varphi_{D_{n}}) = o(\sqrt{n}/\log(n))$ and yields:
$$
\EE_{\theta}\left[\lVert Q^{\tau_{D_n}} - \hat{Q}\rVert^{2}_{\mathrm{F}}\right] = \mathcal{O}(\eta_{3}^{2}(\varphi_{D_{n}})\log(n)/n) = \mathcal{O}(D_{n}^{3}\log(n)/n) \quad(\text{by } \eqref{eta_bound})
$$
for a sequence $(\tau_{D_n})_n$ of permutations.
We will keep the same values of $r_{n}$ and $D_{n}$ in what follows.

One of the advantages of this algorithm with respect to the previous versions is that it allows obtaining the appropriate rate on the errors in the estimation of all the model parameters. This is done thanks to the use of the kernel estimator of the emission densities for which the error of approximation is tuned (through the parameter $L$) independently of the error of estimation of the transition matrix. The shortcoming of the algorithm proposed in \cite{YES17} is that it does not allow controlling the rate on the emission densities without altering that of the transition matrix as it is clear in Corollary 3.3 of that paper.

\paragraph*{Control of $\EE_{\theta}\left[\frac{1}{\hat{\delta}^{2}}\right]$}

Let $\tilde{\delta} = \frac{\delta}{2}$. Then,
\begin{equation*}
\begin{split}
\EE_{\theta}\left[\frac{1}{\hat{\delta}^{2}}\right] &= \EE_{\theta}\left[\frac{1}{\hat{\delta}^{2}}\1_{\hat{\delta} < \tilde{\delta}}\right] + \EE_{\theta}\left[\frac{1}{\hat{\delta}^{2}}\1_{\hat{\delta} \geq \tilde{\delta}}\right]\\
&\leq \EE_{\theta}\left[\frac{1}{\hat{\delta}^{2}}\1_{\hat{\delta} < \tilde{\delta}}\right] + \frac{1}{\tilde{\delta}^{2}}
\end{split}
\end{equation*}
On the other hand:
\begin{equation*}
\tilde{\delta} - \hat{\delta} = \delta - \hat{\delta} - \frac{\delta}{2}\leq \left|\delta - \hat{\delta}\right|- \frac{\delta}{2}\leq \max_{i, j}\left|Q^{\tau_{D_n}}_{i, j} - \hat{Q}_{i, j}\right|- \frac{\delta}{2}\leq\lVert Q^{\tau_{D_n}} - \hat{Q}\rVert_{F} - \frac{\delta}{2}
\end{equation*}
where we have used the inequality $|\min_{i, j} Q_{i, j} - \min_{i, j} \hat{Q}_{i, j}| \leq \max_{i, j}| Q_{i, j} - \hat{Q}_{i, j}|$.

We assume that all the entries of $\hat{Q}$ are between $n^{-\alpha/2}$ and $1 - n^{-\alpha/2}$ for $\alpha \geq 2$. If it is not the case, modifying the entries of $\hat{Q}$ to obtain a similar property induces an error of order $n^{-\alpha/2}$ which is negligible with respect to the rates we seek and all the subsequent results remain unchanged. It follows that $\hat{\delta} \geq n^{-\alpha/2}$ and:
$$
\EE_{\theta}\left[\frac{1}{\hat{\delta}^{2}}\right] \leq n^{\alpha}\PP_{\theta}\left(\lVert Q^{\tau_{D_n}} - \hat{Q}\rVert_{F} > \frac{\delta}{2}\right) + \frac{1}{\tilde{\delta}^{2}}
$$
Choosing for example $x_{n} = y_{n} = r_{n} = \alpha\log(n)$, then one obtains for $n$ large enough:
\begin{equation*}
\begin{split}
\PP_{\theta}\left(\lVert Q^{\tau_{D_n}} - \hat{Q}\rVert_{F} > \frac{\delta}{2}\right) &\leq \PP_{\theta}\left(\lVert Q^{\tau_{D_n}} - \hat{Q}\rVert_{F} > C\eta_{3}^{2}(\varphi_{D_n})x_{n}(y_{n} + \log(r_n))e/n\right)\\
&\leq 4e^{-x_{n}} + 2e^{-y_n}
\end{split}
\end{equation*}
It follows that $\EE_{\theta}\left[\frac{1}{\hat{\delta}^{2}}\right]$ is upper-bounded by an absolute constant.

The values $x_n$, $y_n$ and $r_n$ will be kept the same in what follows.

\paragraph*{Control of $\EE_{\theta}\left[\max_{x\in\mathbb{X}}\lVert f_{\tau_{D_n}(x)} - \hat{f}_{x}\rVert_{\infty}^{2}\right]$}

The difficulty of the control of this quantity lies in ensuring that the same permutation $\tau_{D_n}$ used for the control of $Q$ still works for the control of the emission densities. In the spectral algorithm of \cite{ACG21}, the matrix $\hat{R}$ is chosen independently of $Q$ or $\hat{Q}$ (In fact, this algorithm does not even estimate $\hat{Q}$). Had we used this matrix, there would not be any reason for which the same permutation $\tau_{D_n}$ works for the control of the emission densities. To solve this problem, we choose a matrix $\hat{R}$ that depends explicitly on $\hat{Q}$ so that the same permutation that works for the control of $Q$ works also for that of the emission densities.
We follow here the steps of the proof of Theorem~5 in \cite{ACG21}.

Let $M^{x, L}, P, O$ be the quantities estimated by $\hat{M}^{x, L}, \hat{P}, \hat{O}$ and construct $\hat{f}_{j}$, $\tilde{f}_{j}$ using Algorithm~\ref{Spectral_algorithm}. Let $E_n$ be the event with probability greater than $1 - 4e^{-x_n} - 2e^{-y_n}$ on which the control of (\ref{improved_control}) holds when $x_n = y_n = r_n = \alpha\log(n)$ and $D = D_n$. For $\gamma > 0$ there exists $c = c(\gamma)$ such that the event:
$$
\mathcal{A}_{n} = \left\{\lVert\hat{P}-P \rVert \leq cD_{n}\left(\frac{\log(n)}{n}\right)^{\frac{s}{2s+1}},\quad \sup_{x\in\mathbb{R}}\lVert \hat{M}^{x, L} - M^{x, L}\rVert \leq cD_{n}^{2}\left(\frac{\log(n)}{n}\right)^{\frac{s}{2s+1}} \right\}
$$
is measurable and has probability $n^{-\gamma}$ (Cf. Lemma~25.a of \cite{ACG21}). Given that $E_n$ has probability greater than $1 - 4e^{-x_n} -2e^{-y_n} = 1 - 6n^{-\alpha}$, it follows that $\mathcal{A}_{n}\cap E_n$ has probability greater than $1 - n^{-\gamma}- 6n^{-\alpha}$. Note that the difference with the original proof is that we use the event $\mathcal{A}_n\cap\mathcal{E}_n$ instead of the event $\mathcal{A}_n$. This is compulsory to control the errors of $\hat{Q}$ and $\hat{f}$ simultaneously.

On the event $\mathcal{A}_{n}\cap E_n$, and at step 10 of Algorithm~\ref{Spectral_algorithm}, instead of using the matrix $\hat{R}$ appearing in the spectral algorithm in \cite{ACG21}, we rather use the matrix $\hat{R}_2 = \hat{Q}\hat{O}^{\top}\hat{V}$ where the components $\hat{V}$, $\hat{Q}$ and $\hat{O}$ are constructed in the Algorithm~\ref{Spectral_algorithm}. On the other hand, since the columns of $\hat{R}$ are not normalized, we choose $\tilde{R}_2 = QO^{\top}\hat{V}$ on the contrary to what is done in the proof of Theorem~5 in \cite{ACG21}.
By denoting $Q^{\tau_{D_n}} = P_{\tau_{D_n}} Q P_{\tau_{D_n}}^{-1}$, one obtains:
\begin{equation*}
\hat{R}_2 - P_{\tau_{D_n}}\tilde{R}_2 = \hat{Q}\hat{O}^{\top}\hat{V} - P_{\tau_{D_n}}QO^{\top}\hat{V} =  \left(\hat{Q}(\hat{O} - OP_{\tau_{D_n}}^{\top})^{\top} + (\hat{Q}-Q^{\tau_{D_n}})P_{\tau_{D_n}}O^{\top}\right)\hat{V}
\end{equation*}
It follows by using operator norm:
\begin{equation}\label{ineq2}
\begin{split}
\lVert \hat{R}_2 - P_{\tau_{D_n}}\tilde{R}_2\rVert &\leq \left(\lVert\hat{Q}\rVert\lVert \hat{O} - OP_{\tau_{D_n}}^{\top}\rVert + \lVert P_{\tau_{D_n}}O^{\top}\rVert\lVert \hat{Q} - Q^{\tau_{D_n}}\rVert\right)\lVert \hat{V}\rVert
\end{split}
\end{equation}
First, note that on the event $\mathcal{A}_{n}\cap E_n$:
\begin{equation*}
    \begin{split}
        \lVert\hat{Q}\rVert &\leq 1\\
        \lVert \hat{V}\rVert&\leq\lVert \hat{V}\rVert_{F} = J^{1/2} \quad\text{(columns are normalized)}\\
        \lVert \hat{Q} - Q^{\tau_{D_n}}\rVert &\leq \lVert \hat{Q} - Q^{\tau_{D_n}}\rVert_{\mathrm{F}}\\
        &\leq \left(CC_{\eta}^{2}D_{n}^{3}x_n(y_n + \log(r_n))e^{y_n/r_n}/n\right)^{1/2}\text{ (by \eqref{improved_control} and \eqref{eta_bound}})\\
        \lVert P_{\tau_{D_n}}O^{\top}\rVert &= \lVert O\rVert = \sup_{\lVert v \rVert = 1}\left(\sum_{j=1}^{J}\left(\sum_{l=1}^{D_n}v_l \langle f_{j}, \varphi_{l}\rangle\right)^{2}\right)^{1/2}\\
        &\leq (J D_n)^{1/2}\max_{1\leq l \leq D_n}\lVert \varphi_l\rVert_{\infty}\\
        \lVert \hat{O} - OP_{\tau_{D_n}}^{\top}\rVert &\leq  \lVert \hat{O} - OP_{\tau_{D_n}}^{\top}\rVert_{F} = \left(\sum_{k=1}^{K}\lVert\hat{O}(., k) - O(., \tau_{D_n}(k))\rVert_{2}^{2}\right)^{1/2}\\
        &= \left(\sum_{k=1}^{K}\lVert \hat{f}_{D_{n}, k}^{(r)} - f_{D_{n},\tau_{D_n}(k)}\rVert^{2}_{2}\right)^{1/2}\\
    &\leq\left(KCC_{\eta}^{2}D_{n}^{3}x_n(y_n + \log(r_n))e^{y_n/r_n}/n\right)^{1/2}\text{ (by \eqref{improved_control} and \eqref{eta_bound}}).
    \end{split}
\end{equation*}
By keeping the previous choices of $x_n, y_n$ and $r_n$ then by Inequality \eqref{ineq2}, there exists a constant $C^{\prime}$ such that:
$$
\1_{\mathcal{A}_{n}\cap E_n}\lVert \hat{R}_2 - P_{\tau_{D_n}}\tilde{R}_2\rVert \leq C^{\prime}D_n^{1/2}D_{n}^{3/2}\frac{\log(n)}{\sqrt{n}}.
$$
By Lemma~25.b of \cite{ACG21}, for $n$ large enough, $\hat{P}$ has rank $J$, $(\hat{V}^{T}\hat{P}\hat{V})$ and $(\hat{V}^{T}P\hat{V})$ are invertible and the matrices $(\hat{B}(d))_{1\leq d \leq D_n}$ appearing in Algorithm~\ref{Spectral_algorithm} are then well-defined. By Lemma~11 and 25.b of \cite{ACG21}, $\Tilde{B}^{x} = (\hat{V}^{T}P\hat{V})^{-1}\hat{V}^{T}M^{x}\hat{V}$ satisfies:
\begin{equation}
    \Tilde{B}^{x} = (QO^{T}\hat{V})^{-1}D^{x}(QO^{T}\hat{V}) = \Tilde{R}_{2}^{-1}D^{x}\Tilde{R}_{2}
\end{equation}
where $D^{x} = (K_{L}[f_j](x))_{j\leq J}$. Using the fact that $\sigma_{J}(\hat{V}) = 1$ (the columns of $\hat{V}$ are orthonormal) and $\sigma_{J}(QO^{T}) \geq \sigma_{J}(Q)\sigma_{J}(O) > 0$ (because $Q$ in full rank and $\sigma_{J}(O)$ is bounded from below by an absolute constant by assumption of the algorithm), it follows that $\lVert \Tilde{R}_{2}^{-1}\rVert^{-1}$ is bounded from below by an absolute constant because:
\begin{multline}
\label{R_bound}
    \lVert \Tilde{R}_{2}^{-1}\rVert^{-1} = \frac{1}{\sigma_{1}(\Tilde{R}_{2}^{-1})} = \sigma_{J}(\Tilde{R}_{2})\\
    = \sigma_{J}(QO^{T}\hat{V})\geq\sigma_{J}(QO^{T})\sigma_{J}(\hat{V}) = \sigma_{J}(QO^{T}) \geq \sigma_{J}(Q)\sigma_{J}(O) > 0.
\end{multline}
Thus, for $n$ large enough, the assumption of Lemma~\ref{Lemma_4} (stated below) is verified with $A_t = \tilde{B}^t$, $\hat{A}_t = \hat{B}^{t}$ , $R = \tilde{R}_2 = QO^{T}\hat{V}$. This ensures that:
\begin{multline*}
\1_{\mathcal{A}_{n}\cap E_n}\max_{x\in\mathbb{X}}\lVert\tilde{f}_{x} - K_{L}[f_{\tau_{D_n}(x)}]\rVert_{\infty}\\
\leq 4\kappa(\Tilde{R}_{2})\left[\1_{\mathcal{A}_{n}\cap E_n}\sup_{t}\lVert  \Tilde{B}^{t} - \hat{B}^{t}\rVert + \lambda_{\max}\kappa(\Tilde{R}_{2})\lVert \Tilde{R}^{-1}_{2}\rVert\1_{\mathcal{A}_{n}\cap E_n}\lVert \hat{R}_2 - P_{\tau_{D_n}}\tilde{R}_2\rVert\right]
\end{multline*}
where $\lambda_{\max} = \sup_{t}\max_{j}|\lambda_{t, j}| = \sup_{t}\max_{j}|K_{L}[f_{j}]| < \infty$ and $\lambda_{t, j}$ is the $j$-th eigenvalue of $\Tilde{B}^{t}$. By Lemma~35.c of \cite{ACG21}, one has for some constant $C$:
 $$
\1_{\mathcal{A}_{n}\cap E_n}\sup_{t}\lVert  \Tilde{B}^{t} - \hat{B}^{t}\rVert \leq CD_n^{2}\left(\frac{\log(n)}{n}\right)^{\frac{s}{2s+1}}.
 $$
By Lemma~35.b in \cite{ACG21}, one obtains: $\kappa(\Tilde{R}_{2})\leq \Tilde{C}D_{n}^{1/2}$ and $\lVert \Tilde{R}_{2}^{-1}\rVert$ upper-bounded by an absolute constant by \eqref{R_bound}. We finally obtain that:
\begin{align*}
\1_{\mathcal{A}_{n}\cap E_n}\max_{x\in\mathbb{X}}\lVert\tilde{f}_{x} - K_{L}[f_{\tau_{D_n}(x)}]\rVert_{\infty}
&\leq c^{\prime}D_{n}^{1/2}\left[D_n^{2}\left(\frac{\log(n)}{n}\right)^{\frac{s}{2s+1}} + D_n^{1/2} D_n^{1/2} D_n^{3/2}\frac{\log(n)}{\sqrt{n}}\right]\\
&\leq c^{\prime\prime}D_n^{5/2}\left(\frac{\log(n)}{n}\right)^{\frac{s}{2s+1}}
\end{align*}
for come constants $c^{\prime}, c^{\prime\prime}$. The choice of $L$ (cf. Algorithm~\ref{Spectral_algorithm}) allows obtaining then:
$$
\1_{\mathcal{A}_{n}\cap E_n}\max_{x\in\mathbb{X}}\lVert\tilde{f}_{x} - f_{\tau_{D_n}(x)}\rVert_{\infty}\leq c^{\prime\prime}D_n^{5/2}\left(\frac{\log(n)}{n}\right)^{\frac{s}{2s+1}}.
$$
Given that for $n$ large enough $\lVert f_{j}\rVert_{\infty}\leq n^{\beta}$,  it follows that:
$\lVert\hat{f}_j -  f_{\tau_{D_n}(j)}\rVert_{\infty} \leq \lVert\tilde{f}^{L}_{j}- f_{\tau_{D_n}(j)}\rVert_{\infty}$.

Finally, thanks to the truncation of the emission densities, it is possible to obtain the same rate in expectation:
\begin{equation*}
\begin{split}
\EE_{\theta}\left[\lVert \hat{f}_j - f_{\tau_{D_n}(j)} \rVert_{\infty}^{2}\right] &\leq c^{\prime\prime} D_n^{5/2} \left(\frac{\log(n)}{n}\right)^{\frac{2s}{2s+1}} + 2n^{\beta}\PP_{}((\mathcal{A}_n\cap E_n)^c)\\
&\leq c^{\prime\prime} D_n^{5/2}\left(\frac{\log(n)}{n}\right)^{\frac{2s}{2s+1}} + 2 n^{\beta}( n^{-\gamma} + 6n^{-\alpha}).
\end{split}
\end{equation*}
By choosing $\alpha$ and $\gamma$ sufficiently large, one obtains:
$$
\EE_{\theta}\left[\lVert \hat{f}_j - f_{\tau_{D_n}(j)} \rVert_{\infty}^{2}\right] = \mathcal{O}\left(D_n^{5/2}\left(\frac{\log(n)}{n}\right)^{\frac{2s}{2s+1}}\right).
$$

\paragraph*{Control of $\frac{1}{n}\sum_{i=1}^{n}\sum_{l=1}^{n}\EE_{\theta}\left[\left(\rho\vee\hat{\rho}\right)^{2|l-i|}\max_{x\in\mathbb{X}}\|f_{x} - \hat{f}_{x}\|_{\infty}^{2}\right]^{1/2}$}

Let $\tilde{\rho} = \frac{1 - 2\tilde{\delta}}{1 - \tilde{\delta}}$ where $\Tilde{\delta} = \frac{\delta}{2}$.
\begin{multline*}
    \EE_{\theta}\left[\left(\rho\vee\hat{\rho}\right)^{2|l-i|}\max_{x\in\mathbb{X}}\|f_{x} - \hat{f}_{x}\|_{\infty}^{2}\right]^{1/2}
    \leq \Tilde{\rho}^{|l-i|}\EE_{\theta}\left[\max_{x\in\mathbb{X}}\|f_{x} - \hat{f}_{x}\|_{\infty}^{2}\right]^{1/2}\\
    + \EE_{\theta}\left[\max_{x\in\mathbb{X}}\|f_{x} - \hat{f}_{x}\|_{\infty}^{2}\1_{\hat{\rho}\geq \Tilde{\rho}}\right]^{1/2}.
\end{multline*}
The term $\EE_{\theta}\left[\max_{x\in\mathbb{X}}\|f_{x} - \hat{f}_{x}\|_{\infty}^{2}\1_{\hat{\rho}\geq \Tilde{\rho}}\right]^{1/2}$ can be made of order $n^{-\alpha}$ for arbitrary $\alpha > 0$ by the same large deviation argument used in the control of $\EE_{\theta}\left[\frac{1}{\hat{\delta}^{2}}\right]$. Then, summing up over $i$ and $n$ yields:
$$
\frac{1}{n}\sum_{i=1}^{n}\sum_{l=1}^{n}\EE_{\theta}\left[\left(\rho\vee\hat{\rho}\right)^{2|l-i|}\max_{x\in\mathbb{X}}\|f_{x} - \hat{f}_{x}\|_{\infty}^{2}\right]^{1/2} = \mathcal{O}\left(D_{n}^{5/2}\left(\frac{\log(n)}{n}\right)^{\frac{s}{2s+1}}\right)
$$
Finally, using Proposition~\ref{prop:ub:excess:hmm} and the previous controls of the errors of estimation of the model parameters, one gets:
\begin{equation*}
    \EE[\ClassifRisk(\theta, \hat{h}^{\tau_n})] - \inf_{h\in \mathcal{H}_n}\ClassifRisk(\theta,h) = \mathcal{O}\left(D_{n}^{5/2}\left(\frac{\log(n)}{n}\right)^{\frac{s}{2s + 1}}\right).
\end{equation*}
A similar rate holds for the excess risk of clustering thanks to the relationship between the Bayes risk of classification and the Bayes risk of clustering established in \cite[Theorems~\ref*{main-thm:bayes-risk:lb:hmm:J=2} and~\ref*{main-thm:bayes-risk:lb:hmm:J>2}]{GKN2025main}.

\begin{lemma}\label{Lemma_4}
    Suppose $(A_t, t\in\mathbb{R})$ are $J\times J$ matrices simultaneously diagonalized by a matrix R:
    $$
R A_t R^{-1} = diag(\lambda_{t, 1}, ..., \lambda_{t, J}), t\in\mathbb{R}.
    $$
    Let $\hat{R}$ be a matrix such that for some permutation $\tau$ of $\left\{1, ..., J\right\}$, we have:
    $$
\lVert \hat{R} - P_{\tau}R\rVert = \varepsilon_{R} \leq \frac{\lVert R^{-1}\rVert^{-1}}{2}
    $$
    Assume $\lambda_{\max} = \sup_{t}\max_{j}|\lambda_{t, j}| < \infty$. For matrices $(\hat{A}_{t})_{t\in\mathbb{R}}$, write 
    $\varepsilon_{A} = \sup_{t}\lVert A_t -\hat{A}_t \rVert$ and define
    $$
\hat{\lambda}_{t, j} = e_{j}^{T}\hat{R}\hat{A}_{t}\hat{R}^{-1}e_j\quad \mathrm{and}\quad\lambda_{t, \tau(j)} = e_{j}^{T}(P_{\tau}R) A_{t}(P_{\tau}R)^{-1}e_j.
    $$
    Then
    $$
\sup_{t}\max_{j}|\hat{\lambda}_{t, j} - \lambda_{t, \tau(j)}| \leq 4\kappa(R)\left[\varepsilon_A + \lambda_{\max}\kappa(R)\lVert R^{-1}\rVert\varepsilon_{R}\right].
    $$
\end{lemma}
\begin{proof}
    Let $\hat{\zeta}_{j}^{T} = e_{j}^{T}\hat{R}$, let $\hat{\xi}_{j} = \hat{R}^{-1}e_{j}$ and define $\zeta_{j}^{T} = e_{j}^{T}P_{\tau}R$ and $\xi_{j} = (P_{\tau}R)^{-1}e_{j}$. \\
    Then, $\lambda_{t, \tau(j)} = \zeta^{T}_{j}A_t\xi_{j}$, $\hat{\lambda}_{t, j} = \hat{\zeta}^{T}_{j}\hat{A}_{t}\hat{\xi}_{j}$ and we have:
\begin{equation*}
    \begin{split}
    |\hat{\lambda}_{t, j} - \lambda_{t, \tau(j)}| &= |\hat{\zeta}^{T}_{j}\hat{A}_{t}\hat{\xi}_{j} - \zeta^{T}_{j}A_t\xi_{j}|\\
        &= |\hat{\zeta}_{j}^{T}\hat{A}_{t}(\hat{\xi}_j - \xi_j) + (\hat{\zeta}_j^{T} - \zeta_{j}^{T})A_t \xi_j + \hat{\zeta}_{j}^{T}(\hat{A}_t - A_t)\xi_j|\\
        &\leq \lVert \hat{\zeta}^{T}_{j}\rVert \lVert \hat{A}_t\rVert \lVert \hat{\xi}_j - \xi_j\rVert + \lVert \hat{\zeta}_{j}^{T} - \zeta_{j}^{T} \rVert\lVert A_t\xi_j\rVert + \lVert\hat{\zeta}_{j}^{T}\lVert\lVert \xi_j\rVert\varepsilon_{A}\\
        \lVert\zeta^{T}_j\lVert &= \lVert e_j^{T}P_{\tau}R\rVert \leq \lVert P_{\tau}R \rVert = \lVert R \rVert\\
        \lVert \hat{\zeta}_j^{T} - \zeta_j^{T}\lVert &= \lVert e_j^{T}(\hat{R} - P_{\tau}R)\rVert \leq \lVert\hat{R} - P_{\tau}R \rVert = \varepsilon_R\\
        \lVert\xi_j \rVert &= \lVert (P_{\tau}R)^{-1}e_j \rVert\leq \lVert (P_{\tau}R)^{-1}\rVert = \lVert R^{-1}\rVert\\
        \lVert \hat{\xi}_j - \xi_j \rVert &= \lVert(\hat{R}^{-1} - (P_{\tau}R)^{-1})e_j \rVert \leq \lVert \hat{R}^{-1} - (P_{\tau}R)^{-1}\rVert\\
        &\leq \frac{\lVert R^{-1}\rVert^{2}\varepsilon_R}{1 - \varepsilon_R\lVert R^{-1}\rVert} \text{ (by Lemma~37 of \cite{ACG21})}\\
        \lVert A_t\rVert &=\lVert R^{-1}diag(\lambda_{t, .})R\rVert \leq \lambda_{\max}\kappa(R)\\
        \lVert A_{t}\xi_j\rVert &= \lVert\lambda_{t, \tau(j)}\xi_j\rVert \leq \lambda_{\max} \lVert R^{-1}\rVert\\
        \lVert \hat{\zeta}^{T}_j\rVert &= \lVert \hat{\zeta}^{T}_j - \zeta^{T}_j + \zeta^{T}_j\rVert \leq \lVert \hat{\zeta}^{T}_j - \zeta^{T}_j \rVert + \lVert \zeta_j^{T} \rVert \leq \varepsilon_R + \lVert R\rVert\\
        \lVert\hat{A}_t\rVert &\leq  \varepsilon_A  + \lVert A_t\rVert\leq \varepsilon_A + \lambda_{\max}\kappa(R).
    \end{split}
\end{equation*}
Thus, by the inequalities above:
\begin{equation*}
    \begin{split}
        |\hat{\lambda}_{t, j} - \lambda_{t, \tau(j)}|
        &\leq (\varepsilon_R + \lVert R \rVert)(\varepsilon_A + \lambda_{\max}\kappa(R))\frac{\lVert R^{-1}\rVert^{2}\varepsilon_R}{1 - \varepsilon_R\lVert R^{-1}\rVert}\\
        &\quad+ \lambda_{\max}\varepsilon_R\lVert R^{-1}\rVert + \lVert R^{-1}\rVert(\varepsilon_R + \lVert R \rVert)\varepsilon_A\\
        &\leq 2(\varepsilon_R + \lVert R \rVert)(\varepsilon_A + \lambda_{\max}\kappa(R))\lVert R^{-1}\rVert^{2}\varepsilon_R\\
        &\quad+ \lambda_{\max}\lVert R^{-1}\rVert\varepsilon_R + \lVert R^{-1}\rVert(\varepsilon_R + \lVert R \rVert)\varepsilon_A\\
        &\leq 2\lVert R^{-1}\rVert^{2}\varepsilon_R^{2}\varepsilon_A + 2\kappa(R)\lVert R^{-1}\rVert\varepsilon_A\varepsilon_R + 2\lambda_{\max} \kappa(R)\lVert R^{-1}\rVert^{2}\varepsilon_R^{2}\\
        &\quad+2\lambda_{\max}\kappa(R)^{2}\lVert R^{-1}\rVert\varepsilon_R + \lambda_{\max}\lVert R^{-1}\rVert\varepsilon_R + \kappa(R)\varepsilon_A + \lVert R^{-1}\rVert\varepsilon_R\varepsilon_A\\
        &\leq 2\lVert R^{-1}\rVert^{2}\varepsilon_{R}^{2}\varepsilon_{A} + (2\kappa(R) + 1)\lVert R^{-1} \rVert\varepsilon_A\varepsilon_R\\
        &\quad+ 2\lambda_{\max}\kappa(R)\left((\lVert R^{-1}\rVert \varepsilon_R)^{2}
        + \kappa(R)\left(\lVert R^{-1}\rVert\varepsilon_R\right)\right) + \lambda_{\max}\lVert R^{-1}\rVert\varepsilon_{R} + \kappa(R)\varepsilon_A\\
        &\leq 2 \left(\frac{1}{2}\right)^{2}\varepsilon_{A} + \frac{3}{2}\kappa(R)\varepsilon_{A} + 2\lambda_{\max}\kappa(R)\lVert R^{-1}\rVert\varepsilon_{R}\left(\kappa(R) + \frac{1}{2}\right)\\
        &\quad+ \lambda_{\max}\lVert R^{-1}\rVert\varepsilon_{R} + \kappa(R)\varepsilon_{A}\\
        &\leq 3\kappa(R)\varepsilon_A + \lambda_{\max}(1 + 3 \kappa(R)^{2})\lVert R^{-1}\rVert\varepsilon_{R}\\
        &\leq 4\kappa(R)\left[\varepsilon_A + \lambda_{\max}\kappa(R)\lVert R^{-1}\rVert\varepsilon_{R}\right]
    \end{split}
\end{equation*}
where we have used $\lVert R^{-1}\rVert\varepsilon_{R}\leq 1/2$ and $\kappa(R) \geq 1$.
\end{proof}

\section{Proof of \texorpdfstring{\cite[Lemma~\ref*{main-lem:fastrates}]{GKN2025main}}{[4, Lemma 1]}}
\label{sec:proof-fastrate}

Let define $p_{n, i}(\theta):=\max_{k} \PP_\theta\left(X_i=k \mid Y_{1: n}\right)=\PP_\theta\left(X_i=h_{\theta, i}^{\star} \mid Y_{1: n}\right)$; defining $h_{\theta,i}^{\star}$ realizing the maximum. Suppose $p_{n, i}(\theta) \geq \frac{1}{2}+\gamma$ for some $0<\gamma \leq 1 / 2$. Then,
$$
\begin{aligned}
p_{n, i}(\theta)-\max _{k \neq h_{\theta, i}^{\star}} \mathbb{P}_\theta\left(X_i=k \mid Y_{1: n}\right) & \geq p_{n, i}(\theta)-\sum_{k \neq h_{\theta, i}^{\star}} \mathbb{P}_\theta\left(X_i=k \mid Y_{1: n}\right) \\
& =p_{n, i}(\theta)-\left[1-p_{n, i}(\theta)\right] \\
& =2 p_{n, i}(\theta)-1 \\
& \geq 2 \gamma .
\end{aligned}
$$
Consequently if $p_{n, i}(\theta) \geq \frac{1}{2}+\gamma$,
$$
\begin{aligned}
\mathbb{P}_{\hat{\theta}}\left(X_i=h_{\theta, i}^{\star} \mid Y_{1: n}\right) & \geq p_{n, i}(\theta)-\left\|\phi_{\theta, i \mid n}-\phi_{\hat{\theta}, i \mid n}\right\|_{\mathrm{TV}} \\
& \geq \max _{k \neq h_{\theta, i}^{\star}} \mathbb{P}_\theta\left(X_i=k \mid Y_{1: n}\right)+2 \gamma-\left\|\phi_{\theta, i \mid n}-\phi_{\hat{\theta}, i \mid n}\right\|_{\mathrm{TV}} \\
& \geq \max _{k \neq h_{\theta, i}^{\star}} \mathbb{P}_{\hat{\theta}}\left(X_i=k \mid Y_{1: n}\right)+2 \gamma-2\left\|\phi_{\theta, i \mid n}-\phi_{\hat{\theta}, i \mid n}\right\|_{\mathrm{TV}}
\end{aligned}
$$
We have shown that on the intersection of the two events
$$
E_{n, i}:=\left\{p_{n, i}(\theta) \geq \frac{1}{2}+\gamma\right\}, \quad F_{n, i}:=\left\{\left\|\phi_{\theta, i \mid n}-\phi_{\hat{\theta}, i \mid n}\right\|_{\mathrm{TV}}<\gamma\right\}
$$
the plug-in rule $h_{\hat{\theta}}^{\star}$ maximizing $k \mapsto \mathbb{P}_{\hat{\theta}}\left(X_i=k \mid Y_{1: n}\right)$ is unique and it must be that $h_{\hat{\theta}, i}^{\star}=h_{\theta, i}^{\star}$. Then, we bound the risk as follows,
$$
\begin{aligned}
\ClassifRisk(\theta, h_{\hat{\theta}}^{\star}) & \leq \EE_{\theta}\left[\frac{1}{n} \sum_{i=1}^n \1_{\left\{h_{\hat{\theta}, i}^{\star} \neq X_i\right\}} \1_{E_{n, i} \cap F_{n, i}}\right]+\EE_{\theta}\left[\frac{1}{n} \sum_{i=1}^n \1_{\left\{h_{\hat{\theta}, i}^{\star} \neq X_i\right\}} \1_{E_{n, i}^c}\right]\\
&\quad+\EE_{\theta}\left[\frac{1}{n} \sum_{i=1}^n \1_{\left\{h_{\hat{\theta}, i}^{\star} \neq X_i\right\}} \1_{F_{n, i}^c}\right] \\
& \leq \EE_{\theta}\left[\frac{1}{n} \sum_{i=1}^n \mathbb{P}_\theta\left(h_{\theta, i}^{\star} \neq X_i\right) \1_{E_{n, i}}\right]+\frac{1}{n} \sum_{i=1}^n \mathbb{P}_\theta\left(E_{n, i}^c\right)+\frac{1}{n} \sum_{i=1}^n \mathbb{P}_\theta\left(F_{n, i}^c\right)
\end{aligned}
$$
Finally, notice that,
$$
\begin{aligned}
\mathbb{P}_\theta\left(E_{n, i}^c\right) & =\mathbb{P}_\theta\left(p_{n, i}(\theta)<\frac{1}{2}+\gamma\right) \\
& =\mathbb{P}_\theta\left(\mathbb{P}_\theta\left(X_i=h_{\theta, i}^{\star} \mid Y_{1: n}\right)<\frac{1}{2}+\gamma\right) \\
& =\mathbb{P}_\theta\left(\mathbb{P}_\theta\left(X_i \neq h_{\theta, i}^{\star} \mid Y_{1: n}\right)>\frac{1}{2}-\gamma\right) \\
& \leq \frac{1}{1 / 2-\gamma} \EE_{\theta}\left(\mathbb{P}_\theta\left(h_{\theta, i}^{\star} \neq X_i \mid Y_{1: n}\right) \1_{E_{n, i}^c}\right).
\end{aligned}
$$
Hence the result. 

\section{Equivalence of the definitions of the risk of clustering}


\begin{lemma}\label{lem:matching}
     The risk of clustering of $\pi_n \circ h$ can be rewritten as
\begin{equation}
  \ClusterRisk(\theta,\pi_n \circ h)%
  \coloneqq \EE_{\theta}\Bigg[\min_{\tau \in \mathcal{S}_{J}}\frac{1}{n}\sum_{i=1}^n\1_{h_i(Y_{1:n}) \ne \tau(X_i)} \Bigg]
\end{equation}
\end{lemma}
\begin{proof}
    It suffices to show that 
    $$
    \sup_{\substack{M \subseteq \mathcal{E}(\pi_n \circ h(Y_{1:n}),\Pi_n)\\M\, \textrm{is\, a matching} } 
  } \sum_{ \{C,C' \} \in M }\card{(C \cap C' )} = \sup_{\tau \in \mathcal{S}_{J}}\frac{1}{n}\sum_{i=1}^n\1_{h_i(Y_{1:n}) = \tau(X_i)}
  $$
Let $C_k = \left\{i\in[n]\mid h_{i}(Y_{1:n}) = k\right\}$ and $C_k ^{\prime} = \left\{i\in[n]\mid X_i = k\right\}$. Since the two partitions $\Pi_n$ and $\pi_n\circ h(Y_{1:n})$ have the same number of clusters $J$ (with possibly empty clusters), the supremum is reached on matchings with $J$ edges. Using this fact, it follow that the matching reaching the supremum is of the form:
$$
M = \left\{(C_k, C_{\tau(k)}')\mid 1\leq k \leq J\right\}
$$
where $\tau$ is a permutation of $\left\{1, .., J\right\}$. One obtains:
\begin{equation}
    \begin{split}
        \sup_{\substack{M \subseteq \mathcal{E}(\pi_n \circ h(Y_{1:n}),\Pi_n)\\M\, \textrm{is\, a matching}}}
        \sum_{ \{C,C' \} \in M }\card{(C \cap C' )} &= \sup_{\tau\in\mathcal{S}_{J}}\sum_{k=1}^{J}\card{(C_k \cap C_{\tau(k)}')}\\
        &= \sup_{\tau\in\mathcal{S}_{J}}\sum_{k=1}^{J}\sum_{i=1}^{n}\1_{\tau^{-1}(X_i) = h_{i}(Y_{1:n}) = k }\\
        &= \sup_{\tau\in\mathcal{S}_{J}}\sum_{i=1}^{n}\1_{\tau(X_i) = h_{i}(Y_{1:n})}
    \end{split}
\end{equation}
\end{proof}

\section{Proof of \texorpdfstring{Lemma~\ref{lem:max_dev:Scully}}{Lemma S1.1}}
\label{sec:proof:max_dev}
Without loss of generality, let $\alpha_i \geq \frac{1}{2}$ for all $i$. When $p_i = \alpha_i$, we will say that $i$ is given positive bias, and similarly when $p_i = 1 - \alpha_i$ we say it has negative bias. We show that unless we give all positive bias or all negative bias, we can flip the bias of some $i$ to increase the expectation. This suffices to conclude. Let $\left(Z_i\right)_{i\in[n]}$  be a sequence of independent Bernoulli random variables such that $Z_n\sim \mathcal{B}(\alpha_i)$, and let $\beta_i \in \{-1,+1\}$ be the bias we give $i$. We then let
\[
Y_i = 2Z_i - 1 \in \{-1,+1\}
 \text{ and }
X_i = \frac{1 + \beta_i Y_i}{2} \in \{0,1\}.
\]
Consequently, $X_i \sim \mathcal{B}\left(\alpha_i \1_{\beta_i = 1} + (1-\alpha_i)\1_{\beta_i = -1}\right)$ and $\sum_{i=1}^{n} X_i - \frac{n}{2} = \frac{1}{2} \sum_{i=1}^{n} \beta_i Y_i$.
Letting $S_n = \sum_{i=1}^{n} \beta_i Y_i$, we intend to choose the $\beta_i$ to maximize $\EE[|S_n|]$.\\
Let $S_{\neq k} = \sum_{i \neq k} \beta_i Y_i$ and define:
\[
\text{sign}(x) = \left\{
    \begin{array}{ll}
        1 & \mbox{if } x < 0 \\
        0 & \mbox{if } x =0 \\
       -1 & \mbox{if } x > 0
    \end{array}
\right.
\]
Then,
\[
\begin{split}
    \left|S_n\right| &= S_n \text{sign}(S_n)\\
    &= S_n \text{sign}\left(S_{\neq k}\right) + S_n \left(\text{sign}(S_n) - \text{sign}(S_{\neq k})\right)\\
    &= S_{\neq k}\text{sign}\left(S_{\neq k}\right) + \beta_k Y_k \text{sign}\left(S_{\neq k}\right) + S_{n}\left(\text{sign}\left(S_n\right) - \text{sign}\left(S_{\neq k}\right)\right)\\
    &= \left|S_{\neq k}\right| + \beta_k Y_k \text{sign}\left(S_{\neq k}\right) + \1_{S_{\neq k} = 0}
\end{split}
\]
By the computation above and the fact that $S_{\neq k}$ and $Y_k$ are independent:
\[
\EE[|S_n|] = \EE[|S_{\neq k}|] + \PP(S_{\neq k} = 0) + \EE[\text{sign}(S_{\neq k})]\beta_k \EE[Y_k].
\]
Since $\EE[Y_k] \geq 0$, we conclude that if we fix the values of $\left(\beta_i\right)_{i\neq k}$, then the value of $\beta_k$ that maximizes $\EE[|S_n|]$ is:
\begin{equation}\label{eq:optim:condition}
    \beta_k = \text{sign}\left(\EE[\text{sign}(S_{\neq k})]\right).
\end{equation}

Assume for the moment that all the $\alpha_i$ are distinct and obey $\alpha_i > \frac{1}{2}$.
This assumption will guarantee that for non-same-sign biases, there is always at least one bias we can flip to strictly increase $\EE[|S_n|]$. We will remove this assumption at the end of the proof. Consider any assignment of the biases $\beta_1, \dots, \beta_n$.
\begin{lemma}\label{lem:fact:1}
     If all the $\alpha_i$ are distinct and ensure $\alpha_i > \frac{1}{2}$, the values $\left(\EE[\text{sign}(S_{\neq j})]\right)_{j\in[n]}$ are distinct, so there exists $k$ such that $\EE[\text{sign}(S_{\neq k})] \neq 0$.
\end{lemma}

\begin{lemma}\label{lem:fact:2}
    Suppose $\beta_j = 1$ and $\beta_k = -1$. Then $\EE[\text{sign}(S_{\neq k})] \geq \EE[\text{sign}(S_{\neq j})]$.
\end{lemma}

Both facts will be proved at the end of this section. With these facts in hand, consider any non-same-sign biases $\beta_1, \dots, \beta_n$. By Lemma~\ref{lem:fact:1}, there exists $k$ such that:
\[
\EE[\text{sign}(S_{\neq k})] \neq 0.
\]
Without loss of generality, suppose $\beta_k = -1$. There are two cases to consider.

\begin{itemize}
    \item If $\EE[\text{sign}(S_{\neq k})] > 0$, then $\beta_k$ disobeys the condition \eqref{eq:optim:condition}, so swapping to $\beta_k = 1$ increases $\EE_{\theta}[|S_n|]$.
    \item If $\EE[\text{sign}(S_{\neq k})] < 0$, then $\beta_k$ obeys the condition \eqref{eq:optim:condition}, so we need to find another bias to swap. Because the assignment is non-same-sign, there exists $j$ such that $\beta_j = 1$. And by Lemma~\ref{lem:fact:2} , $\EE[\text{sign}(S_{\neq j})] < 0$. This means $\beta_j$ disobeys the condition \ref{eq:optim:condition}, so swapping to $\beta_j = -1$ increases $\EE[|S_n|]$.
\end{itemize}

We have shown that any non-same-sign bias assignment is suboptimal for maximizing $\EE[|S_n|]$, so only the same-sign cases can be optimal. And it is clear by symmetry that they both yield the same $\EE[|S_n|]$ value, so both are optimal. Note that $\EE[|S_n|]$ is a polynomial in the parameters $\left(\alpha_i\right)_{i\in[n]}$. When the assumption that all $\alpha_i$ are distinct and obey $\alpha_i > \frac{1}{2}$ does not hold, the result is still true thanks to continuity with respect to $\left(\alpha_i\right)_{i\in[n]}$.

\subsection{Proof of \texorpdfstring{Lemma~\ref{lem:fact:1}}{Lemma S14.14}}
Let $S_{\notin \{j,k\}} = \sum_{i \notin \{j,k\}} \beta_i Y_i$ and $j\neq k\in[n]$.
\begin{equation*}
    \begin{split}
        &\EE\left[\text{sign}(S_{\neq j)}\right] - \EE\left[\text{sign}(S_{\neq k})\right] = \PP\left(S_{\neq j} > 0\right) - \PP\left(S_{\neq j} < 0\right) - \PP\left(S_{\neq k} > 0\right) + \PP\left(S_{\neq k} < 0\right)\\
        &= \PP\left(S_{\notin \{j,k\}} + \beta_k Y_k > 0\right) - \PP\left(S_{\notin \{j,k\}} + \beta_k Y_k < 0\right) - \PP\left(S_{\notin \{j,k\}} + \beta_j Y_j > 0\right) + \PP\left(S_{\notin \{j,k\}} + \beta_j Y_j < 0\right)\\
        &= \alpha_k\PP\left(S_{\notin \{j,k\}} > - \beta_k\right) + (1 - \alpha_k)\PP\left(S_{\notin \{j,k\}} > \beta_k\right) - \alpha_k\PP\left(S_{\notin \{j,k\}} < - \beta_k\right) - (1 - \alpha_k)\PP\left(S_{\notin \{j,k\}} < \beta_k\right)\\
        & -\alpha_j\PP\left(S_{\notin \{j,k\}} > - \beta_j\right) - (1 - \alpha_j)\PP\left(S_{\notin \{j,k\}} > \beta_j\right) + \alpha_j\PP\left(S_{\notin \{j,k\}} < - \beta_j\right) + (1 - \alpha_j)\PP\left(S_{\notin \{j,k\}} < \beta_j\right)
    \end{split}
\end{equation*}
When $\beta_j = \beta_k = 1$,
\begin{equation*}
    \EE\left[\text{sign}(S_{\neq j)}\right] - \EE\left[\text{sign}(S_{\neq k})\right] = (\alpha_k - \alpha_j) \left(\PP\left(S_{\notin \{j,k\}}\in \left\{0, 1\right\}\right) + \PP\left(S_{\notin \{j,k\}}\in \left\{0, -1\right\}\right)\right) \neq 0.
\end{equation*}
The remaining cases can be analyzed similarly.

\subsection{Proof of \texorpdfstring{Lemma~\ref{lem:fact:2}}{Lemma S14.15}}
We want to show that:
\begin{equation*}
    \EE\left[\text{sign}(S_{\notin \{j,k\}} + Y_j)\right] \geq \EE\left[\text{sign}(S_{\notin \{j,k\}} - Y_k)\right]
\end{equation*}
The key observation is that we have a stochastic dominance $+Y_j\geq_{st} -Y_k$. This means that for any nondecreasing function $f$:
\begin{equation*}
    \EE\left[f(Y_j)\right] \geq \EE\left[f(-Y_k)\right].
\end{equation*}
The desired result follows by applying the above to
\begin{equation*}
    f(x) = \EE\left[\text{sign}(S_{\notin \{j,k\}} + x)\right]
\end{equation*}
which is clearly nondecreasing.

%% file: __arxiv.bbl
\begin{thebibliography}{31}

\bibitem{ACG21}
\begin{barticle}[author]
\bauthor{\bsnm{Abraham},~\bfnm{Kweku}\binits{K.}},
  \bauthor{\bsnm{Castillo},~\bfnm{Isma{\"e}l}\binits{I.}} \AND
  \bauthor{\bsnm{Gassiat},~\bfnm{Elisabeth}\binits{E.}}
(\byear{2022}).
\btitle{Multiple testing in nonparametric hidden Markov models: An empirical
  Bayes approach}.
\bjournal{J. Mach. Learn. Res.}
\bvolume{23}
\bpages{4061--4117}.
\end{barticle}
\endbibitem

\bibitem{AGN23}
\begin{barticle}[author]
\bauthor{\bsnm{Abraham},~\bfnm{Kweku}\binits{K.}},
  \bauthor{\bsnm{Gassiat},~\bfnm{Elisabeth}\binits{E.}} \AND
  \bauthor{\bsnm{Naulet},~\bfnm{Zacharie}\binits{Z.}}
(\byear{2023}).
\btitle{Frontiers to the learning of nonparametric hidden Markov models}.
\bjournal{arXiv preprint}.
\bdoi{10.48550/arXiv.2306.16293}
\end{barticle}
\endbibitem

\bibitem{AGN21}
\begin{barticle}[author]
\bauthor{\bsnm{Abraham},~\bfnm{Kweku}\binits{K.}},
  \bauthor{\bsnm{Gassiat},~\bfnm{Elisabeth}\binits{E.}} \AND
  \bauthor{\bsnm{Naulet},~\bfnm{Zacharie}\binits{Z.}}
(\byear{2023}).
\btitle{Fundamental limits for learning hidden {M}arkov model parameters}.
\bjournal{IEEE Trans. Inform. Theory}
\bvolume{69}
\bpages{1777--1794}.
\bdoi{10.1109/tit.2022.3213429}
\bmrnumber{4564681}
\end{barticle}
\endbibitem

\bibitem{GHA16}
\begin{barticle}[author]
\bauthor{\bsnm{Alexandrovich},~\bfnm{G.}\binits{G.}},
  \bauthor{\bsnm{Holzmann},~\bfnm{H.}\binits{H.}} \AND
  \bauthor{\bsnm{Leister},~\bfnm{A.}\binits{A.}}
(\byear{2016}).
\btitle{{Nonparametric identification and maximum likelihood estimation for
  hidden Markov models}}.
\bjournal{Biometrika}
\bvolume{103}
\bpages{423-434}.
\bdoi{10.1093/biomet/asw001}
\end{barticle}
\endbibitem

\bibitem{MR3509896}
\begin{barticle}[author]
\bauthor{\bsnm{Alexandrovich},~\bfnm{G.}\binits{G.}},
  \bauthor{\bsnm{Holzmann},~\bfnm{H.}\binits{H.}} \AND
  \bauthor{\bsnm{Leister},~\bfnm{A.}\binits{A.}}
(\byear{2016}).
\btitle{Nonparametric identification and maximum likelihood estimation for
  hidden {M}arkov models}.
\bjournal{Biometrika}
\bvolume{103}
\bpages{423--434}.
\bdoi{10.1093/biomet/asw001}
\bmrnumber{3509896}
\end{barticle}
\endbibitem

\bibitem{anandkumar2014tensor}
\begin{barticle}[author]
\bauthor{\bsnm{Anandkumar},~\bfnm{Animashree}\binits{A.}},
  \bauthor{\bsnm{Ge},~\bfnm{Rong}\binits{R.}},
  \bauthor{\bsnm{Hsu},~\bfnm{Daniel~J}\binits{D.~J.}},
  \bauthor{\bsnm{Kakade},~\bfnm{Sham~M}\binits{S.~M.}},
  \bauthor{\bsnm{Telgarsky},~\bfnm{Matus}\binits{M.}} \betal{et~al.}
(\byear{2014}).
\btitle{Tensor decompositions for learning latent variable models.}
\bjournal{J. Mach. Learn. Res.}
\bvolume{15}
\bpages{2773--2832}.
\end{barticle}
\endbibitem

\bibitem{CMT05}
\begin{bbook}[author]
\bauthor{\bsnm{Capp\'{e}},~\bfnm{Olivier}\binits{O.}},
  \bauthor{\bsnm{Moulines},~\bfnm{Eric}\binits{E.}} \AND
  \bauthor{\bsnm{Ryden},~\bfnm{Tobias}\binits{T.}}
(\byear{2005}).
\btitle{Inference in Hidden Markov Models (Springer Series in Statistics)}.
\bpublisher{Springer-Verlag}.
\end{bbook}
\endbibitem

\bibitem{YEC16}
\begin{barticle}[author]
\bauthor{\bsnm{Castro},~\bfnm{Yohann~De}\binits{Y.~D.}},
  \bauthor{\bsnm{Gassiat},~\bfnm{{\'E}lisabeth}\binits{{\'E}.}} \AND
  \bauthor{\bsnm{Lacour},~\bfnm{Claire}\binits{C.}}
(\byear{2016}).
\btitle{Minimax Adaptive Estimation of Nonparametric Hidden Markov Models}.
\bjournal{J. Mach. Learn. Res.}
\bvolume{17}
\bpages{1-43}.
\end{barticle}
\endbibitem

\bibitem{YohannGH}
\begin{bmisc}[author]
\bauthor{\bsnm{De~Castro},~\bfnm{Yohann}\binits{Y.}}
(\byear{2016}).
\btitle{Nonparametric HMM}.
\bhowpublished{\url{https://github.com/ydecastro/nonparametric-hmm}}.
\end{bmisc}
\endbibitem

\bibitem{YES17}
\begin{barticle}[author]
\bauthor{\bsnm{De~Castro},~\bfnm{Yohann}\binits{Y.}},
  \bauthor{\bsnm{Gassiat},~\bfnm{{\'E}lisabeth}\binits{{\'E}.}} \AND
  \bauthor{\bsnm{Le~Corff},~\bfnm{Sylvain}\binits{S.}}
(\byear{2017}).
\btitle{Consistent Estimation of the Filtering and Marginal Smoothing
  Distributions in Nonparametric Hidden Markov Models}.
\bjournal{IEEE Trans. Inform. Theory}
\bvolume{63}
\bpages{4758-4777}.
\bdoi{10.1109/TIT.2017.2696959}
\end{barticle}
\endbibitem

\bibitem{DGL96}
\begin{bbook}[author]
\bauthor{\bsnm{Devroye},~\bfnm{Luc}\binits{L.}},
  \bauthor{\bsnm{Györfi},~\bfnm{László}\binits{L.}} \AND
  \bauthor{\bsnm{Lugosi},~\bfnm{Gábor}\binits{G.}}
(\byear{1996}).
\btitle{A Probablistic Theory of Pattern Recognition}
\bvolume{31}.
\bdoi{10.1007/978-1-4612-0711-5}
\end{bbook}
\endbibitem

\bibitem{MR3845014}
\begin{barticle}[author]
\bauthor{\bsnm{Gao},~\bfnm{Chao}\binits{C.}},
  \bauthor{\bsnm{Ma},~\bfnm{Zongming}\binits{Z.}},
  \bauthor{\bsnm{Zhang},~\bfnm{Anderson~Y.}\binits{A.~Y.}} \AND
  \bauthor{\bsnm{Zhou},~\bfnm{Harrison~H.}\binits{H.~H.}}
(\byear{2018}).
\btitle{Community detection in degree-corrected block models}.
\bjournal{Ann. Statist.}
\bvolume{46}
\bpages{2153--2185}.
\bdoi{10.1214/17-AOS1615}
\bmrnumber{3845014}
\end{barticle}
\endbibitem

\bibitem{GKN2025SM}
\begin{barticle}[author]
\bauthor{\bsnm{Gassiat},~\bfnm{\'{E}lisabeth}\binits{E.}},
  \bauthor{\bsnm{Kaddouri},~\bfnm{Ibrahim}\binits{I.}} \AND
  \bauthor{\bsnm{Naulet},~\bfnm{Zacharie}\binits{Z.}}
(\byear{2025}).
\btitle{Clustering risk in non-parametric Hidden Markov and i.i.d. models:
  supplementary material}.
\bjournal{Ann. Statist.}
\end{barticle}
\endbibitem

\bibitem{GKN2025main}
\begin{barticle}[author]
\bauthor{\bsnm{Gassiat},~\bfnm{\'{E}lisabeth}\binits{E.}},
  \bauthor{\bsnm{Kaddouri},~\bfnm{Ibrahim}\binits{I.}} \AND
  \bauthor{\bsnm{Naulet},~\bfnm{Zacharie}\binits{Z.}}
(\byear{2025}).
\btitle{Clustering risk in non-parametric Hidden Markov and i.i.d. models}.
\bjournal{Ann. Statist.}
\end{barticle}
\endbibitem

\bibitem{GGM14}
\begin{barticle}[author]
\bauthor{\bsnm{Ghassempour},~\bfnm{Shima}\binits{S.}},
  \bauthor{\bsnm{Girosi},~\bfnm{Federico}\binits{F.}} \AND
  \bauthor{\bsnm{Maeder},~\bfnm{Anthony}\binits{A.}}
(\byear{2014}).
\btitle{Clustering Multivariate Time Series Using Hidden Markov Models}.
\bjournal{International Journal of Environmental Research and Public Health}
\bvolume{3}.
\bdoi{10.3390/ijerph110302741}
\end{barticle}
\endbibitem

\bibitem{G21}
\begin{bbook}[author]
\bauthor{\bsnm{Giraud},~\bfnm{Christophe}\binits{C.}}
(\byear{2021}).
\btitle{Introduction to high-dimensional statistics},
\bedition{Second} ed.
\bpublisher{Chapman and Hall/CRC}.
\end{bbook}
\endbibitem

\bibitem{GV19}
\begin{barticle}[author]
\bauthor{\bsnm{Giraud},~\bfnm{Christophe}\binits{C.}} \AND
  \bauthor{\bsnm{Verzelen},~\bfnm{Nicolas}\binits{N.}}
(\byear{2018}).
\btitle{Partial recovery bounds for clustering with the relaxed K-means}.
\bjournal{Math. Stat. Learn.}
\bvolume{3}
\bpages{317–374}.
\end{barticle}
\endbibitem

\bibitem{MR3889693}
\begin{bincollection}[author]
\bauthor{\bsnm{Gr\"{u}n},~\bfnm{Bettina}\binits{B.}}
(\byear{2019}).
\btitle{Model-based clustering}.
In \bbooktitle{Handbook of mixture analysis}.
\bseries{Chapman \& Hall/CRC Handb. Mod. Stat. Methods}
\bpages{157--192}.
\bpublisher{CRC Press, Boca Raton, FL}.
\bmrnumber{3889693}
\end{bincollection}
\endbibitem

\bibitem{KSH05}
\begin{binproceedings}[author]
\bauthor{\bsnm{Kannan},~\bfnm{Ravindran}\binits{R.}},
  \bauthor{\bsnm{Salmasian},~\bfnm{Hadi}\binits{H.}} \AND
  \bauthor{\bsnm{Vempala},~\bfnm{Santosh}\binits{S.}}
(\byear{2005}).
\btitle{The spectral method for general mixture models}.
In \bbooktitle{International conference on computational learning theory}
\bpages{444--457}.
\bpublisher{Springer}.
\end{binproceedings}
\endbibitem

\bibitem{KG17}
\begin{binproceedings}[author]
\bauthor{\bsnm{Khiatani},~\bfnm{Diksha}\binits{D.}} \AND
  \bauthor{\bsnm{Ghose},~\bfnm{Udayan}\binits{U.}}
(\byear{2017}).
\btitle{Weather forecasting using Hidden Markov Model}.
In \bbooktitle{2017 International Conference on Computing and Communication
  Technologies for Smart Nation (IC3TSN)}
\bpages{220-225}.
\bdoi{10.1109/IC3TSN.2017.8284480}
\end{binproceedings}
\endbibitem

\bibitem{MR3862446}
\begin{barticle}[author]
\bauthor{\bsnm{Leh\'{e}ricy},~\bfnm{Luc}\binits{L.}}
(\byear{2018}).
\btitle{State-by-state minimax adaptive estimation for nonparametric hidden
  {M}arkov models}.
\bjournal{J. Mach. Learn. Res.}
\bvolume{19}
\bpages{Paper No. 39, 46}.
\bmrnumber{3862446}
\end{barticle}
\endbibitem

\bibitem{MR4347381}
\begin{barticle}[author]
\bauthor{\bsnm{Leh\'{e}ricy},~\bfnm{Luc}\binits{L.}}
(\byear{2021}).
\btitle{Nonasymptotic control of the {MLE} for misspecified nonparametric
  hidden {M}arkov models}.
\bjournal{Electron. J. Stat.}
\bvolume{15}
\bpages{4916--4965}.
\bdoi{10.1214/21-ejs1890}
\bmrnumber{4347381}
\end{barticle}
\endbibitem

\bibitem{LH16}
\begin{barticle}[author]
\bauthor{\bsnm{Lu},~\bfnm{Yu}\binits{Y.}} \AND
  \bauthor{\bsnm{H.~Zhou},~\bfnm{Harrison}\binits{H.}}
(\byear{2016}).
\btitle{Statistical and Computational Guarantees of Lloyd’s Algorithm and Its
  Variants}.
\bjournal{arXiv preprint}.
\bdoi{10.48550/arXiv.1612.02099}
\end{barticle}
\endbibitem

\bibitem{MRR23}
\begin{barticle}[author]
\bauthor{\bsnm{Marandon},~\bfnm{Ariane}\binits{A.}},
  \bauthor{\bsnm{Rebafka},~\bfnm{Tabea}\binits{T.}} \AND
  \bauthor{\bsnm{Roquain},~\bfnm{Etienne}\binits{E.}}
(\byear{2023}).
\btitle{False clustering rate control in mixture models}.
\bjournal{arXiv preprint}.
\bdoi{10.48550/arXiv.2203.02597}
\end{barticle}
\endbibitem

\bibitem{meila2005comparing}
\begin{binproceedings}[author]
\bauthor{\bsnm{Meil{\u{a}}},~\bfnm{Marina}\binits{M.}}
(\byear{2005}).
\btitle{Comparing clusterings: an axiomatic view}.
In \bbooktitle{Proceedings of the 22nd international conference on Machine
  learning}
\bpages{577--584}.
\end{binproceedings}
\endbibitem

\bibitem{meila2001experimental}
\begin{barticle}[author]
\bauthor{\bsnm{Meil{\u{a}}},~\bfnm{Marina}\binits{M.}} \AND
  \bauthor{\bsnm{Heckerman},~\bfnm{David}\binits{D.}}
(\byear{2001}).
\btitle{An experimental comparison of model-based clustering methods}.
\bjournal{Machine learning}
\bvolume{42}
\bpages{9--29}.
\end{barticle}
\endbibitem

\bibitem{MVW17}
\begin{barticle}[author]
\bauthor{\bsnm{Mixon},~\bfnm{Dustin~G}\binits{D.~G.}},
  \bauthor{\bsnm{Villar},~\bfnm{Soledad}\binits{S.}} \AND
  \bauthor{\bsnm{Ward},~\bfnm{Rachel}\binits{R.}}
(\byear{2017}).
\btitle{Clustering subgaussian mixtures by semidefinite programming}.
\bjournal{Information and Inference: A Journal of the IMA}
\bvolume{6}
\bpages{389--415}.
\end{barticle}
\endbibitem

\bibitem{Ndaoud18}
\begin{barticle}[author]
\bauthor{\bsnm{Ndaoud},~\bfnm{Mohamed}\binits{M.}}
(\byear{2022}).
\btitle{{Sharp optimal recovery in the two component Gaussian mixture model}}.
\bjournal{Ann. Statist.}
\bvolume{50}
\bpages{2096 -- 2126}.
\bdoi{10.1214/22-AOS2178}
\end{barticle}
\endbibitem

\bibitem{SSS03}
\begin{barticle}[author]
\bauthor{\bsnm{Schliep},~\bfnm{Alexander}\binits{A.}},
  \bauthor{\bsnm{Schönhuth},~\bfnm{Alexander}\binits{A.}} \AND
  \bauthor{\bsnm{Steinhoff},~\bfnm{Christine}\binits{C.}}
(\byear{2003}).
\btitle{Using hidden Markov models to analyze gene expression time course
  data}.
\bjournal{Bioinformatics}.
\bdoi{10.1093/bioinformatics/btg1036}
\end{barticle}
\endbibitem

\bibitem{SBY09}
\begin{barticle}[author]
\bauthor{\bsnm{Shi},~\bfnm{Tao}\binits{T.}},
  \bauthor{\bsnm{Belkin},~\bfnm{Mikhail}\binits{M.}} \AND
  \bauthor{\bsnm{Yu},~\bfnm{Bin}\binits{B.}}
(\byear{2009}).
\btitle{Data spectroscopy: Eigenspaces of convolution operators and
  clustering}.
\bjournal{The Annals of Statistics}
\bpages{3960--3984}.
\end{barticle}
\endbibitem

\bibitem{MR3546450}
\begin{barticle}[author]
\bauthor{\bsnm{Zhang},~\bfnm{Anderson~Y.}\binits{A.~Y.}} \AND
  \bauthor{\bsnm{Zhou},~\bfnm{Harrison~H.}\binits{H.~H.}}
(\byear{2016}).
\btitle{Minimax rates of community detection in stochastic block models}.
\bjournal{Ann. Statist.}
\bvolume{44}
\bpages{2252--2280}.
\bdoi{10.1214/15-AOS1428}
\bmrnumber{3546450}
\end{barticle}
\endbibitem

\end{thebibliography}
